\documentclass{article}
\usepackage{comment}
\usepackage{parskip}
\usepackage{mathtools}
\usepackage[utf8]{inputenc}
\usepackage{amsmath}
\usepackage{amssymb}
\usepackage{amsfonts}
\usepackage{graphicx}
\usepackage{tikz-cd}
\usepackage{array}
\usepackage{amsthm}
\usepackage{hyperref}
\usepackage[capitalise,nameinlink]{cleveref}
\usepackage{array}
\usepackage{makecell}
\usepackage{calligra}
\usepackage{setspace}
\usepackage{enumitem}
\usepackage{colonequals}
\usepackage{fullpage}
\usepackage{float}	
\usepackage{lscape}
\usepackage{longtable}
\usepackage{ltcaption}
\usepackage{geometry}
\usepackage{authblk}
\usepackage{bm}
\usepackage{multirow}
\usepackage[only,llbracket,rrbracket]{stmaryrd}
\usepackage{resizegather}
\usepackage{fancyhdr}
\newtheorem{theorem}{Theorem}[subsection]
\newtheorem{proposition}[theorem]{Proposition}
\newtheorem{lemma}[theorem]{Lemma}
\newtheorem{corollary}[theorem]{Corollary}

\newtheorem{properties}[theorem]{Properties}

\theoremstyle{remark}
\newtheorem{remark}[theorem]{Remark}
\theoremstyle{definition}
\newtheorem{definition}[theorem]{Definition}
\theoremstyle{remark}
\newtheorem{example}[theorem]{Example}
\theoremstyle{definition}

\crefname{theorem}{Theorem}{Theorems}
\crefname{definition}{Definition}{Definitions}
\crefname{lemma}{Lemma}{Lemmas}
\crefname{proposition}{Proposition}{Propositions}
\crefname{corollary}{Corollary}{Corollaries}
\crefname{equation}{Equation}{Equations}
\crefname{condition}{Condition}{Conditions}
\setcellgapes{10pt}
\newcommand{\numberthis}{\refstepcounter{equation}\tag{\arabic{equation}}}
\expandafter\appto\csname align*\endcsname{\numberthis}
\allowdisplaybreaks
\newcommand{\Z}{\mathbb{Z}}
\newcommand{\F}{\mathbb{F}}
\newcommand{\Q}{\mathbb{Q}}
\newcommand{\m}{\mathfrak{m}}
  
\renewcommand{\subset}{\subseteq}

\DeclareMathOperator{\Spf}{Spf}
\DeclareMathOperator{\Spec}{Spec}

\DeclareMathOperator{\End}{End}
\DeclareMathOperator{\cont}{cont}
\DeclareMathOperator{\soc}{soc}
\DeclareMathOperator{\cosoc}{cosoc}
\DeclareMathOperator{\Ext}{Ext}
\DeclareMathOperator{\Hom}{Hom}
\DeclareMathOperator{\Inj}{Inj}
\DeclareMathOperator{\Proj}{Proj}
\DeclareMathOperator{\JH}{JH}
\DeclareMathOperator{\GL}{GL}
\DeclareMathOperator{\SL}{SL}

\DeclareMathOperator{\Res}{Res}
\DeclareMathOperator{\res}{res}
\DeclareMathOperator{\Rep}{Rep}
\DeclareMathOperator{\Rad}{Rad}
\DeclareMathOperator{\sgn}{sgn}

\DeclareMathOperator{\Coker}{Coker}
\DeclareMathOperator{\Sym}{Sym}
\DeclareMathOperator{\Adm}{Adm}

\DeclareMathOperator{\Disc}{Disc}
\DeclareMathOperator{\HT}{HT}

\DeclareMathOperator{\Mod}{Mod}
\DeclareMathOperator{\reg}{reg}
\DeclareMathOperator{\Mat}{Mat}
\DeclareMathOperator{\inv}{inv}
\DeclareMathOperator{\triv}{triv}
\DeclareMathOperator{\adj}{adj}

\DeclareMathOperator{\Supp}{Supp}
\DeclareMathOperator{\Frob}{Frob}
\DeclareMathOperator{\Gal}{Gal}
\DeclareMathOperator{\poly}{poly}

\DeclareMathOperator{\Id}{Id}
\DeclareMathOperator{\et}{\textrm{\small \'et}}
\title{On Breuil's Lattice Conjecture for $\GL_2$}
\author{Hymn Chan\footnote{Department of Mathematics, University of Toronto, 40 St. George St., BA6290, Toronto, ON M5S 2E4, Canada}}
\date{\today}
\begin{document}
\maketitle
\begin{abstract}
    We prove Breuil's lattice conjecture for higher Hodge--Tate weights in the case of $\GL_2(K)$ where $K$ is an unramified extension of $\Q_p$. More precisely, under some genericity conditions, we show that the lattice inside a locally algebraic type induced by the completed cohomology of a $U(2)$-arithmetic manifold depends only on the Galois representation at the fixed place above $p$ for arbitrary Hodge--Tate weights, which are small relative to $p$. We further prove that the patched modules of all lattices inside the locally algebraic types with irreducible cosocle are cyclic.\par
    One key input of the paper is a structure theorem for mod $p$ representations of $\GL_2(\mathcal{O}_K)$, which are residually multiplicity free and of finite length. Another input is an explicit computation of universal framed Galois deformation rings, which parameterize potentially crystalline lifts with fixed tame inertial types and higher Hodge--Tate weights. 
\end{abstract}
\tableofcontents
\section{Introduction}
\subsection{Breuil's lattice conjecture} The motivation for this paper is the $p$-adic Langlands Program. Let $E$ be a finite extension of $\Q_p$, which is sufficiently large, with ring of integer $\mathcal{O}$, uniformizer $\varpi$ and residue field $\F$. Roughly speaking, the $p$-adic Langlands Program is looking for a correspondence between continuous representations $\rho:G_K\to \GL_n(E)$ and certain unitary Banach space representations of $\GL_n(K)$ over $E$. Taking modulo $p$ on both sides, we expect a mod $p$ Langlands correspondence. When $n=2$ and $K=\Q_p$, we have a precise $p$-adic and mod $p$ Langlands correspondence due to the culmination of works by Breuil, Colmez, Pa{\v s}k\=unas and many others \cite{modplanglands}, \cite{padiclanglands} and \cite{padiclanglandspaskunas}. However, little is known beyond this case, even when $n=2$ and $K$ is unramified over $\Q_p$.\par
Indeed, Breuil and Pa{\v s}k\=unas prove that there are too many smooth admissible mod $p$ representations of $\GL_2(K)$ to have a naive correspondence with Galois representations $G_K\to \GL_2(\overline{\F}_{p})$ when $K\neq \Q_p$ \cite{breuil2012towards}. In order to determine which representations should appear in the correspondence, we need guiding principles from the global setting. Emerton proves that the $p$-adic Langlands correspondence for $\GL_2(\Q_p)$ is realized in the completed cohomology of modular curves \cite{localglobal}. Therefore, the $p$-adic and mod $p$ Langlands correspondences for $\GL_2(K)$ are expected to be realized in the completed cohomology of Shimura curves or arithmetic manifolds as follows. Fix a CM field $F$ with totally real field $F^+$ with $p$ unramified in $F$, and let $v$ be a place dividing $p$ in $F^+$ and splits in $F$ and let $w$ be a place lying above $v$ in $F$. Let $E/\Q_p$ be sufficiently large with ring of integer $\mathcal{O}$, uniformizer $\varpi$ and residue field $\F$. Let $r\colon G_F\to \GL_2(E)$ be an automorphic Galois representation with $\overline{r}$ absolutely irreducible. We are looking for a smooth admissible representation $\pi(r_w)$ of $\GL_2(F_w)$ corresponding to $r_w\colonequals r|_{G_{F_w}}$. This $\pi(r_w)$ is constructed in \cite{CEG}, and we would like to show that $\pi(r_w)$ depends only on $r_w$ but not on the choices made in the global set-up. Breuil suggested a lattice conjecture \cite{Breuiloriginal}, which provides evidence for such a claim. Assume that $r_w$ is potentially crystalline with Hodge--Tate weight $\lambda$ and tame inertial type $\tau$. Using the inertial local Langlands correspondence (proven by Henniart in the appendix of \cite{Breuil-Mezard}), we have a locally algebraic type $\sigma( \lambda,\tau)$, which is a representation of $\GL_2(\mathcal{O}_{F_w})$ over $E$, associated to $\lambda$ and $\tau$.
We write $M(K,\mathcal{O})$ for the cohomology of an appropriate Shimura curve or arithmetic manifold at level $K$. Let $\mathfrak{m}$ be the maximal ideal in the Hecke algebra corresponding to $\overline{r}$ and $\mathfrak{p}$ for the kernel of the system of Hecke eigenvalues given by $r$. Then we can consider \[\widetilde{H}[\mathfrak{p}]=\varprojlim_n\; \varinjlim_{K_p}M(K_pK^p,\mathcal{O}/\varpi^n)_{\mathfrak{m}}[\mathfrak{p}]\] where the direct limit is over all compact open subgroups of $\GL_2(F_w)$. 
Then, by local-global compatibility, $\widetilde{H}[\mathfrak{p}][\frac{1}{p}]$ as a $\GL_2(\mathcal{O}_{F_w})$-representation contains $ \sigma(\lambda,\tau)$ with multiplicity one. 
Then, the cohomology with integral coefficients $\widetilde{H}[\mathfrak{p}]$ determines a $\GL_2(\mathcal{O}_{F_w})$-stable lattice $\sigma^\circ(\lambda,\tau)$ inside $\sigma(\lambda, \tau)$. When $r_w$ is potentially Barsotti--Tate, using Shimura curves (and using a totally real field $F$ instead), Breuil showed that there are many homothety classes of lattices in $\sigma(\lambda,\tau)$, and conjectured that $\sigma^\circ(\lambda,\tau)$ is a local invariant and is determined by the Dieudonn\'e module of $r_w$. Under a genericity condition on $\overline{r}_w$ and usual Taylor--Wiles conditions, Breuil's original lattice conjecture was proven by Emerton, Gee and Savitt \cite{EGS}. Note that $r_w$ is potentially Barsotti--Tate if and only if $r_w$ has Hodge--Tate weight $(1,0)$ at all embeddings \cite{barsottitatedef}. 
In this paper, our main theorem generalizes the result to higher Hodge--Tate weights, as predicted by \cite{EGS}. (The notion of $m$-generic for $\overline{r}_w$ and for $\tau$ is defined in \cref{notation for Gal def ring}.)
\begin{theorem}\label{lattice conjecture theorem}(\cref{lattice conjecture})
Fix $n\geq 1$. If the gap of Hodge--Tate weights of $r_w$ are between $1$  and $n$, $\overline{r}_w$ and $\tau$ are sufficiently generic (depending on $n$), then the lattice $\sigma^\circ(\lambda,\tau)$ depends only on $r_w$.
\end{theorem}
In \cref{lattice conjecture}, we give an explicit formula for the lattice in terms of the Breuil--Kisin modules associated to $r_w$. In order to allow the Hodge--Tate weights to vary at different embeddings and avoid parity issues, we use arithmetic manifolds associated with unitary groups, instead of an inner form of $\GL_2$.
\subsection{Representation theory result}
Given a Serre weight $\kappa\in \JH(\overline{\sigma}(\lambda,\tau))$, we can find a unique, up to homothety, $\GL_2(\mathcal{O}_{F_w})$-invariant lattice inside $\sigma(\lambda,\tau)$, with cosocle $\kappa$, which we label as $\sigma_\kappa$. Inspired by the approach of \cite{EGS}, we study the cosocle filtration of $\sigma_\kappa$, and of its reduction $\overline{\sigma}_\kappa$. We call an irreducible representation of $\GL_2(\mathcal{O}_K)$ over $\F$ a \emph{Serre weight}; equivalently, it is an irreducible representation of $\Gamma\colonequals \GL_2(k)$ over $\F$. Putting an edge between two Serre weights for which there exists a nontrivial extension, we form an extension graph \cite{LMS} (with the idea coming from \cite{LatticeforGL3}). Let $\Inj_\Gamma \sigma$ be the injective envelope of $\sigma$ in the category of representations of $\GL_2(k)$ over $\F$. Assuming $\sigma,\tau$ are Serre weights, with $\tau$ an irreducible subquotient of $\Inj_{\Gamma}\sigma$, Breuil and Pa{\v s}k\=unas showed that there is a unique representation $I(\sigma,\tau)$ with socle $\sigma$ and cosocle $\tau$, which is multiplicity free and whose cosocle filtration is given by the extension graph between $\tau$ and $\sigma$ \cite[\S3-4]{breuil2012towards}. However, these injective envelopes in the category of representation of $\GL_2(k)$ over $\F$ are too small for our purposes. Let $K_1$ be the first congruence subgroup of $\GL_2(\mathcal{O}_K)$, $Z$ be the centre of $\GL_2(\mathcal{O}_K)$, and let $Z_1\colonequals Z\cap K_1$. We have the Iwasawa algebra $\F\llbracket K_1/Z_1\rrbracket$ which is local with maximal ideal $\mathfrak{m}_{K_1}$. We abuse notation and denote the ideal generated by the image of $\m_{K_1}$ under the natural inclusion $\F\llbracket K_1/Z_1\rrbracket\hookrightarrow\F\llbracket\GL_2(\mathcal{O}_K)/Z_1\rrbracket$ also as $\m_{K_1}$. Then $\F[\Gamma]=\F\llbracket \GL_2(\mathcal{O}_K)/Z_1\rrbracket/\m_{K_1}$. Instead of representations of $\GL_2(k)$ over $\F$, we consider representations of $\GL_2(\mathcal{O}_K)$ over $\F$ killed by $\m_{K_1}^n$ for some fixed positive integer $n$. Let $\Inj \sigma$ be the injective envelope of $\sigma$ in the category of smooth representations of $\GL_2(\mathcal{O}_K)$ over $\F$. We generalize the results in \cite[3-4]{breuil2012towards}($n=1$) and \cite[2]{HuWang2}($n=2$), and obtain the following theorem. (The notion of $m$-generic Serre weight is defined in \cref{definition in general}.)
\begin{theorem}\label{representation result}(\cref{Main theorem on generalization})
    Assuming $\sigma, \tau$ are Serre weights, which are $(2n-1)$-generic, and $\tau\in\JH((\Inj \sigma)[\mathfrak{m}_{K_1}^n]) $, there is a unique multiplicity-free representation $I(\sigma, \tau)$ of $\Inj \sigma$ with cosocle $\tau$. Moreover, the cosocle filtration of $I(\sigma, \tau)$ is determined by the extension graph between $\tau$ and $\sigma$. In particular, if $\tau\in \JH(\Inj_\Gamma \sigma)$, then $I(\sigma,\tau)$ recovers the $\Gamma$-representation defined in \cite[Corollary~3.12]{breuil2012towards}.
\end{theorem}
This theorem not only allows us to deduce the submodule structure of $\sigma_\kappa$, but also allows us to deduce that certain subquotients of $\sigma_\kappa$ are $\Gamma$-representations.
\subsection{Galois deformation ring result}
Another key input for the proof of the lattice conjecture is the notion of a patching functor, which was first developed in \cite{EGS}. We let $R_{\infty}$ be a suitable power series ring over $R^{\square}_{\overline{r}_w}$, the universal framed deformation ring of $\overline{r}_w$. A patching functor $M_\infty$ is a functor from the category of finitely generated $\mathcal{O}$-modules with a continuous $\GL_2(\mathcal{O}_{F_w})$-action to the category of coherent sheaves over $R_\infty$ satisfying some natural properties. A fundamental property of a patching functor is that
\begin{equation}\label{moding out augmented ideal intro}
\Hom_{\GL_2(\mathcal{O}_{F_w})}(\sigma_\kappa, \widetilde{H}[\mathfrak{p}])=(M_\infty(\sigma_\kappa)/\mathfrak{p})^\vee
\end{equation}
for a prime ideal $\mathfrak{p}\subset R_\infty$ corresponding to $r_w$. Let $R^{\lambda,\tau}_{\overline{\rho}}$ (resp. $R^{\leq \lambda,\tau,\reg}_{\overline{\rho}}$) be the Galois deformation ring which parametrizes potentially crystalline lifts of $\overline{\rho}$ with Hodge Tate weight $\lambda$ (resp. regular $\lambda'\leq \lambda$) and with inertial type $\tau$. As the action of $R_\infty$ on $M_\infty (\sigma_\kappa)$ factors through $R_\infty\otimes_{R_{\overline{r}_w}^\square}R^{\lambda,\tau}_{\overline{r}_w} $, we need to compute the Galois deformation rings $R^{\lambda, \tau}_{\overline{r}_w}$ with explicit genericity bounds. (The results in \cite{localmodel} are insufficient for our purpose, as the genericity bound is not explicit.)\par
We generalize the result in \cite[4]{BHHMS}, which is based on the method developed in \cite{potentiallycrystalline} and \cite{localmodel}. Let $\lambda_j\colonequals (\ell_j,0)$ for some positive integers $\ell_j$ for each $j$ and let $n\colonequals \max\{\ell_j\}$. Let $W(\overline{\rho})$ denote the set of modular Serre weights of $\overline{\rho}$ defined in \cite{BDJ}. We compute some explicit height and monodromy conditions and deduce that
\[R_{\overline{\rho}}^{\leq\lambda, \tau, \reg}\llbracket X_1, \dots,X_{2f}\rrbracket \cong(R/\sum_j I^{(j)})\llbracket Y_1,\dots,Y_4\rrbracket,\]
where $R$ is a certain power series ring over $\mathcal{O}$, and $I^{(j)}$ is generated by a set of equations that are explicit modulo $p^{n}$ and $\reg$ denotes the quotient that kills all components of non-maximal dimension. On the other hand, from these explicit equations we can deduce that $p^{2n+1}\in H$, where $H$ is the ideal used in Elkik's approximation. With these two calculations, we deduce the following theorem.
\begin{theorem}\label{Galois deformation ring intro}(\cref{normal domain})
    Assume that $\overline{\rho}$ is $(4n+1)$-generic and the tame type $\tau$ is $(2n+1)$-generic. If $W(\overline{\rho})\cap \JH(\overline{\sigma}(\lambda,\tau))\neq \emptyset$, then
    \[R^{\lambda,\tau}_{\overline{\rho}}\cong \mathcal{O}\llbracket(x_j,y_j)_{j=1}^m, Z_1,\ldots, Z_{f-m+4}\rrbracket/(x_jy_j-p)_{1\leq j\leq m}\]
    for some positive integer $m$. Recall that $m$ is determined by $2^m=|W(\overline{\rho})\cap \JH(\overline{\sigma}(\lambda,\tau))|$. In particular, $R^{\lambda,\tau}_{\overline{\rho}}$ is a normal domain and a complete intersection ring. Moreover, the special fibre $\overline{R}^{\lambda,\tau}_{\overline{\rho}}$ is reduced, and every component of the special fibre is formally smooth over $\F$ and can be identified explicitly with $W(\overline{\rho})\cap \JH(\overline{\sigma}(\lambda,\tau))$ via the isomorphism above.
\end{theorem}
This explicit description of the Galois deformation rings for higher Hodge--Tate weights was only previously known for $\lambda\leq (3,0)$ \cite[4]{BHHMS}, \cite[4]{Yitong}, and may be of independent interest. We showed these Galois deformation rings are complete intersection rings, generalizing the result in~\cite[Theorem 1.1]{crystabellinedef}, although with a more restrictive bound on the Hodge--Tate weights. We can thus deduce that certain global Galois deformation rings to be $p$-torsion free \cite[8]{crystabellinedef}. Moreover, the property of complete intersection may have applications to derived Galois deformation rings~\cite{derivedGaldef}.\par
Inducting on the distance in the extension graph, we deduce that $M_\infty(\sigma_\kappa)=\varpi(\kappa, \kappa')M_\infty(\sigma_{\kappa'})$ where $\varpi(\kappa, \kappa')$ is given by a certain element in $R^{\lambda,\tau}_{\overline{r}_w}$. Since any lattice $\sigma^\circ$ inside $\sigma(\lambda,\tau)$ can be written as $\sum_{\kappa\in \JH(\overline{\sigma}(\lambda,\tau))}p^{v(\kappa)}\sigma_\kappa$ such that $p^{v(\kappa)}\sigma_\kappa\hookrightarrow\sigma^\circ$ is saturated. We conclude the lattice conjecture using \cref{moding out augmented ideal intro}.
\subsection{Cyclicity of patched modules}
In this paper, we also show that certain patched modules are cyclic, which is closely related to proving a multiplicity one result at the Iwahori level and a conjecture of Demb\'el\'e (appendix of \cite{Breuiloriginal}).
\begin{theorem}(\cref{cyclicity of patched module})
    Under some mild genericity conditions on $\overline{r}_w$ and $\tau$, given a minimal patching functor $M_\infty$, $M_\infty(\sigma_\kappa)$ is a cyclic module over its scheme-theoretic support.
\end{theorem}
Its scheme-theoretic support is irreducible by \cref{Galois deformation ring intro}, and hence it is sufficient to show that $M_\infty(\overline{\sigma}_\kappa)$ is a cyclic $R_\infty$-module by Nakayama's lemma. Since the patching functor is an exact functor, by \cref{representation result}, we can show that $M_\infty(\overline{\sigma}_\kappa)=M_\infty(W)$ where $W$ is a subquotient of $\overline{\sigma}_\kappa$ and is isomorphic to a quotient of a lattice of $\sigma(\tau')$ for another tame type $\tau'$. We therefore deduce our theorem from the analogous result proven in the potentially Barsotti--Tate case in \cite[Theorem 10.1.1]{EGS}.
\subsection{Candidate for the mod \texorpdfstring{$p$}{TEXT} %
Langlands Correspondence}
Now let $F$ be a totally real number field in which $p$ is unramified. Fix a place $v$ lying above $p$. Let $D$ be a quaternion algebra with centre $F$, which splits at exactly one infinite place. Fix $U^v$ a compact open subgroup of $D\otimes_F\mathbb{A}_{F,f}^{v}$. Given a compact open subgroup $U$ of $(D\otimes_F\mathbb{A}_{F,f})^\times$, we let $X_U$ be the associated smooth projective Shimura curve over $F$. Letting $U_v$ run over compact open subgroups of $(D\otimes_F F_v)^\times\cong \GL_2(F_v)$, we consider
\[\pi(\overline{\rho})\colonequals \varinjlim_{U_v}\Hom_{G_F}(\overline{r},H^1_{\textrm{\small \'et}}(X_{U^vU_v}\times_F \overline{F}, \F)),\]
which is a smooth admissible representation of $\GL_2(F_v)$ over $\F$. A priori, $\pi(\overline{\rho})$ depends on $\overline{r}$ rather than $\overline{\rho}\colonequals \overline{r}|_{G_{F_{v}}}$, but it is the global candidate to correspond to $\overline{\rho}$ under the conjectural mod $p$ Langlands correspondence; hence we abuse notation and write it as $\pi(\overline{\rho})$. By \cite[Theorem 1.1]{LMS}, \cite[Theorem 1.1]{HuWang2} and \cite[Theorem 1.1]{multone}, we deduce that $\pi^{K_1}=\pi[\mathfrak{m}_{K_1}]$ is the maximal $\Gamma$-representation with socle given by $W(\overline{\rho})$ (same as \cref{pimultone}) and multiplicity free. 
We have the following result regarding $\pi$, which generalizes the result above.
\begin{theorem}\label{pimultone}(\cref{nth torsion of pi is mult free})
Under a genericity condition on $\sigma$ depending on $n$, 
   $\pi[\mathfrak{m}_{K_1}^n]$ is the largest multiplicity-free representation of $\GL_2(\mathcal{O}_{F_v})$ with socle $\bigoplus_{\sigma\in W(\overline{\rho})}\sigma$, which is killed by $\mathfrak{m}_{K_1}^n$.
\end{theorem}
Again, \cref{representation result} plays an important role in the proof, as it allows us to reduce the statement to the case where $n=2$, which was previously proven in \cite[Theorem 8.4.2]{BHHMS} and \cite[Corollary~8.13]{hu2022mod}.
\subsection{Notation}If $F$ is a field, we write $G_F\colonequals \Gal(\overline{F}/F)$. If $F$ is a number field and $v$ is a finite place in $F$, we write $F_v$ for the completion of $F$ at $v$, and we denote the ring of integers of $F_v$ as $\mathcal{O}_{F_v}$ with residue field $k_v$. We fix an embedding $\overline{F}\hookrightarrow\overline{F}_v$, which allows us to identify the decomposition group of $F$ at $v$ with $G_{F_v}$. If $K/\Q_p$ is finite, we write $I_K$ for the inertial subgroup of $G_K$. We normalize the Artin's reciprocity map so that the uniformizers are mapped to geometric Frobenius elements. We will always have $p>3$.\par
We assume $E$ to be a finite extension over $\Q_p$, which is sufficiently large, in particular, $E$ contains all embeddings of $F$. We will take $E$ to be unramified in \cref{ch6: cyclicity of patched module}. We denote $\F$ for the residue field of $E$ and $\varpi$ for the uniformizer of $E$ and $\mathcal{O}$ the ring of integers of $E$. We use $[x]$ to denote the Techn\"uller lift of $x$.
We write $\varepsilon$ (resp. $\omega$) for the $p$-adic (resp. mod $p$) cyclotomic character. We write $\omega_f$ for the Serre's fundamental character of level $f$. \par
If $F$ is a $p$-adic field and $V$ is a de Rham representation of $G_F$ over $E$, then for each embedding $\kappa\colon F\hookrightarrow E$, we have $\HT_\kappa(V)$, the multiset of Hodge--Tate weights of $V$ with respect to $\kappa$. We take the normalization such that $\HT_\kappa(\varepsilon)=\{1\}$ for all embeddings $\kappa$.\par
Let $K$ be an unramified extension of $\Q_p$ of degree $f$ with residue field $k$. We fix an embedding $\sigma_0\colon k\hookrightarrow \F$. If $\varphi$ is the arithmetic Frobenius and we let $\sigma_j\colonequals \sigma_0\circ \varphi^j$, then we can identify $\mathcal{J}\colonequals \Hom(k,\F)$ with $(\Z/f\Z)$ via $\sigma_j\xleftrightarrow[]{}j$.\par
If $V$ is a finite-dimensional representation of a group $G$ over $\mathcal{O}$, then we denote by $\overline{V}$ the reduction modulo $\varpi$ of the semi-simplification of a $G$-stable $\mathcal{O}$-lattice in V. For readability, we write $\overline{\sigma}(\lambda,\tau)$ instead of $\overline{\sigma(\lambda,\tau})$, $\overline{\sigma}(\lambda,\tau)_\kappa$ instead of $\overline{\sigma(\lambda,\tau)_\kappa}$ etc. If $R$ is a ring (for example, $\F\llbracket G\rrbracket$) and $M$ is a left $R$-module, we denote by $\soc(M)$ (resp. $\cosoc(M))$ for the socle (resp. cosocle) of $M$. (See \cite[Definition~A.3]{hu2022mod}.) We can then define the socle and cosocle filtration of $M$ inductively. If $M$ is of finite length, we denote by $\JH(M)$ the \emph{Jordan--H\"{o}lder factors} (i.e., the multiset of the composition factors). In the case where $M$ is a finite representation of $G$, this is the set of Jordan--H\"{o}lder factors of $M$ in the usual sense. If $\sigma$ is a simple $R$-module and $M$ is a finite length $R$-module, and we denote the multiplicity of $\sigma$ in $M$ as the number of times $\sigma$ appears in $\JH(M)$.\par
If $s\in S_2$ is a permutation, we let $\sgn(s)\in \{\pm1\}$ be the signature of $s$. And $\dot(s)$ is the corresponding permutation matrix. \par
We write $R^\vee$ for the Pontryagin dual $\Hom_{\mathcal{O}}^{\cont}(R, E/\mathcal{O})$ and $V^d$ for the Schikhof dual $\Hom_{\mathcal{O}}^{\cont}(V,\mathcal{O})$.
\subsection{Acknowledgment}
First and foremost, I want to thank Florian Herzig; I could not imagine having a better supervisor. I thank him for suggesting this problem to me. Throughout the process of writing this paper, he has been patient, kind, and supportive, answering many of my questions, providing many valuable insights for my research, and pointing out some mistakes in the earlier drafts. \par
I want to thank Heejong Lee, Chol Park, and Yitong Wang for many helpful discussions and teaching me many things related to the $p$-adic Langlands Program. I also thank Yitong Wang for careful reading of an early draft. I thank David Savitt for a helpful comment on the dimension of Galois deformation rings for irregular Hodge--Tate weights. I want to thank Daniel Le for a helpful insight on the Breuil--M\'ezard Conjecture. \par
\section{Representation theory results}\label{ch2: representation result}
\subsection{Notation and background}
First, we recall some notation and results concerning the extension graph for $\GL_2$ from \cite[\S~2]{BHHMS}.
Let $K$ be a fixed finite unramified extension of $\Q_p$ of degree $f$, with $\mathcal{O}_K$ its ring of integers and $k$ its the residue field. We consider the group scheme $\GL_2$ defined over $\Z$, let $T\subset \GL_2$ be the diagonal maximal torus and $Z$ its centre. We write $R$ for the set of roots of $(\GL_2, T)$, $W$ for its Weyl group, with the longest element $\mathfrak{w}$. Let $G_0$ be the algebraic group $\Res _{\mathcal{O}_K/\Z_p} \GL_{2/\mathcal{O}_K}$ with $T_0$ the diagonal maximal torus and the
centre $Z_0$. Let $\underline{G}$ be the base change $G_0 \times_{ \Z_p} \mathcal{O}$, and similarly define $\underline{T} $ and $\underline{Z}$. Let $\underline{R}$ denote the set of roots of $(\underline{G}, \underline{T})$. \par
There is a natural isomorphism $\underline{G}\cong \prod_{\mathcal{J}} \GL_{2/\mathcal{O}}$, and analogously for $\underline{T}$, $\underline{Z},\underline{R}$. We identify $X^{*}(\underline{T})$ with $(\Z^2)^{\mathcal{J}}$, and so for $\mu \in X^{*}(\underline{T})$, we write correspondingly $\mu=(\mu_j)_{j\in \mathcal{J}}$. We let $\eta_j\in X^*(\underline{T})$ be $(1,0)$ in the $j$th coordinate and 0 otherwise. We write $\eta_J \colonequals \sum_{j\in J} \eta_j$ for all $J\subset \mathcal{J}$. We define $\eta \colonequals \eta_\mathcal{J}$. Let $\alpha_j\in \underline{R}$ be $(1,-1)$ in the $j$th coordinate and $0$ otherwise. The set of positive roots of $\underline{G}$ with respect to the upper triangular Borel in each embedding is given by $\underline{R}^+=\{\alpha_j\colon 0\leq j\leq f-1\}$. We have the following definitions;
$$\text{dominant weights: }X^{*}_{+}(\underline{T}) \colonequals \{\lambda\in X^{*}(\underline{T})\colon  0\leq \langle \lambda,\alpha^{\vee} \rangle \text{ for all }\alpha\in \underline{R}^{+} \}.$$ 
$$\text{$p$-restricted weights: }X_{1}(\underline{T}) \colonequals \{\lambda\in X^{*}_{+}(\underline{T})\colon  0\leq \langle \lambda,\alpha^{\vee} \rangle\leq p-1\text{ for all }\alpha\in \underline{R}^{+} \}.$$ 
 $$\text{regular weights: } X_{\mathrm{reg}}(\underline{T}) \colonequals \{\lambda\in X^{*}_{+}(\underline{T})\colon  0\leq \langle \lambda,\alpha^{\vee} \rangle< p-1\text{ for all }\alpha\in \underline{R}^{+} \}.$$ 
 $$X^0(\underline{T}) \colonequals \{ \lambda\in X^{*}_{+}(\underline{T})\colon  \langle \lambda,\alpha^{\vee} \rangle=0\text{ for all }\alpha\in \underline{R}^{+} \}.$$
 \par
 The \emph{lowest alcove} is defined as $\underline{C}_0 \colonequals \{\lambda\in X^{*}(\underline{T})\otimes \mathbb{R}\colon 0< \langle \lambda+\eta,\alpha^{\vee}\rangle <p \text{ for all }\alpha\in \underline{R}^{+} \}$.
 Let $\underline{W}$ be the affine Weyl group of $(\underline{G}, \underline{T})$. We can identify $\underline{W}$ with $\prod_{j\in\mathcal{J}}W$, which acts on $X^*(\underline{T})\cong (\Z^2)^\mathcal{J} $ in a compatible manner. Let $\underline{W}_a \cong \Lambda_R \rtimes \underline{W}$ be the \emph{affine Weyl group}, where $\Lambda_R$ is the root lattice. And let $\widetilde{W}\cong X^{*}(\underline{T}) \rtimes \underline{W}$ be the \emph{extended Weyl group}. For $\lambda\in X^{*}(\underline{T})$, we denote  its image in $\widetilde{W}$ by $t_\lambda$. We have a $p$-dot action of $\widetilde{W}$ on $X^{*}(\underline{T})$, defined as follows: if $\widetilde{w}=wt_v \in \widetilde{W}$ and $\mu \in X^{*}(\underline{T}) $, then $\widetilde{w}\cdot \mu \colonequals w(\mu+\eta+pv)-\eta$.\par
 Let $\Omega$ be the stabilizer of the lowest alcove $C_0 $ in $\widetilde{W}$. One checks that $\widetilde{W}=\underline{W}_a \rtimes \Omega$ and $\Omega$ is the subgroup of $\widetilde{W}$ generated by $X^0(\underline{T})$ and $\{1, \mathfrak{w} t_{-(1,0)}\}$.\par
 A Serre weight of $\underline{G}_0\times_{\Z_p}\F_p$, or simply a Serre weight if it is clear from the context, is an isomorphism class of an absolutely irreducible representation
of $\underline{G}_0(\F_p) = \GL_2(k)$ over $\F$. If $\lambda\in X_{1}(\underline{T})$, we write $L(\lambda)$ for the irreducible algebraic representation of $\underline{G}\times _{\mathcal{O}} \F$ of highest weight $\lambda$, and $F(\lambda)$ for the restriction of
$L(\lambda)$ to the group $\underline{G}_0(\F_p)$. We define an automorphism $\pi$ on $X^{*}(\underline{T})$ by $\pi(\mu)_j  \colonequals \mu_{j-1}$. The map $\lambda\mapsto F(\lambda)$ induces a bijection between $X_{1}(\underline{T})/(p-\pi)X^0(\underline{T})$
and the set of Serre weights of $\underline{G}_0\times_{\Z_p}\F_p$. A Serre weight $\sigma$ is \emph{regular}
if $\sigma\cong F(\lambda)$ with $\lambda\in X_{\mathrm{reg}}(\underline{T})$.\par
Let $\Lambda_W \colonequals X^*(\underline{T})/X^0(\underline{T})$ denote the weight lattice of $(\Res_{\mathcal{O}_K/\Z_p} \SL_{2/\mathcal{O}_K}) \times_{\Z_p} \mathcal{O}$. We identify $\Lambda_W$ with $\Z^f$ as above. Given $\mu \in X^{*}(\underline{T})$, we define
 $$\Lambda_W^\mu \colonequals \{\omega \in \Lambda_W\colon  0\leq \langle \overline{\mu}+\omega, \alpha^{\vee}\rangle <p-1 \text{ for all }\alpha\in \underline{R}^{+} \},$$
where $\overline{\mu}$ denotes the image of $\mu$ in $\Lambda_W$. The set $\Lambda_W^\mu$ is called the extension graph associated to $\mu$. Moreover, we define $\Sigma \subset \Lambda_W$ as the set $\{\overline{\eta}_{J}\colon  J\subset \mathcal{J}\}$.\par
We construct a map $\mathfrak{t}'_\mu\colon  X^*(\underline{T}) \to X^*(\underline{T})/(p-\pi) X^0(\underline{T}) $ as follows: given $\omega'\in X^*(\underline{T})$, there is a unique $\widetilde{\omega}'\in \Omega\cap t_{-\pi^{-1}(\omega')}\underline{W}_a$. We set
$$\mathfrak{t}'_\mu(\omega') \colonequals \widetilde{\omega}'\cdot (\mu +\omega') \text{ mod} (p-\pi) X^0(\underline{T}).$$
This map factors through $X^*(\underline{T})/X^0(\underline{T})\cong \Lambda_W$ by the definition. Therefore, we have an induced map $$\mathfrak{t}_{\mu}\colon  \Lambda_W^\mu \to X_{\mathrm{reg}}(\underline{T})/(p-\pi) X^0(\underline{T}).$$
This map gives a bijection between $\Lambda^\mu_W$ and regular Serre weights
with central character $\mu|_{\underline{Z}_0(\F_p)}$.\par
We have the following ``change of origin" formula for the map $\mathfrak{t}_{\mu}$. For $\omega\in \Lambda^\mu_W$, we take a lift $\omega'\in X^*(\underline{T})$. Then we define $w_\omega$ as the unique image of $ \widetilde{\omega}'$ under the map $\Omega\cap t_{-\pi^{-1}(\omega')}\underline{W}_a\to \underline{W}$ as above. It can be easily checked that $w_\omega$ does not depend on the choice of the lift.
\begin{lemma}\label{change of origin} (\cite[Lemma~2.4]{BHHMS})
    Let $\omega\in \Lambda^\mu_W$ and let $\lambda \in X^*(\underline{T})$ be such that $\mathfrak{t}_{\mu}(\omega)\equiv \lambda \text{ mod } (p-\pi)X^0(\underline{T})$. Then $w_\omega^{-1}(\beta)+\omega \in \Lambda^\mu_W$ and $\mathfrak{t}_\lambda(\beta)=\mathfrak{t}_{\mu}(w_\omega^{-1}(\beta)+\omega)$ for all $\beta\in \Lambda^{\lambda}_W$. Equivalently $\mathfrak{t}_\mu(\omega')=\mathfrak{t}_{\lambda}(w_\omega(\omega'-\omega))$ for all $\omega'\in \Lambda^\mu_W$.
\end{lemma}
Following \cite[Remark~2.4.7]{BHHMS}, we see that the change of origin map is a graph automorphism. \par
 Let $K_1$ be the first principal congruence subgroup, i.e.\ , the kernel of the mod $p$ reduction morphism $\GL_2(\mathcal{O}_K)\twoheadrightarrow \GL_2(k)$. Let $Z$ be the centre of $\GL_2(K)$ and let $Z_1 \colonequals Z\cap K_1$. We have an Iwasawa algebra $\F\llbracket K_1/Z_1\rrbracket$ and we denote its maximal ideal by $\m_{K_1}$. Abusing notation, we denote the ideal generated by the image of $\m_{K_1}$ under the natural inclusion $\F\llbracket K_1/Z_1\rrbracket\hookrightarrow\F\llbracket\GL_2(\mathcal{O}_k)/Z_1\rrbracket$ as $\m_{K_1}$. Let $\Gamma \colonequals \GL_2(k)$, then $\F\llbracket\Gamma\rrbracket=\F\llbracket K_1/Z_1\rrbracket/\m_{K_1}$. We now begin to develop some terminology for the $\m_{K_1}^n$-torsion representation for some small $n$.
\begin{definition}\label{definition of ing}
Given a Serre weight $\sigma$, by inflation, we consider it as an admissible smooth $\F$-representation of $\GL_2(\mathcal{O}_K)/Z_1$. Then we define $\Proj \sigma^\vee$ (respectively $\Proj \sigma^\vee$) as the \emph{projective} (respectively \emph{injective}) envelope of $\sigma^\vee$ in the category of pseudo-compact $\F\llbracket \GL_2(\mathcal{O}_K)/Z\rrbracket$-modules. Let $\Inj \sigma$ (respectively $\Proj \sigma$) be the algebraic dual of $\Proj \sigma^\vee$ (respectively $\Inj \sigma^\vee$). Define $\Inj_n\sigma$ (respectively $\Proj_n\sigma$) to be $(\Inj_{K/Z_1} \sigma) [\m_{K_1}^n]$ (respectively $(\Proj_{K/Z_1} \sigma) /\m_{K_1}^n$). Note that $\Inj_1\sigma=\Inj_{\Gamma} \sigma$, the injective envelope in the category of admissible smooth representations of $\Gamma$. We further define $V^0 \colonequals 0$ and $V^n \colonequals V[\m_{K_1}^n]$ for positive integers $n$. Fix a Serre weight $\sigma$, then a Serre weight $\tau$ is called an \emph{$n$-weight} (with respect to $\sigma$) if $\tau$ is a subquotient of $\Inj_n\sigma$ but not of $\Inj_{n-1}\sigma$.
\end{definition}
\begin{definition}\label{definition in general}
Given $\tau=F(\mathfrak{t}_\mu(\omega))$, we let $\widetilde{\tau} \colonequals F(\mathfrak{t}_\mu(\widetilde{\omega}))$ such that $\widetilde{\omega}_j=2\lfloor\frac{\omega_j}{2}\rfloor$ if $\omega_j\geq 0$ and $\widetilde{\omega}_j=2\lceil\frac{\omega_j}{2}\rceil$ if $\omega_j\leq 0$. Note that $|\widetilde{\omega}_j|=2\lfloor\frac{|\omega_j|}{2}\rfloor$.
Given $\sigma=F(\mu)$, let
$$\Delta^k(\sigma) \colonequals \{F(\mathfrak{t}_\mu(\omega))\colon \omega\in \Z^f, \omega_j\in 2\Z \text{ for all } j \text{ and } \sum_j \frac{|\omega_j|}{2} =k \}.$$
We say $\mu\in \underline{C_0}$ is \emph{$N$-deep} if $N<\langle\mu+\eta, \alpha\rangle<p-N$ for all $\alpha\in \underline{R}^{+}$ and $F(\mu)$ is \emph{$N$-generic} if $\mu$ is $N$-deep.
\end{definition}
All the $1$-weights, i.e.\ , subquotients of $\Inj_\Gamma \sigma$ are described in the following lemma.
\begin{lemma}\cite[Lemma~6.2.1]{hu2022mod}\label{Inj1}
    Suppose that $F(\mu)$ is $0$-generic. The set of Jordan--H\"{o}lder factors of $\Inj_1 F(\mu)$ is given by $\{F(\mathfrak{t}_\mu( ( a_j)_j))\colon a_j\in \{0,\pm1\} \text{ for all }j \in \mathcal{J}\}$.
\end{lemma}
Let $\mathfrak{sl}_{2,L}$ be the space of $2\times 2$ matrices with coefficients in $L$ with trace $0$. Then $\GL_2(L)$ acts on it by conjugation, and such representation is isomorphic to $F((1,-1))_{/L}\cong \Sym^2(L^2)\otimes \det^{-1}$.
\begin{lemma}\label{lemma 3}
Suppose $\sigma$ is a $(2n-1)$-generic Serre weight.
    $$\Inj_n \sigma /\Inj_{n-1} \sigma \cong  \bigoplus_{i=0}^{n-1} \bigoplus_{\delta\in \Delta^i(\sigma)}(\Inj_1 \delta)^{\oplus k_i},$$ where $k_i\in \Z_{>0}$ with $k_{n-1}=1$.
\end{lemma}
\begin{proof}
    Consider the dual (\textit{cf.} \cite[Proposition~18.4]{Alperin}) $$\m_{K_1}^{n-1}{\Proj}_{K/Z_1} \sigma^{\vee} /\m_{K_1}^n {\Proj}_{K/Z_1} \sigma^{\vee} \cong (\m_{K_1}^{n-1}/\m_{K_1}^n) \otimes_{\mathbb{F}} \Proj_1 \sigma^{\vee}.$$
    We adapt the argument in \cite[\S7.1]{BHHMS}. For $p>2$, the map $K_1\cap\SL_2(L)\to K_1/Z_1$ is an isomorphism. Furthermore, the map $x\mapsto\exp(px)$ induces a homeomorphism from $\mathfrak{sl}_{2,\mathcal{O}_L}$ to $K_1\cap\SL_2(L)$. (See \cite[III.1.1.4, III.1.1.5, III.1.1.8]{Lazard}.) Therefore, the group $K_1/Z_1$ is a standard group, hence is uniform by \cite[Theorem 8.31]{analyticpropgroup}. Therefore, by \cite[\S 7.4]{analyticpropgroup}, the ring $\F\llbracket K_1/Z_1\rrbracket$ is a polynomial ring, and $$\m_{K_1}^{n-1}/\m_{K_1}^n\cong \text{Sym}^{n-1}(\m_{K_1}/\m_{K_1}^2).$$
    Moreover, $$\m_{K_1}/\m_{K_1}^2\cong \oplus_{j=0}^{f-1} F((1,-1)^{(j)}),$$
    and 
    $$F(a,b)\otimes F((1,-1)^{(j)})\cong F((a_i+\delta_{ij}, b_i-\delta_{ij})_i)\oplus F((a_i,b_i)_i)\oplus F((a_i-\delta_{ij}, b_i+\delta_{ij})_i),$$
    for all $2\leq a_j-b_j\leq p-3$, which is always satisfied because of the genericity condition. (Here $(1,-1)^{(j)}$ denotes the weight vector, which is $(1,-1)$ in the $j$th coordinate and $0$ otherwise, and $\delta_{ij}$ is the Kronecker delta.) The result then follows.
\end{proof}
\begin{lemma} \label{description of JH(Injk)}
    Let $\sigma=F(\mu)$ be a $(2n-1)$-generic Serre weight and $\tau=F(\mathfrak{t}_{\mu}(\omega))$ be a $k$-weight, where $k\leq n$. Then, $\tau\in \Inj_1\widetilde{\tau}$, where $ \widetilde{\tau}\in \Delta^{k-1}(\sigma)$. In particular, $\tau$ is a $k$-weight if and only if $\sum_j\lfloor\frac{|\omega_j|}{2}\rfloor=k-1$. In such a case, $\tau$ is $2(n-k)$-generic and hence a regular Serre weight.
\end{lemma}
\begin{proof}
    If $\tau$ is a $k$-weight with respect to $\sigma$, then by \cref{lemma 3}, $\tau\in \JH(\Inj_1 \theta)$ for some $\theta\in \Delta^{k-1}(\sigma)$ and $\tau\notin \JH(\Inj_1 \theta')$ for all $\theta'\in \Delta^{k'}(\sigma)$ with $k'<k-1$.\par
    Suppose $\tau\in \Inj_1 \theta$ where $\theta=\colon F(\mathfrak{t}_{\mu}(\xi))\in \Delta^{k-1}(\sigma)$. Then $\xi_j\in 2\Z$ for all $j$ and $ \sum_j |\frac{\xi_j}{2}|=k-1$. We will show that $\xi=\widetilde{\omega}$.
    Assume for the sake of contradiction that there exists some $j$, $\xi_j\neq\widetilde{\omega}_j$, then as $\xi_j, \widetilde{\omega}_j\in 2\Z$, by the definition of $\widetilde{\omega}$, $\xi_j\neq \omega_j$. Therefore, we must have $|\xi_j|<|\widetilde{\omega}_j|\leq|\omega_j|$ or $|\widetilde{\omega}_j|\leq |\omega_j|\leq|\xi_j|$. By \cref{Inj1}, as $\tau\in \JH(\Inj_1\theta)$, $|\xi_j-\omega_j|=1$; hence the first scenario is impossible as $\xi_j, \widetilde{\omega}_j\in 2\Z$. If the second scenario holds, then $|\widetilde{\omega}_j|=|\xi_j|-2$, and hence $\widetilde{\tau}\in \Delta^{k'}(\sigma)$, for some $k'<k$, also a contradiction. 
    For the last part, since $\sigma$ is a $2(n-1)$-generic Serre weight, $2n-2 <\langle \mu, \alpha_i^{\vee}\rangle< p-2n.$ 
If $\tau=F(\mathfrak{t}_{\mu}(\omega))$ is a $k$-weight, then for all $i$,
$$2n-2k-1< \langle \mu, \alpha_i^{\vee} \rangle-|\omega_i|\leq \langle \mu+\omega, \alpha_i^{\vee}\rangle\leq\langle \mu, \alpha_i^{\vee} \rangle+|\omega_i|< p-2n+2k-1.$$
Therefore, $\mathfrak{t}_{\mu}(\omega)\in \Lambda^\lambda_W$, and $\tau$ is a regular Serre weight.
\end{proof}
 From now on, we will assume that $\sigma=F(\mu)$ and that all Serre weights are regular.
\begin{definition}\label{definition on relations}
Given $\omega, \omega'\in \Lambda^\mu_W$, we write $\omega\leq  \omega'$, if for each $j$, $0\leq \omega_j\leq \omega'_j $ or $0\geq \omega_j \geq \omega'_j $. Suppose $\kappa=F(\mathfrak{t}_\mu(\omega)), \kappa'=F(\mathfrak{t}_\mu(\omega')),\kappa''=F(\mathfrak{t}_\mu(\omega''))$ with $\omega, \omega', \omega''\in \Lambda^\mu_W$. We write $\kappa'-\kappa\leq \kappa''-\kappa$ if we have $\omega'-\omega\leq \omega''-\omega $ in the above sense. Note that $ \kappa'-\kappa\leq \kappa''-\kappa$ is equivalent to $\kappa'-\kappa''\leq \kappa-\kappa''$. We simply write $\kappa'\leq \kappa''$ if $\kappa=F(\mu)$. Because of the bijection between $\Lambda_W^\mu$ and regular Serre weights with central character $\mu|_{\underline{Z}_0(\F_p)}$, given $\kappa=F(\mathfrak{t}_\mu(\omega))$, we sometimes simply write $\{\kappa'\leq\kappa\}$ to denote the set $\{\kappa'=F(\mathfrak{t}_\mu(\omega'))\colon \omega'\leq \omega\}$.\par
    Fix $\tau=F(\mathfrak{t}_\mu(\omega))$ a regular Serre weight. We define the following:
$$\Omega^\tau_k \colonequals \{F(\mathfrak{t}_\mu(\omega'))\colon  F(\mathfrak{t}_\mu(\omega'))\leq \tau \text{ and } \sum_j \left\lfloor\frac{|\omega'_j|}{2}\right\rfloor =k \}.$$ 
$$\prescript{0}{}\Omega^\tau_k \colonequals \{F(\mathfrak{t}_\mu(\omega'))\colon  F(\mathfrak{t}_\mu(\omega'))\leq \tau,  \omega'_j\in 2\Z 
 \text{ for all }j \text{ and } \sum_j \frac{|\omega'_j|}{2} =k\}\subset \Omega_k^\tau.$$
Moreover, given $\kappa=F(\mathfrak{t}_\mu(\nu))\in \Delta^m(\sigma)$ for some $m$.
Let $$(\nu_+)_k \colonequals \nu_k+ \epsilon(\omega_k-\nu_k),$$
where $\epsilon(x)=\sgn(x)$ if $x\neq 0$ and $\epsilon(0)=0$. Define $\kappa_+ \colonequals F(\mathfrak{t}_\mu(\nu_+))$. If $\kappa\leq \tau$, then $\kappa_+\leq \tau$.
The condition that $\xi\in  \prescript{0}{}\Omega^\tau_k$ is equivalent to $\xi\in \Delta^k(\sigma)$ and $\xi\leq \tau$.\par
Moreover, we define $\omega^{(i)}$ to be the element such that $\omega^{(i)}_k=\omega_k$ for some $k\neq i$ and $\omega^{(i)}_i=0$.
Define $\tau^{(i)} \colonequals F(\mathfrak{t}_{\mu}(\omega^{(i)})) $. We further define $\delta_i^{\epsilon_i}(\sigma) \colonequals F(\mathfrak{t}_{\mu}(2\epsilon_i\overline{\eta}_i))$ and $\mu_i^{\epsilon_i}(\sigma) \colonequals F(\mathfrak{t}_{\mu}(\epsilon_i\overline{\eta}_i)).$
\end{definition}
 Suppose that $\sigma, \tau$ are regular Serre weights and $\tau$ is a subquotient of $\Inj_1 \sigma'$. By \cite[Corollary 3.12]{breuil2012towards}, there exists a unique $\Gamma$-representation $I(\sigma,\tau)$ with socle $\sigma$, cosocle $\tau$, such that $I(\sigma,\tau)$ is multiplicity-free. Moreover, its socle filtration is determined in \cite[Corollary 4.11]{breuil2012towards}, which we will reinterpret using the extension graph as follows.
\begin{lemma}\label{reformulation of I}
   In the setting of the discussion above, a regular Serre weight $\tau'$, which is a subquotient of $\Inj_1 \sigma'$, occurs as a subquotient in $I(\sigma, \tau)$, if and only if $
    \tau'-\sigma\leq \tau-\sigma$. 
\end{lemma}
\begin{proof}
 Suppose $\sigma=F(\mathfrak{t}_{\mu}(\gamma))$, $\tau \colonequals F(\mathfrak{t}_{\mu}(\omega)),$ $\tau' \colonequals F(\mathfrak{t}_{\mu}(\omega'))$. We apply the change of origin formula \cref{change of origin}, and send $F(\mathfrak{t}_\mu(\omega))\mapsto F(\mathfrak{t}_\mu(\omega-\gamma))$. By \cref{lemma 3}, for all $\tau, \tau'\in \JH(\Inj_1 \kappa)$, we must have $(\omega-\gamma)_j, (\omega'-\gamma)_j\in \{-1, 0,1\}$ for all $j$. In \cite[Corollary~4.11]{breuil2012towards}, the condition for $\tau'$ to occur as a subquotient in $I(\kappa, \tau)$ is given by $\mathcal{S}(\lambda')\subset \mathcal{S}(\lambda)$ and $\lambda, \lambda'$ being compatible. In our notation, $\mathcal{S}(\lambda)=\{j\colon \omega_j\neq0\}$ and $\mathcal{S}(\lambda')=\{j\colon \omega'_j\neq0\}$. Moreover, $\lambda, \lambda'$ is compatible if and only if $\sgn(\omega_j)=\sgn(\omega'_j)$ when $\omega_j, \omega'_j \neq 0$. Therefore, the condition that $\mathcal{S}(\lambda')\subset \mathcal{S}(\lambda)$ and $\lambda, \lambda'$ are compatible is equivalent to $\omega-\gamma \leq \omega'-\gamma$.
\end{proof}
\begin{lemma}\label{lemma 10}
    For any Serre weight $\sigma$ and any $\tau\in \JH(\Inj_n\sigma)$, there exists a subrepresentation $V$ of $\Inj_n \sigma$ such that $\cosoc(V) = \tau$ and $[V \colon  \sigma] = 1$ (hence $\sigma$ occurs in $V$ as a sub-object).
\end{lemma}
\begin{proof}
    The proof goes exactly as in \cite[Lemma~2.22]{hu2022mod} with $\Inj_{\widetilde{\Gamma}} \sigma$ (respectively, $\Proj_{\widetilde{\Gamma}} \sigma$) replaced by $\Inj_n \sigma$ (respectively $\Proj_n \sigma)$. 
\end{proof}
\subsection{Main result}
From now on, we will assume $n,m\in \Z_{\geq 1}$.
\begin{theorem}\label{Main theorem on generalization}
Let $\sigma=F(\mu)$ be a $(2n-1)$-generic Serre weight. 
Assume that $V$ is a subrepresentation of $\Inj_n \sigma$ with irreducible cosocle $\tau$ and $[V\colon \sigma]=1$. If $\tau$ is an $m$-weight, then $V$ is multiplicity free, $\m_{K_1}^{m}$-torsion (that is, $m=n)$, and uniquely determined by $\sigma, \tau$ up to scalar multiplication. Moreover, for $0\leq k\leq m-1$, we have $$V^{k+1}/V^k\cong \bigoplus_{\nu \in \prescript{0}{}\Omega_k^\tau} I(\nu,\nu_+)$$
where $\nu_+$ is defined before \cref{reformulation of I} and $I(\nu,\nu_+)$ is the $\Gamma$-representation defined in \cite[Corollary~3.12]{breuil2012towards}  (cf. \cref{reformulation of I}).  
\end{theorem}
By \cref{lemma 3}, $m$ is the smallest positive integer such that $\tau\in\JH(\Inj_m(\sigma))$.\par
As before, we denote such a representation by $I(\sigma, \tau)$. When $\tau$ is a $1$-weight, then $I(\sigma, \tau)$ is a $\Gamma$-representation and coincides with the definition in \cite[Corollary~3.12]{breuil2012towards}, and when $\tau$ is a $2$-weight, $I(\sigma, \tau)$ also coincides with the definition in \cite[Theorem~2.30]{hu2022mod}. This theorem is a generalization of \cite[Theorem~2.23]{hu2022mod}. \par
To elucidate the theorem, we have the following lemma as a remark.
\begin{lemma}\label{remark of theorem}
Given $\nu \in \prescript{0}{}\Omega_k^\tau$ for some $k<n$. Assume that $\kappa$ is a regular Serre weight. 
\hfill\begin{enumerate}[label=(\roman*)]
    \item $\kappa-\widetilde{\kappa}\leq \widetilde{\kappa}_+-\widetilde{\kappa}$ if and only if $\kappa\leq \tau$.
    \item $\kappa\in \JH(I(\nu,\nu_+))$ if and only if $\nu=\widetilde{\kappa}$ and $\kappa\leq \tau$. In this case, $\kappa$ is a $k$-weight.
    \item Assume that \cref{Main theorem on generalization} holds, the Jordan--H\"{o}lder factors of $V$ are exactly those $\kappa$ where $\kappa\leq \tau$.
    \item Given $\nu' \in \prescript{0}{}\Omega_{k'}^\tau$, $I(\nu,\nu_+)$ and $I(\nu',\nu'_+)$ do not share a common Jordan--H\"{o}lder factor if $\nu\neq \nu'$.
\end{enumerate}
\end{lemma}
\begin{proof}
We can assume $\kappa=F(\mathfrak{t}_\mu(\omega'))$ and $\nu=F(\mathfrak{t}_\mu(\alpha))$.\par
(i) 
Assume $\omega'\leq \omega$, then for each $j$, $0\leq \omega'_j\leq \omega_j$ or $0\geq \omega'_j\geq \omega_j$. If it is the former case, then it follows from the definition that $0\leq \widetilde{\omega}'_j\leq \omega'_j\leq \widetilde{\omega}'_{+_j}$, and analogously for the latter case. Conversely, assume $ \omega'-\widetilde{\omega}'\leq \widetilde{\omega}'_+-\widetilde{\omega}'$. For each $j$, if $0\leq 2\lfloor \frac{\omega'_j}{2}\rfloor \leq \omega'_j \leq  \omega'_j+\epsilon(\omega_j-\omega'_j)$, then $0\leq  \omega'_j \leq \omega_j $. And the same result holds if $\leq $ is replaced by $\geq$ and the floor function is replaced by the ceiling function.\par
(ii) If $\kappa\in \JH(I(\nu, \nu_+))$, by \cref{reformulation of I}, $\kappa-\nu\leq \nu_+-\nu$, and hence $|\alpha_j|\leq|\omega'_j|\leq |\alpha_{+_j}|$. As $\nu,\nu_+$ are $k+1$-weights, by \cref{description of JH(Injk)}, so is $\kappa$. Moreover, then by \cref{description of JH(Injk)}, we deduce that $\nu=\widetilde{\kappa}$. By (i), we deduce that $\omega'\leq \omega$. The converse follows from (i) and \cref{reformulation of I}.\par 
(iii) Given $\kappa\in \JH(V)$, then $\kappa\in \JH(I(\nu,\nu_+))$ for some $ \nu \in \prescript{0}{}\Omega_k^\tau$, and by (ii), $\nu\cong \widetilde{\kappa}$ and $\kappa\leq \tau$. Conversely, if $\kappa\leq \tau$, by (ii), $\kappa\in \JH(I(\widetilde{\kappa}, \widetilde{\kappa}_+))$ and the result follows by (ii).\par
(iv) It follows from (ii).
\end{proof}
These three corollaries follow immediately from the theorem.
\begin{corollary}\label{cor on multiplicities}
Let $V$ be a subrepresentation of $\Inj_n \sigma^{\oplus s}$ for some $s \geq 1$. Then for any irreducible Serre weight $\tau $, we have $[V \colon  \sigma] \geq [V \colon  \tau]$. Moreover, if $\cosoc(V)$ is isomorphic to $\tau^{\oplus r}$ for some $(2n-1)$-generic Serre weight $\tau$ and some $r \geq 1$, then $[V \colon  \sigma] = [V \colon  \tau]$.
\end{corollary}
\begin{proof}
    It follows verbatim from \cite[Corollary~2.3]{hu2022mod}. Since $\soc(V)$ has the form $\sigma^{\oplus s'}$ for some $s' \leq s$, we can construct a finite filtration of $V$ such that each graded piece has socle isomorphic to $\sigma$ and $\sigma$ occurs only once there. Hence, we reduce it to the situation where $\soc(V)=\sigma$ and $[V \colon  \sigma] = 1$, and the result follows from \cref{Main theorem on generalization}. The second assertion follows by duality.
\end{proof}
\begin{corollary}\label{Cor on socle filtration}
 Assume that $\sigma$ is $(2n-1)$-generic, and $\theta,\theta',\tau\in \JH(\Inj_{n}\sigma)$. Then $\theta'$ is a subquotient of $I(\theta,\tau)$ if and only if $\theta'-\theta\leq \tau-\theta$. Furthermore, if $\theta=F(\mathfrak{t}_\mu(\omega'))$ and $\tau=F(\mathfrak{t}_\mu(\omega))$, then the graded pieces of its socle filtration are given by:
$$I(\theta,\tau)_k\cong \bigoplus_{\omega'-\omega\leq \omega''-\omega, \sum_j|\omega''_j-\omega_j|=k} F(\mathfrak{t}_{\mu}(\omega'')).  $$
In other words, the socle filtration coincides with the filtration given by the distance from the socle $\theta$ in the extension graph. 
\end{corollary} 
\begin{proof}
As in \cref{reformulation of I}, by the change of origin map $F(\mathfrak{t}_\mu(\omega))\mapsto F(\mathfrak{t}_\mu(\omega-\gamma))$, we change the origin to $\mathfrak{t}_\mu(\gamma)$. The proof follows the same way as in \cite[Corollary~2.35]{hu2022mod}, which reformulates the theorem using \cite[Corollary~4.11]{breuil2012towards}, which is reinterpreted in light of \cref{reformulation of I} and \cref{remark of theorem}.
\end{proof}
\begin{remark}
Comparing the results when $n=1$ and $n=2$ with \cite{breuil2012towards} and \cite{HuWang2}, our genericity condition is higher by $1$ for the induction argument in \cref{uniqueness when m=n}.\par
We refer to the Serre weights in the extension graph as lattice points here. Intuitively, the theorem is saying that the Jordan H\"older factors of $I(\sigma, \tau)$ are given by the lattice points within (including the boundary of) the smallest hypercuboid with opposite corners given by $\sigma$ and $\tau$, which we simply call the hypercuboid. Moreover, the $n$-th socle filtration of $I(\sigma,\tau)$ is given by the lattice points of distance $n$ from $\sigma$ in the hypercuboid. \par
Let $d_j(\alpha,\beta)$ denote the distance between $\alpha$ and $\beta$ in the $j$-th direction in the extension graph. The general strategy of the proof is as follows. We first handle the case where $n=3$ and $\tau$ is right next to $\sigma$ \cref{simplest 1-weight}. In general, we first show that $\soc(V/V^1)$ are exactly the lattice points at distances two from $\sigma$ in the direction $j$ where $d_j(\sigma,\tau)\geq 2$ (\textit{cf.} \cref{description of soc(V/V1)}, \cref{exact soc(V/V^1) II}). (It is empty if $\tau$ is a $1$-weight.) We need to make sure these Serre weights appear with multiplicity one; in particular, they are not in $\JH(V/V^{n-1})$ \cref{soc(V/Vn-1)}.
If $m<n$, then we apply the induction hypothesis to $V/V^1$ and deduce that $V/V^{1}$ is $\mathfrak{m}_{K_1}^{m-1}$-torsion \cref{theorem for m<n} and finish the proof. \par
If $n=m$, the second step is to deduce that $\JH(V)$ is as conjectured by the theorem up to multiplicity \cref{JH factors when m=n}. We deduce that $\JH(V/V^1)$ is correct using our induction hypothesis. Then, we consider the Jordan--H\"{o}lder factors given by shifting the Serre weights in $\soc(V/V^1)$ towards $\tau$ by at most $1$ in every direction (no shifting in the direction where it is equal to $\tau$), and deduce that $\JH(V^1)$ is correct. In the example below for $n=3$ and $f=2$ (\cref{fig:example}), the green dots denote the Serre weights in $\soc(V/V^1)$ and the lattice points inside the green rectangle are the Jordan--H\"older factors of $V/V^1$. Then we know that $x,y\in \JH(V^2)$, and hence $\JH(V^2)$ contains and is indeed given by the lattice points inside the orange rectangle. \par
The third step is to prove that $V$ is multiplicity free \cref{m=n multiplicity free}. As $V^{n-1}$ is multiplicity free by the induction hypothesis, it suffices to show that for $V/V^{n-1}$. Equivalently, we have to show that $\soc(V/V^{n-1})$, which is irreducible, appears with multiplicity $1$. We quotient $V$ by a subrepresentation whose Jordan--H\"older factors are given by all the integral points on a face of the hypercuboid and do not contain any Jordan--H\"older factor of $V/V^{n-1}$. By the induction hypothesis, we can deduce that $\soc(V/V^{n-1})$ appears with multiplicity $1$. In the example below, the Jordan--H\"older factors of this quotient are given by the lattice points inside the blue rectangle. \par
Finally, we show the uniqueness of $V$ by showing that the dimension of the extension between a subrepresentation $V'$ of $V$ and the quotient $\widetilde{V}=V/V'$ is $1$ \cref{uniqueness when m=n}. Instead, we replace $V'$ by its quotient, where the Jordan--H\"older factor of the subrepresentation we quotient out are the integral points on a face of the cuboid. Similarly, we replace $\widetilde{V}$ with a subrepresentation of $\widetilde{V}$, such that the Jordan--H\"older factors ``we lost" are the integral points of the face in the hypercuboid parallel to the one we removed above. By the induction hypothesis, we conclude that the extension between the subquotients has dimension one. In the example below, the Jordan--H\"older factors of $V'$ (resp.$\widetilde{V}$) are given by the lattice points inside the red rectangle (resp. green rectangle), and the subquotients we replace them with have Jordan--H\"older factors given by the lattice points on the two vertical line segments in the middle. \par
\begin{figure}[h!]
\begin{minipage}[c]{0.3\linewidth}
    \centering
    \includegraphics[width=\linewidth]{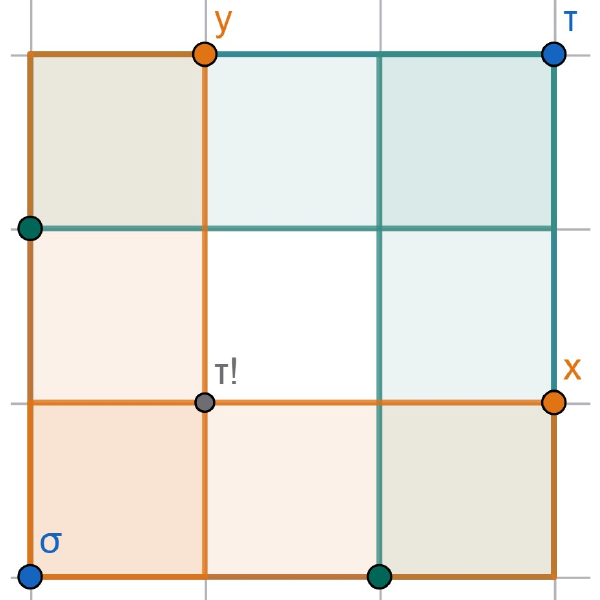}
\caption{step 2}
\end{minipage}
\begin{minipage}[c]{0.3\linewidth}
    \centering
\includegraphics[width=\linewidth]{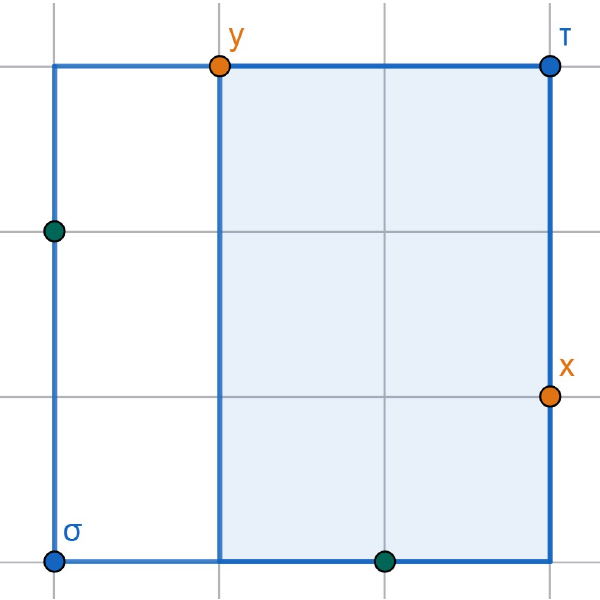}
 \caption{Step 3}     
\end{minipage}
\begin{minipage}[c]{0.3\linewidth}  
    \centering
    \includegraphics[width=\linewidth]{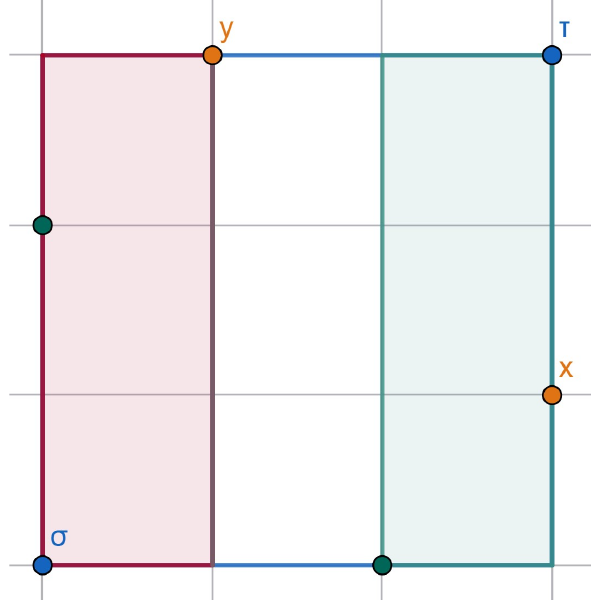}
    \caption{Step 4}
\end{minipage} \caption{Example when $m=n=3$}\label{fig:example}
\end{figure}
\end{remark}
\subsection{Preliminary lemmas}
Before we prove the theorem, we first prove the following lemmas.
   \begin{lemma}\label{lemma on socles}
    Let $V$ be a $\m_{K_1}^n$-torsion representation.
    \hfill\begin{enumerate}    
        \item $\soc(V^k/V^{k-1})=\soc(V/V^{k-1})$ and $\cosoc(V^k)=\cosoc(V^k/V^{k-1})$ for all $k$.
        \item Let $T$ be a proper subrepresentation of $V^{n-1}$. In addition, assume that $\soc(V/V^{n-1})\eqqcolon\theta$ is irreducible, and $\theta\not \subset\soc(V)$, and $\Ext^1_{K/Z_1}(\theta, \sigma')=0$ for all $\sigma'\in \JH(T)$. Then $\soc (V^{n-1}/T)=\soc(V/T)$.
    \end{enumerate}
    \end{lemma}
\begin{proof}
(i) From the exact sequence
$0 \to V^k/V^{k-1}\to V/V^{k-1} \to V/V^k \to 0$, we have $\soc(V^k/V^{k-1})\hookrightarrow\soc(V/V^{k-1}) $. Conversely, if $\sigma\subset\soc(V/V^{k-1})$ is a Serre weight, then $\sigma\subset (V/V^{k-1})[\m_{K_1}]=V^k/V^{k-1}$. Therefore, $\soc(V^k/V^{k-1})=\soc(V/V^{k-1})$.\\
For the cosocle case, we apply the same argument dually, noticing that the $\m_{K_1}$-torsion quotient of $V^k$ is $V^k/V^{k-1}$.\par
(ii) Since $\soc(V/T)\hookrightarrow\soc(V^{n-1}/T)\oplus \soc(V/V^{n-1})$, suppose to the contrary that $T\neq V^{n-1}$ but $\theta\subset \soc(V/T)$. Then let $\pi_T:V\to V/T$ be the quotient map, and consider the subrepresentation $\pi_T^{-1}(\theta)$. By construction, all the Jordan--H\"{o}lder factors of $\pi_T^{-1}(\theta)$ except $\theta$ are in $\JH(T)$. Since $\Ext^1_{K/Z_1}(\theta, \sigma')=0$ for all $\sigma'\in \JH(T)$, by d\'evissage, we deduce that $\Ext_{K/Z_1}^1(\theta, T)=0$. 
Therefore, since $\pi_T^{-1}(\theta)$ is a subrepresentation of $V$, $\theta \subset \soc(\pi_T^{-1}(\theta)) \subset \soc(V)$, a contradiction.
\end{proof}
\begin{lemma}\label{n=m irred soc(V/Vn-1)} Suppose $V$ as in \cref{Main theorem on generalization} with $m=n$, then $\soc(V/V^{n-1})\cong \widetilde{\tau}$ is irreducible.
\end{lemma}
 \begin{proof} By \cref{lemma 3}, there is an embedding,
$V/V^{n-1}\hookrightarrow\oplus_{i=0}^{n-1} \oplus_{\theta\in \Delta^i(\sigma)}(\Inj_1 \theta)^{\oplus k_i}$ where $k_n=1$. 
Since $\cosoc(V)=\tau$ is an $n$-weight, by \cref{description of JH(Injk)}, we must have $V/V^{n-1} \hookrightarrow \Inj_1 \widetilde{\tau}$.
\end{proof}
\begin{lemma}\label{description of soc(V/V1)}
Assume that $\sigma=F(\mu)$ and $\theta=F(\mathfrak{t}_{\mu}(\omega))$ is an $n$-weight. Let $V$ be a subrepresentation of $\Inj_n \sigma$ with $[V \colon \sigma] = 1$ and $\soc(V/V^{n-1})=\theta$. We have $\soc (V/V^1)\hookrightarrow \bigoplus_{|\omega_i|>1} \delta_i^{\sgn(\omega_i)}(\sigma)$.
\end{lemma}
\begin{proof}
By \cref{lemma 3}
\begin{equation}\label{equation on soc(V^2/V^1)}
V^2/V^1 \hookrightarrow (\Inj_1 \sigma) ^{\oplus k_0}\oplus_{\delta' \in \Delta(\sigma)} \Inj_1 \delta'.\end{equation}
Hence, as $[V^2/V^1\colon \sigma]=0, $ $\soc (V^2/V^1)\hookrightarrow \oplus_{(i,\epsilon_i)} \delta_i^{\epsilon_i}(\sigma)$ where $i\in \mathcal{J}$ and $\epsilon_i\in \{\pm\}$. By \cref{lemma on socles}, we have $\soc(V/V^1)=\soc(V^2/V^1).$
If $\delta_i^{\epsilon_i}(\sigma)\subset \soc(V/V^1)$, we can find a subquotient $V'$ of $V/V^1$, such that $\soc(V')=\delta_i^{\epsilon_i}(\sigma)$ and $\cosoc(V')=\theta$. Then $\theta\in \JH(\Inj_{n-1}\delta_i^{\epsilon_i}(\sigma))$. Therefore, $\sum_{j\neq i}\frac{|\omega_j|}{2 }+| \frac{\omega_i}{2 }-\epsilon_i 1|\leq n-2$. Since $\sum_j \frac{|\omega_j|}{2}=n-1$, we conclude that $\omega_i\geq 2$ and $\epsilon_i=\sgn(\omega_i)$.
\end{proof}
Assuming $m<n$, we will prove by contradiction that $V/V^{n-1}= 0$. We now show that it is sufficient to disprove the case where $V/V^{n-1}$ is irreducible.\par
\begin{lemma}\label{reduction}
   Suppose $V$ as in \cref{Main theorem on generalization}. Assume $\theta \subset \soc(V/V^{n-1})$ for some Serre weight $\theta$, then $V$ contains a subrepresentation $V'$ with $V'/V'^{n-1}\cong I(\theta, \tau)$, as well as a subrepresentation $V''$ with $V''/V''^{n-1}\cong \theta$.
\end{lemma}
\begin{proof} Assume $V/V^{n-1}\neq 0$. By \cref{lemma 3}, since $[V\colon \sigma]=1$, 
\begin{equation}\label{equation on V/Vn-1}
    V/V^{n-1}\hookrightarrow \bigoplus_{i=1}^n \bigoplus_{\theta\in \Delta^i(\sigma)}(\Inj_1 \theta)^{\oplus k_i}.
\end{equation}
For $\tau$ a $m$-weight, with $m<n$, $\tau$ may occur in distinct $\Inj_1 \theta'$ for some $\theta' \in \bigcup_{s=1}^{n-1}\Delta^s(\sigma)$. By assumption, \cref{equation on V/Vn-1} induces a nonzero map $\pi_\theta\colon V/V^{n-1} \to \Inj_1 \theta$ when composed with the natural projection to $\Inj_1 \theta$. We call the image $C_\theta$.
If $[C_\theta\colon \theta]=1$, then we are done by \cite[Corollary~3.12]{breuil2012towards}. Otherwise, we dualize $C_\theta$, such that $\soc(C_\theta^{\vee})=\tau^{\vee}$ and $\cosoc(C_\theta^{\vee})=\theta^{\vee}$. Then we can find a quotient $\widetilde{V'}$ in $C_\theta^{\vee}/\tau^{\vee}$ with socle $\tau^{\vee}$, and hence $[\widetilde{V'}\colon \tau^{\vee}]<[C_\theta^{\vee}\colon \tau^{\vee}]$. By repeating the process, we eventually find a quotient $\widetilde{V}$ of $C_\theta^{\vee}$, with $\soc(\widetilde{V})=\tau^{\vee}$, $\cosoc(\widetilde{V})=\theta^{\vee}$ and $[\widetilde{V}\colon \tau^{\vee}]=1$. Then, by \cite[Corollary~2.3]{hu2022mod}, $\widetilde{V}^{\vee}\cong I(\theta, \tau)$ and is a subrepresentation of $C_\theta$. Furthermore, let $\pi\colon V\to V/V^{n-1}$ be the projection map, then $V' \colonequals \pi^{-1}(I(\theta, \tau))$ is a subrepresentation of $V$. Moreover, $V'^{n-1}= V'\cap V^{n-1}$. Therefore, $V'/V'^{n-1}\cong I(\theta, \tau)$. For the last part, since $C_\theta$ contains $\theta$ as a subrepresentation, $\pi_\theta^{-1}(\theta)$ is a subrepresentation of $V$ and is the $V''$ we are looking for using the same argument as above.
\end{proof}
\begin{lemma}\label{lemma on k-weight in Vk}
Assume that \cref{Main theorem on generalization} holds for $\m_{K_1}^{n-1}$-torsion representations. Suppose $V$ is a subrepresentation of $\Inj_n \sigma$ as in \cref{Main theorem on generalization}. If $\sigma'\in \JH(V)$ and $\sigma'$ is a $k$-weight for some $k< n$, then $I(\sigma, \sigma')$ is a subrepresentation of $V^k$ and $\sigma'$ is a $k$-weight. 
\end{lemma}
\begin{proof}
We can find a subrepresentation $V'$ of $V$ with cosocle $\sigma'$. Since the theorem holds for all $m<n$; therefore, $V'\cong I(\sigma, \sigma')$, a $\m_{K_1}^{k}$-torsion representation. Therefore, $V'$ is a subrepresentation of $V^k$. Moreover, this implies $\sigma'\in \JH(V^{k})$.
\end{proof}
\subsection{Proof of the main theorem}
\begin{proof}
We will prove by induction on lexicographic order $(n,m)$. 
When $n=1$, it is given by \cite[Corollary~3.12]{breuil2012towards} and when $n=2$, it is given by \cite[Theorem~2.23]{hu2022mod}. Since the following propositions, where the case $n<3$ applies, can be deduced from \cref{Main theorem on generalization}, we will assume without loss of generality that $n\geq 3$.\par
When we apply the induction hypothesis in the case where $(m,n)=(a,b)$, we write \cref{Main theorem on generalization}$ [(a,b)]$, \cref{cor on multiplicities}$ [b]$ or \cref{Cor on socle filtration}$ [b]$ (since \cref{Main theorem on generalization} shows that $a=b$ in such a case).\par
First, we prove the theorem in the case where $m<n$.
\begin{proposition}\label{simplest 1-weight}
    Suppose $V$ as in \cref{Main theorem on generalization} with $n=3$ and $\tau=\mu_i^{\epsilon}(\sigma)$ for some $\epsilon\in \{\pm\}$, then $V\cong I(\sigma, \tau)$, a $\Gamma$-representation, as predicted in the theorem.
\end{proposition}
\begin{proof}
Suppose that $V/V^2\neq 0$. By \cref{reduction}, it suffices to consider the subrepresentation $V'$ of $V$ with $\soc(V'/V'^2)\eqqcolon\theta$ irreducible and to show that such a subrepresentation does not exist. If $\theta=F(\mathfrak{t}_\mu(\xi))\in \Delta^2(\sigma)$, there exists $i$ with $|\xi_i|\geq 4$, or there exists $i\neq j$ with $|\xi_i|,|\xi_j|\geq 2$. By \cref{Inj1}, $\tau\notin \JH(\Inj_1 \theta)$. As $[V\colon \sigma]=1$, $\theta\neq \sigma$. Therefore, we can assume $\theta=\delta_j^{\epsilon'}(\sigma)\in \Delta^1(\sigma)$ for some $j\in \mathcal{J}$ and $\epsilon'\in \{\pm\}$, then we must have $|2\epsilon'\delta_{jk}-\epsilon\delta_{ik}|\leq 1$ for all k. Therefore, we must have $j=i$ and $\epsilon'=\epsilon$.\par
By \cref{description of soc(V/V1)}, $\soc(V/V^1)=\soc(V^2/V^1)\hookrightarrow\bigoplus_{\theta'\in \Delta^1(\sigma)} \theta'$. If $\delta_j^{\epsilon_j}(\sigma)\subset \soc(V/V^1)$, then we can form a quotient $\widetilde{V}$ of $V/V^1$ with socle $\delta_j^{\epsilon_j}(\sigma)$. Then $\soc(\widetilde{V}/\widetilde{V}^1)\subset\soc(V/V^2)\cong \theta$. Therefore, by \cref{lemma 3} $\theta'\in \Delta^1(\theta)$ or $\theta\cong \theta'$. The former is impossible, so $(j, \epsilon_j)=(i, \epsilon)$ and $\theta'\cong \theta$. As the theorem holds for $n=2$, by \cref{lemma on k-weight in Vk}, if $\tau\in \JH(V^2)$, then $\tau\in \JH(V^1)$; hence $[V^2/V^1\colon \tau]=0$. By \cite[Corollary~2.26]{hu2022mod}, $V^2$ is multiplicity free; hence $[V^2/V^1\colon \theta]=1$. Therefore, $$[V^2/V^1\colon \theta]=1> [V^2/V^1\colon \tau]=0.$$
On the other hand, applying \cite[Corollary~2.3]{hu2022mod} to $V/V^2$,
$$[V/V^2\colon \theta]\geq [V/V^2\colon \tau].$$
Therefore, we have
$$[V/V^1\colon \theta]> [V/V^1\colon \tau].$$
As $V/V^1$ has socle $\theta'\cong \theta$ and cosocle $\tau$, this contradicts \cite[Corollary~2.26]{hu2022mod}. Therefore, we conclude that $V/V^2=0$.
\end{proof}
\begin{lemma}
\label{soc(V/V^2) for n-3}
    Suppose $V$ as in \cref{Main theorem on generalization} with $m<n=3$, then $\JH(\soc(V/V^2))\subset \Delta^2(\sigma)$. 
\end{lemma}
\begin{proof}
By \cref{equation on V/Vn-1}, it suffices to show that there does not exist $\theta\in \JH(\soc(V/V^2))\cap \Delta^1(\sigma)$. Assume for the sake of contradiction that such $\theta$ exists. Let $\pi\colon V\to V/V^2$ be the projection map, then $V' \colonequals \pi^{-1}(\theta)$ is a subrepresentation of $V$ with $V'/V'^2\cong \theta$, and it suffices to prove such a representation does not exist. Without loss of generality, we assume $V=V'$. \par
Suppose $\theta=\delta_i^{\epsilon}(\sigma)\in \Delta(\sigma)$ for  some $\epsilon\in\{\pm\}$. Assume $\theta'\in \JH( \soc(V/V^1))$. Then we can find a quotient $\widetilde{V}$ of $V/V^1$ with socle $\theta'$. Therefore, $\soc(\widetilde{V}/\widetilde{V}^1)\subset \soc(V/V^2)\cong\theta$ and hence by \cref{lemma 3}, $\theta\subset \Delta^1(\theta')$ or $\theta\cong \theta'$. The former is impossible; therefore, we have $\theta\cong \theta'$.\par
As $\soc(V^2/V^1)\cong \theta$, $V^2$ contains a subrepresentation with cosocle $\theta$. By \cref{Cor on socle filtration} [2] (\textit{cf.} \cite[Corollary~2.28]{hu2022mod}), such a subrepresentation has socle filtration
$$\sigma\mbox{---}\mu_i^{\epsilon}(\sigma)\mbox{---}\theta.$$
Applying \cref{lemma on socles} with $T=\sigma$, we have $$\soc(V/\sigma)=\soc(V^2/\sigma)=\soc(V^1/\sigma)=\mu_i^{\epsilon}(\sigma).$$
Note that $V/V^2\cong \theta$; therefore,
$[V\colon \mu_i^{\epsilon}(\sigma)]=[V^2\colon \mu_i^{\epsilon}(\sigma)]= 1$ and $[V\colon \theta]=2$. However, as $\theta=\mu_i^{\epsilon}(\mu_i^{\epsilon}(\sigma))$, $V/\sigma$ contradicts \cref{simplest 1-weight}.
\end{proof}
\begin{proposition}\label{soc(V/Vn-1)}
Assume that \cref{Main theorem on generalization} holds for all pairs $<(n,m)$. Suppose $V$ as in \cref{Main theorem on generalization} with $m<n$. If $V/V^{n-1}\neq 0$, then all the Serre weights of $\soc(V/V^{n-1})$ are in $\Delta^{n-1}(\sigma)$ and $\soc(V/V^{n-1})$ is multiplicity free.
\end{proposition}
\begin{proof}

When $n=3$, 
By \cref{soc(V/V^2) for n-3}, $\JH(\soc(V/V^2))\subset  \Delta^2(\sigma)$.\par
For general $n>3$, suppose to the contrary that $\soc(V/V^{n-1})\not\subset \Delta^{n-1}(\sigma)$, then by \cref{lemma 3}, there exists $\theta \subset \soc(V/V^{n-1})$ such that $\theta\in \Delta^k(\sigma)$ for some $k<n-1$, in particular, $\theta$ is not an $n$-weight. Similar to the proof of \cref{reduction}, for each such $\theta$, we can find a subrepresentation $V'$ with $V'/V'^{n-1}=\theta$. It is enough to show that such a representation does not exist. Therefore, we reduce to the case where $V/V^{n-1}\cong\theta$ is irreducible and $\theta\in \Delta^k(\sigma)$ for some $k<n-1$. Let $\theta=\colon F(\mathfrak{t}_{\mu}(\xi))$.\par
By \cref{lemma on socles}, $\soc(V/V^{n-2})=\soc(V^{n-1}/V^{n-2})$; hence by the induction hypothesis $\soc(V/V^{n-2})\subset \Delta^{n-2}(\sigma)$. Pick any $\theta'=F(\mathfrak{t}_{\mu}(\xi'))\subset \soc(V/V^{n-2})$, then $\sum_j \frac{|\xi'_j|}{2}=n-2$. Therefore, $V/V^{n-2}$ contains a quotient $\widetilde{V}$ with $\soc(\widetilde{V})=\theta'$ and $\cosoc(\widetilde{V})=\theta$. As $\widetilde{V}^1\subset V^{n-1}/V^{n-2}$, $\soc(\widetilde{V}/\widetilde{V}^1)\subset \soc(V/V^{n-1})=\theta$, which is therefore an equality. Therefore, by \cref{lemma 3}, $\theta\cong \theta'$ or $\theta\in \Delta^1(\theta')$. We deduce that $k=n-2$ if $\theta\cong \theta'$ and $k= n-1$ or $n-3$ if $\theta\in \Delta^1(\theta')$. As we assume $\theta\notin \Delta^{n-1}(\sigma)$, $k=n-2$ or $n-3$. \par
Now we show that $\theta\hookrightarrow\soc(V/V^{\ell-1})$ for some $\ell\leq n-2$. If $k=n-2$, then $\theta'\cong\theta$ and we are done. If $k=n-3$, $\theta\in \Delta^1(\theta')$, and this is only possible if $\theta\leq\theta'$. The 
subrepresentation in $V^{n-1}$ with cosocle $\theta'$ (as $\theta'\subset \soc(V^{n-1}/V^{n-2})$) is isomorphic to $I(\sigma, \theta')$ by \cref{Main theorem on generalization}$ [(n-1, n-2)]$. By \cref{Main theorem on generalization}$ [(n-1, n-2)]$, $\theta\subset \soc(I(\sigma, \theta')^{n-2}/I(\sigma, \theta')^{n-3})$; hence $\theta\hookrightarrow \soc(V^{n-2}/V^{n-3})$. By \cref{lemma on socles}, $\soc(V/V^{n-3})=\soc(V^{n-2}/V^{n-3})$; therefore, this finishes the proof of the claim.\par
By \cref{Main theorem on generalization}$ [(n-1, n-2)]$, there is a unique subrepresentation $V'$ of $V^{n-1}$ with cosocle $\theta$. Pick a $\delta_i^{\epsilon_i}(\sigma)\subset \soc(V'/V'^1)\subset \soc(V^{n-1}/V^1)=\soc(V/V^1)$. We claim $\theta\not \cong \delta_i^{\epsilon_i}(\sigma)$.
Otherwise, we must have $n=4$. By the discussion above, $\theta'=\delta_i^{\epsilon}(\theta)\in \Delta^2(\sigma)$ for some $\epsilon\in \{\pm\}$. We consider the subrepresentation $V'$ with cosocle $\theta'$ in $V^3$. As $ \mu_i^\epsilon(\theta)-\theta\leq\delta_i^\epsilon(\theta)-\theta$, by applying \cref{Cor on socle filtration}$[2]$ to $V'$ and observing that $V'/V'^2\cong \theta'$, we see that $\mu_i^\epsilon(\theta)\in \JH(V'^2)\subset\JH(V^2)$. On the other hand, $V/V^2$, admits a quotient with socle $\theta'$, which is $(2n-5)$-generic, and cosocle $\theta$, which is isomorphic to $I(\theta',\theta)$ by \cref{Main theorem on generalization}$ [(n-2, 2)]$, call it $\widetilde{V}$. Since $ \mu_i^\epsilon(\theta)-\delta_i^\epsilon(\theta)\leq\theta-\delta_i^\epsilon(\theta)$, by applying \cref{Cor on socle filtration} [2] to $\widetilde{V}$ and observing that $\widetilde{V}/\widetilde{V}^1\cong \theta$, we deduce that $\widetilde{V}^1$ contains $\mu_i^\epsilon(\theta)$ as a subquotient. In particular, $\mu_i^\epsilon(\theta)\in \JH(V^3/V^2)$. Then 
\begin{equation*}
    \begin{split}
        [V^3\colon \mu_i^\epsilon(\theta)]&=[V^2\colon \mu_i^\epsilon(\theta)]+[V^3/V^2\colon \mu_i^\epsilon(\theta)]\\
        &=2.
    \end{split}
\end{equation*}
However, $V^3$ is multiplicity free by \cref{cor on multiplicities}$[3]$ (we assume $n=4$ here), which is a contradiction.\par
Therefore, $\delta_i^{\epsilon_i}(\sigma)\not \cong \theta$. Furthermore, $V^{n-1}$ is multiplicity free by \cref{cor on multiplicities}$ [n-1]$. Therefore $[V\colon \delta_i^{\epsilon_i}(\sigma)]=1$. Similar to the proof in \cref{soc(V/V^2) for n-3}, we have a (unique up to scalar) nonzero map $f\colon V\twoheadrightarrow V/V^1 \to \Inj_{n-1} \delta_i^{\epsilon_i}(\sigma)$. Then we claim that $[f(V)\colon \theta]=2$. Assume $[f(V)\colon \theta]\leq 1$, then $[\ker(f)\colon \theta]\geq 1$. As $f$ is nonzero, $\ker(f)\subset \Rad(f)=V^{n-1}$, which is $\m^{n-1}$-torsion. Then by \cref{Main theorem on generalization}$ [(n-1,k+1)]$, $\ker(f)$ contains a subrepresentation isomorphic to $I(\sigma, \theta)$. Moreover, by \cref{Cor on socle filtration}$ [k+1]$, as $\delta_i^\epsilon(\sigma)\leq \theta$, $\delta_i^{\epsilon_i}(\sigma)$ is a subquotient of $I(\sigma, \theta)\subset \ker(f)$. However, this contradicts $f$ being nonzero, as $[V\colon \delta_i^{\epsilon_i}(\sigma)]=1$. Therefore, $[f(V)\colon \theta]=2$.
On the other hand, $\soc(f(V))=\delta_i^{\epsilon_i}(\sigma)$, which is $(2n-3)$-generic, and $[f(V)\colon \delta_i^{\epsilon_i}(\sigma)]\leq[V\colon \delta_i^{\epsilon_i}(\sigma)]= [V^{n-1}\colon \delta_i^{\epsilon_i}(\sigma)]=1$. Therefore, applying \cref{cor on multiplicities}$ [n-1]$ to $f(V)$, we obtain $[f(V)\colon \theta]\leq [f(V)\colon \delta_i^{\epsilon_i}(\sigma)]=1$, which is a contradiction. This finishes the proof.\par
The statement on multiplicity free follows from the first assertion and \cref{lemma 3} that $k_n=1$.
\end{proof}

\begin{proposition}\label{theorem for m<n}
   Fix a pair $(n,m)$ with $m<n$. Assume that \cref{Main theorem on generalization} holds for all pairs $(n',m')<(n,m)$. Then the theorem holds for $(n,m)$.
\end{proposition}
\begin{proof}
 It is enough to show that $V/V^{n-1}=0$ and then by the induction hypothesis, i.e.\ , \cref{Main theorem on generalization}$ [(n-1,m)]$, we can conclude the result. Assume for the sake of contradiction that $V/V^{n-1}\neq 0$, then by \cref{reduction}, it suffices to disprove the case where $V/V^{n-1}\cong I(\theta, \tau)$. By \cref{soc(V/Vn-1)}, $\theta\in \Delta^{n-1}(\sigma)$. Let $\theta\eqqcolon F(\mathfrak{t}_\mu(\xi))$. As $\tau\in \JH(\Inj_1 \theta)$ and $\xi_j\in 2\Z$ for all $j$, we have $|\xi_j|>1\iff \omega_j\neq0$ and $\sgn(\xi_j)=\sgn(\omega_j)\eqqcolon\epsilon_j$ in this case.\par
Furthermore, as $\soc(V/V^{n-1})$ is an $n$-weight, we can apply
\cref{description of soc(V/V1)} and deduce that $$\soc(V/V^1)\hookrightarrow\bigoplus_{|\xi_i|>1}\delta_i^{\sgn(\xi_i)}(\sigma)\cong\bigoplus_{\omega_i\neq0}\delta_i^{\sgn(\omega_i)}(\sigma).$$
We will show that if $|\omega_i|=1$, then $\delta_i^{\epsilon_i}(\sigma)\not \hookrightarrow \soc(V/V^1)$. Assume for the sake of contradiction that there exists an $i$ with $|\omega_i|=1$ and $\delta_i^{\epsilon_i}(\sigma)\hookrightarrow \soc(V/V^1)$. Then $V/V^1$ admits a subquotient with socle $\delta_i^{\epsilon_i}(\sigma)$, which is $(2n-3)$-generic, and cosocle $\theta$, which is isomorphic to $I(\delta_i^{\epsilon_i}(\sigma), \theta)$ by \cref{Main theorem on generalization}$ [(n-1,n-1)]$. Then as $|\omega_i|=1$, $ \mu_i^{\epsilon_i}(\sigma)-\delta_i^{\epsilon_i}(\sigma)\leq \tau-\delta_i^{\epsilon_i}(\sigma)$. By \cref{Cor on socle filtration}$ [n-1]$, we deduce that $\mu_i^{\epsilon_i}(\sigma)\in \JH(I(\delta_i^{\epsilon_i}(\sigma), \theta))$. Since $I(\delta_i^{\epsilon_i}(\sigma), \theta)/(I(\delta_i^{\epsilon_i}(\sigma), \theta))^{n-2}=\theta$; therefore $\mu_i^{\epsilon_i}(\sigma)\in \JH((I(\delta_i^{\epsilon_i}(\sigma), \theta))^{n-2})\subset\JH(V^{n-1}/V^1)$. On the other hand, $V^2$ admits a subrepresentation with cosocle $\delta_i^{\epsilon_i}(\sigma)$, which is isomorphic to $I(\sigma, \delta_i^{\epsilon_i}(\sigma))$ by \cref{Main theorem on generalization} $[(2,2)]$. By \cref{Cor on socle filtration}$[2]$ (\textit{cf.} \cite[Corollary~2.28]{hu2022mod}), we deduce that $\mu_i^{\epsilon_i}(\sigma)\in \JH((I(\sigma, \delta_i^{\epsilon_i}(\sigma)))^1)\subset \JH(V^1)$. Therefore, $[V^{n-1}\colon \mu_i^{\epsilon_i}(\sigma)]=[V^{n-1}/V^1\colon \mu_i^{\epsilon_i}(\sigma)]+[V^1\colon \mu_i^{\epsilon_i}(\sigma)]\geq 2$. By \cref{cor on multiplicities}$[n-1]$, $V^{n-1}$ is multiplicity free, a contradiction. Therefore, $\soc(V/V^1)\hookrightarrow\bigoplus_{|\omega_i|>1}\delta_i^{\sgn(\omega_i)}(\sigma)$. As a result, we have an induced map $g\colon V/V^1\hookrightarrow \bigoplus_{|\omega_i|>1} \Inj_{n-1}\delta_i^{\epsilon_i}(\sigma)$.\par
 By \cref{soc(V/Vn-1)}, we know that then $\theta\in \Delta^{n-1}(\sigma)$. Since $n\geq 3$, by \cref{Inj1} $\delta_i^{\epsilon_i}(\sigma)\notin \JH(\Inj_1 \theta)\supset\JH(V/V^{n-1})$ for all $(i, \epsilon_i)$. By \cref{cor on multiplicities}$[n-1]$, $V^{n-1}$ is multiplicity free. Therefore, $[V\colon \delta_i^{\epsilon_i}]=1$. Therefore, the projection of the image of $g$ to each $\Inj_{n-1}\delta_i^{\epsilon_i}(\sigma)$ is $I(\delta_i^{\epsilon_i}(\sigma), \tau)$ or $0$, by \cref{Main theorem on generalization}$ [(n-1, m-1)]$, noting that $\delta_i^{\epsilon_i}(\sigma)$ is $(2n-3)$-generic. Therefore, $g$ factors through $ V/V^1\hookrightarrow \bigoplus_{|\omega_i|>1} I(\delta_i^{\epsilon_i}(\sigma), \tau)$. As $\sum_j \lfloor\frac{|\omega_j|}{2}\rfloor=m$, and $\epsilon_i=\sgn(\omega_i)$ for $\omega_i=0$, $\sum_j \lfloor\frac{|\omega_j-{\epsilon_i}2\delta_{ij}|}{2}\rfloor=m-1$. Therefore, each $I(\delta_i^{\epsilon_i}(\sigma), \tau)$ is $\m_{K_1}^{m-1}$-torsion, so is $V/V^1$. It follows that $V$ is $m$-torsion, and hence $V/V^{n-1}=0$, a contradiction.
 \end{proof}
It remains to prove by induction for the case where $m=n$.
From now on, we write $\epsilon_i$ for $\sgn (\omega_i)$ when $\omega_i\neq 0$ and $n_i$ for $\lfloor\frac{|\omega_i|}{2}\rfloor$. 
\begin{proposition}\label{exact soc(V/V^1) II}
Assume that \cref{Main theorem on generalization} holds for all pairs $(m',n')<(n,n)$. Suppose $V$ as in \cref{Main theorem on generalization} with $m=n$. Then
$$\soc(V/V^1)\cong\bigoplus_{|\omega_i|>1}\delta_i^{\sgn(\omega_i)}(\sigma).$$
\end{proposition}
\begin{proof}
As $n=m$, by \cref{n=m irred soc(V/Vn-1)}, we have $\soc(V/V^{n-1})=\widetilde{\tau}$, an $n$-weight. By \cref{description of soc(V/V1)}, we have $\soc(V/V^1)\hookrightarrow\oplus_{|\omega_i|>1}\delta_i^{\sgn(\omega_i)}(\sigma)$. We will now prove that we have an injection in the other direction.
Let $\pi\colon V\to V/V^{n-1}$ be the projection map, then $V' \colonequals \pi^{-1}(\widetilde{\tau})$ is a subrepresentation of $V$. As $V'/V'^1\subset V/V^1$, it suffices to show that $\delta_i^{\sgn(\omega_i)}(\sigma)\hookrightarrow\soc(V'/V'^1)$ for all $i$ with $|\omega_i|>1$. Therefore, we can assume without loss of generality that $V/V^{n-1}\cong\tau$.\par
If we have a unique $i$ with $\lfloor\frac{|\omega_i|}{2}\rfloor=n$, then as $\soc(V/V^1)\neq \varnothing$, we must have $\soc(V/V^1)=\delta_i^{\epsilon_i}(\sigma)$. Assume that there exist $i\neq j$, with $|\omega_i|, |\omega_j|>1$. Assume for the sake of contradiction that there exists a $i$ such that $|\omega_i|>1$, but $\delta_i^{\epsilon_i}(\sigma) \not\subset \soc(V/V^1)$.
When $n=3$, then $V/V^2\cong F(\mathfrak{t}_{\mu}( \epsilon_i 2\overline{\eta}_i+\epsilon_j 2\overline{\eta}_j))$ for some $i\neq j$.
By \cref{description of soc(V/V1)}, we know that $\soc(V/V^1)\hookrightarrow \delta_i^{\epsilon_i}(\sigma)\oplus\delta_j^{\epsilon_j}(\sigma)$. Assume for the sake of contradiction that $\soc(V/V^1)\cong\delta_i^{\epsilon_i}(\sigma)$ or $\delta_j^{\epsilon_j}(\sigma)$. Without loss of generality, assume $\soc(V/V^1)=\delta_i^{\epsilon_i}(\sigma)$. Then as $V/V^1$ is a $\m_{K_1}^2$-torsion representation with socle $\delta_i^{\epsilon_i}(\sigma)$, cosocle $\tau$, with $[V/V^1\colon \delta_i^{\epsilon_i}(\sigma)]=[V^2/V^1\colon \delta_i^{\epsilon_i}(\sigma)]=1$, as $V^2$ is multiplicity free by \cref{cor on multiplicities}$[2]$. Therefore, applying \cref{Main theorem on generalization}$[(2,2)]$ to $V/V^1$, we can conclude that $V/V^1\cong I(\delta_i^{\epsilon_i}(\sigma), \tau)$. In particular, $V/V^1$ has socle filtration $$\delta_i^{\epsilon_i}(\sigma) \mbox{---}F(\mathfrak{t}_{\mu}( \epsilon_i 2\overline{\eta}_i+\epsilon_j \overline{\eta}_j)) \mbox{---} F(\mathfrak{t}_{\mu}( \epsilon_i 2\overline{\eta}_i+\epsilon_j 2\overline{\eta}_j)).$$ Therefore, we deduce that $\cosoc(V^2)=\cosoc(V^2/V^1)\cong F(\mathfrak{t}_{\mu}( \epsilon_i 2\overline{\eta}_i+\epsilon_j \overline{\eta}_j))$. Therefore, by \cref{Main theorem on generalization}$ [(2,2)]$, we have $$V^2\cong I(\sigma, F(\mathfrak{t}_{\mu}( \epsilon_i 2\overline{\eta}_i+\epsilon_j \overline{\eta}_j))).$$
In particular, as $\epsilon_i\overline{\eta}_i+\epsilon_j\overline{\eta}_j\leq \epsilon_i 2\overline{\eta}_i+\epsilon_j\overline{\eta}_j$ for all $k$, by \cref{Cor on socle filtration}$[2]$ on $V^2$, we have $\mu_j^{\epsilon_j}(\sigma)=F(\mathfrak{t}_{\mu}( \epsilon_j\overline{\eta}_j))\in \JH(V)$. Therefore, we can find a quotient $\widetilde{V}$ of $V$ with socle $\mu_j^{\epsilon_j}(\sigma)$. As $\tau$ is a $2$-weight with respect to $\mu_j^{\epsilon_j}(\sigma)$, by \cref{theorem for m<n}, $\widetilde{V}\cong I(\mu_j^{\epsilon_j}(\sigma),\tau)$. As $ \epsilon_j2\overline{\eta}_j-\epsilon_j\overline{\eta}_j\leq \epsilon_i2\overline{\eta}_i+\epsilon_j2\overline{\eta}_j-\epsilon_j\overline{\eta}_j$ for all $k$, by \cref{Cor on socle filtration}$[2]$, $\delta_j^{\epsilon_j}(\sigma)\in \JH(I(\mu_j^{\epsilon_j}(\sigma),\tau))\subset \JH(V)$, a contradiction.\par
Now assume $n>3$. As $\soc(V/V^1)\neq 0$, There exists a $(j, \epsilon_j)$ with $j\neq i$ such that $\delta_j^{\epsilon_j}(\sigma) \subset \soc(V/V^1)$. We can find a quotient of $V/V^1$ with socle $\delta_j^{\epsilon_j}(\sigma)$, which is $(2n-3)$-generic, and we call it $W_{\delta_j}$.
Since $\epsilon_j 2\overline{\eta}_j, \epsilon_i 2\overline{\eta}_i\leq \omega$, $\epsilon_j 2\overline{\eta}_j+\epsilon_i 2\overline{\eta}_i\leq \omega$. Hence, by \cref{Cor on socle filtration}$[n-1]$, $F(\mathfrak{t}_{\mu}(\epsilon_j 2\overline{\eta}_j+\epsilon_i 2\overline{\eta}_i))$ is a subquotient of $I(\delta_j^{\epsilon_j}(\sigma), \tau)$. As $V/V^{n-1}\cong\tau$, $F(\mathfrak{t}_{\mu}(\epsilon_j 2\overline{\eta}_j+\epsilon_i 2\overline{\eta}_i))$ is a subquotient of $V^{n-1}$. Then $V^{n-1}$ admits a subrepresentation $V'$ with cosocle $F(\mathfrak{t}_{\mu}(\epsilon_j 2\overline{\eta}_j+\epsilon_i 2\overline{\eta}_i))$. By \cref{Main theorem on generalization}$[(n-1,3)]$, $V'\cong I(\sigma,F(\mathfrak{t}_{\mu}(\epsilon_j 2\overline{\eta}_j+\epsilon_i 2\overline{\eta}_i))) $. Again, as $\delta_j^{\epsilon_j}(\sigma)\leq F(\mathfrak{t}_{\mu}(\epsilon_j 2\overline{\eta}_j+\epsilon_i 2\overline{\eta}_i))$ \cref{Cor on socle filtration}$[3]$ implies that $\delta_j^{\epsilon_j}(\sigma)\subset \soc(V'/V'^1)\subset (V/V^1)$, a contradiction. 
\end{proof}
\begin{proposition}\label{JH factors when m=n} 
Assume that \cref{Main theorem on generalization} holds for all pairs $<(n,n)$. Suppose $V$ as in \cref{Main theorem on generalization} with $m=n$. Then the Jordan--H\"{o}lder factors of $V$ are exactly as described in \cref{Main theorem on generalization} up to multiplicity. In other words, $F(\mathfrak{t}_{\mu}(\omega'))\in \JH(V)\iff \omega'\leq \omega$.
 \end{proposition}
 \begin{proof}
By \cref{exact soc(V/V^1) II}, we have $\soc(V/V^1)\cong\bigoplus_{|\omega_i|>1}\delta_i^{\epsilon_i}(\sigma)$. For each $\delta_i^{\epsilon_i}(\sigma)\subset \soc(V/V^1)$, we consider the quotient $\widetilde{W}_i$ of $V/V^1$ with socle $\delta_i^{\epsilon_i}(\sigma)$, which is $(2n-3)$-generic. Moreover, by \cref{cor on multiplicities}$[n-1]$, $V^{n-1}$ is multiplicity free. Moreover, by \cref{soc(V/Vn-1)}, $\soc(V/V^{n-1})\cong \widetilde{\tau}$ and $\delta_i^{\epsilon_i}(\sigma)\notin \Inj_1\widetilde{\tau}$. Therefore, $[V\colon \delta_i^{\epsilon_i}(\sigma)]=1$ for all such $\delta_i^{\epsilon_i}(\sigma)$. Then, since $\widetilde{W}_i$ is $\m_{K_1}^{n-1}$-torsion, we can apply \cref{Main theorem on generalization}$ [(n-1, n-1)]$ to each $\widetilde{W}_i$ and show that $\widetilde{W}_i\cong I(\delta_i^{\epsilon_i}(\sigma),\tau)$. As $\widetilde{W}_i$ is $\m_{K_1}^{n-1}$-torsion, we can apply \cref{Cor on socle filtration} to $\widetilde{W}_i$, and deduce that 
\begin{equation}\label{equation for JH(V/V1)}
   \begin{aligned}   
   \bigcup_i\JH(\widetilde{W}_i)&=\{F(\mathfrak{t}_{\mu}(\omega'))\colon  \epsilon_i 2\overline{\eta}_i\leq \omega' \&\;\omega' \leq \omega\}\\
   &=\{F(\mathfrak{t}_{\mu}(\omega'))\colon  \omega'\leq \omega \text{ and there exists an } i \text{ s.t } |\omega'_i|>1\}.
   \end{aligned}
\end{equation}
In particular, $F(\mathfrak{t}_{\mu} (\epsilon_i 2\overline{\eta}_i))_+\in \JH(V)$ for all $i$ such that $|\omega_i|>1$. Fix one of such $i$. By \cref{Main theorem on generalization}$[(n,2)]$, $I(\sigma, F(\mathfrak{t}_{\mu} (\epsilon_i 2\overline{\eta}_i))_+)$ is $\mathfrak{m}_{K_1}^2$-torsion; therefore, it is a subrepresentation of $V^2$. Moreover, as $\sigma_+=F(\mathfrak{t}_\mu(\sum_{\omega_j\neq0} \epsilon_j\overline{\eta}_j))\leq F(\mathfrak{t}_{\mu} (\epsilon_i 2\overline{\eta}_i))_+$, we can also deduce that $\sigma_+\in \JH(V^1)$. From this we see that if $ \omega'\leq \omega$ and $|\omega'_j|\leq 1$, then $F(\mathfrak{t}_{\mu} (\omega'))\in \JH(V^1)$. Therefore, if $\omega'\leq \omega$, then $F(\mathfrak{t}_{\mu} (\omega'))\in \JH(V)$. \par
 Now we prove the converse that $F(\mathfrak{t}_{\mu} (\omega'))\in \JH(V)$ then $\omega'\leq \omega$. By the argument above, we have a map $f\colon V/V^1\to \bigoplus_{|\omega_i|>1} I(\delta_i^{\epsilon_i}(\sigma), \tau)$. As $\soc(V/V^1)\cong\bigoplus_{|\omega_i|>1}\delta_i^{\epsilon_i}(\sigma)$, and since $f$ is injective on the socles, it must be injective overall. Therefore, if $F(\mathfrak{t}_{\mu} (\omega'))\in \JH(V/V^1)$, by \cref{equation for JH(V/V1)}, $\omega'\leq \omega$.
 We claim that if $\sigma' \in \JH(V^1)\setminus \JH(I(\sigma, \sigma_+))$ and $\tau' \in \JH(V/V^1)$, $\Ext^1_{K/Z_1}(\tau', \sigma')=0$.
 Write $\tau' \colonequals F(\mathfrak{t}_{\mu}(\omega''))$ and $\sigma' \colonequals F(\mathfrak{t}_{\mu}(\omega'))$.
 If $\tau'\in \JH(V/V^2)$, then by \cref{description of JH(Injk)}, $\sum_j\left | \frac{\lfloor \omega''_j\rfloor}{2}\right |\geq 2$ and $\sum_j\left | \frac{\lfloor \omega'_j\rfloor}{2}\right |=0$. Therefore, there exists a $j$ with $|\omega''_j-\omega'_j|\geq 2$, or there exists $i\neq j $ with $\omega''_i\neq \omega'_i$ and $\omega''_i\neq \omega'_i$. Therefore, by \cite[Lemma~2.4.6]{BHHMS}, $\Ext^1_{K_1/Z_1}(\tau', \sigma')=0$. If $\tau'\in \JH(V^2/V^1)$, then we can apply \cite[Lemma~2.2.1]{hu2022mod}, noting that $\lambda_!(\sigma)\leq \sigma_+$ ($\lambda_!$ is defined in \cite{hu2022mod}), and deduce that if $\sigma'\not\in \JH(I(\sigma,\sigma_+))$, then $\Ext^1_{K_1/Z_1}(\tau', \sigma')=0$. This proves the claim.
 Therefore, by d\'evissage, if $\sigma'\notin \JH(I(\sigma,\sigma_+))$,
 $$\Ext^1_{K/Z_1}(V/V^1, \sigma')=0.$$
 Consequently, $\Hom_{K/Z_1}(V^1, \sigma')=\Hom_{K/Z_1}(V, \sigma')$. However, as $V$ has an $n$-weight as its cosocle; therefore $\Hom_{K/Z_1}(V, \sigma')=0$ for any $\sigma'$ as above and so $\Hom_{K/Z_1}(V^1, \sigma')=0$. We deduce that $\JH(V^1)=\JH(I(\sigma,\sigma_+))$. Therefore, we conclude the result.
 The second assertion follows from \cref{remark of theorem}.
\end{proof}
\begin{lemma} \label{Cor on Vk+1/Vk}
     Assume that \cref{Main theorem on generalization} holds for all pairs $(n',m')<(n,n)$. Suppose $V$ as in \cref{Main theorem on generalization} with $m=n$. Then for all $0\leq k<n-1$,
     $$V^{k+1}/V^{k}\cong \bigoplus_{\xi \in \prescript{0}{}\Omega_{k}^\tau} I(F(\mathfrak{t}_{\mu}(\xi)),F(\mathfrak{t}_{\mu}(\xi_+))).$$
\end{lemma}
\begin{proof}
Given $\theta\in \soc(V^{k+1}/V^{k})$ for some $0\leq k<n-1$. Then, as the theorem holds for all $m<n$, by \cref{lemma on k-weight in Vk}, $\theta$ is a $k+1$ weight. By \cref{lemma 3}, we deduce that $\theta\in \Delta^{k}(\sigma)$. By the remark in \cref{definition on relations} and \cref{JH factors when m=n}, we conclude that $\theta\in \prescript{0}{}\Omega_k^\tau$. By definition $\theta_+\leq \omega$; hence by \cref{JH factors when m=n}, we deduce that $\theta_+\in \JH(V)$. Moreover, by \cref{reformulation of I}, $\theta_+\in \JH(\Inj_1 \theta)$; hence $I(F(\mathfrak{t}_{\mu}(\xi)),F(\mathfrak{t}_{\mu}(\xi_+)))\hookrightarrow V^{k+1}/V^k$. Therefore, we have $\bigoplus_{\xi \in \prescript{0}{}\Omega_{k}^\tau} I(F(\mathfrak{t}_{\mu}(\xi)),F(\mathfrak{t}_{\mu}(\xi_+)))\hookrightarrow V^{k+1}/V^k$. By \cref{cor on multiplicities}$ [n-1]$, $V^{n-1}$ is multiplicity free, so is $V^{k+1}/V^k$ for all $0\leq k<n-1$. As $\bigoplus_{\xi \in \prescript{0}{}\Omega_{k}^\tau} I(F(\mathfrak{t}_{\mu}(\xi)),F(\mathfrak{t}_{\mu}(\xi_+)))$ and $V^{k+1}/V^k$ have the same Jordan--H\"{o}lder factors by \cref{JH factors when m=n} and both are multiplicity free, they are isomorphic. 
\end{proof}
\begin{proposition}\label{quotienting out 1st time}
Assume that \cref{Main theorem on generalization} holds for all pairs $<(n,n)$. Suppose $V$ as in \cref{Main theorem on generalization} with $m=n$. If $|\omega_i|>1$, then $\tau^{(i)}\in \JH(V^{n-1})$ (\textit{cf.} \cref{definition on relations}), and $I(\sigma,\tau^{(i)})$ is isomorphic to a proper subrepresentation of $V^{n-1}$. Moreover, $$\soc(V/I(\sigma,\tau^{(i)}))=\mu_i^{\epsilon_i}(\sigma).$$
\end{proposition}
\begin{proof}
Let $n_i=\frac{|\omega_i|}{2}\geq 1$. By definition, $\tau^{(i)}$ is a $(n-n_i)$-weight and $\tau^{(i)}\leq\tau$. By \cref{JH factors when m=n}, $V$ has a subrepresentation with cosocle $\tau^{(i)}$.
By \cref{Main theorem on generalization}$[(n-1, n-n_i)]$, this subrepresentation is isomorphic to $I(\sigma,\tau^{(i)})$ and is a subrepresentation of $V^{n-1}$.
As $|\omega_i|>1$, by \cref{Cor on Vk+1/Vk},
 $\mu_i^{\epsilon_i}(\sigma)\in \JH(V^1)$. However, applying \cref{Cor on socle filtration}$ [(n-n_\mathrm{i})]$ to $I(\sigma,\tau^{(i)})$, we deduce that for any $F(\mathfrak{t}_{\mu}(\omega'))\in \JH(I(\sigma, \tau^{(i)}))$, $\omega'_i=0$. Therefore, $\mu_i^{\epsilon_i}(\sigma)\notin \JH(I(\sigma, \tau^{(i)}))$ and hence $$I(\sigma, \tau^{(i)})\subsetneq V^{n-1}.$$ By \cref{n=m irred soc(V/Vn-1)}, $\soc(V/V^{n-1})\cong \widetilde{\tau}$ is an $n$-weight, $\widetilde{\tau}\not\subset\soc(V)$.
Furthermore, for any $F(\mathfrak{t}_{\mu}(\omega'))\in \JH(I(\sigma, \tau^{(i)}))$, we just showed that $\omega'_i=0$. Since $|\omega_i|>1$, by \cite[Lemma~2.4.6] {BHHMS}, we deduce that $$\Ext^1_{K_1/Z_1}(\widetilde{\tau}, F(\mathfrak{t}_{\mu}(\omega')))=0.$$
Therefore, all the assumptions of \cref{lemma on socles}(ii) are satisfied, and we deduce that $\soc(V/I(\sigma,\tau^{(i)}))=\soc(V^{n-1}/I(\sigma,\tau^{(i)}))$. 
Now, we claim that $$\soc(V^{n-1}/I(\sigma,\tau^{(i)}))=\mu_i^{\epsilon_i}(\sigma).$$
As $\mu_i^{\epsilon_i}(\sigma)\in \JH(V^1)$, we have a (unique up to scalar) nonzero map 
$$f\colon V^{n-1}\to \Inj_{n-1}\mu_i^{\epsilon_i}(\sigma).$$
The claim is equivalent to $\ker(f)\cong I(\sigma, \tau^{(i)})$.\par
First, I will show that $I(\sigma, \tau^{(i)})$ is a subrepresentation of $ \ker(f)$. It suffices to show $\tau^{(i)}\in \JH(\ker(f))$, since then $\ker(f)$ admits a subrepresentation with socle $\sigma$ and cosocle $\tau^{(i)}$. As $[\ker(f)\colon \sigma]\leq[V\colon \sigma]=1$, by \cref{Main theorem on generalization}$ [(n-1, n-n_i)]$, such a representation is isomorphic to $I(\sigma, \tau^{(i)})$. Assume for the sake of contradiction that $\tau^{(i)}\notin \JH(\ker(f))$, then $\tau^{(i)}\in \JH(\text{Im}(f))$. As $V^{n-1}$ is multiplicity free by \cref{Main theorem on generalization}$[n-1]$, $[\text{Im}(f)\colon \mu_i^{\epsilon_i}(\sigma)]=1$. Therefore, $\text{Im}(f)$ admits a subrepresentation with cosocle $\tau^{(i)}$, which is isomorphic to $I(\mu_i^{\epsilon_i}(\sigma), \tau^{(i)})$ by \cref{Main theorem on generalization}$ [(n-1, n-n_i)]$. Since $ 0-\epsilon_i\overline{\eta}_i\leq \omega^{(i)}-\epsilon_i\overline{\eta}_i$ for all $k$, from the theorem, we further deduce that $\sigma$ is a subquotient of $I(\mu_i^{\epsilon_i}(\sigma), \tau^{(i)})\subset\mathrm{Im}(f)$. However, as $V^{n-1}$ is multiplicity free, $[\ker(f)\colon \sigma]=0$. Therefore, $\ker(f)=0$ and $f$ is injective. However, this is a contradiction as $\soc(\text{Im}(f))=\mu_i^{\epsilon_i}(\sigma)\neq \sigma$. \par
Conversely, assume that we have some $F(\mathfrak{t}_{\mu}(\omega'))\in \JH(\ker(f)\setminus I(\sigma, \tau^{(i)}))\subset\JH(V^{n-1})$. Then by \cref{JH factors when m=n}, $\omega'\leq\omega$, if $\omega'_i\neq0$, then we must have $\omega'-\epsilon_i\overline{\eta}_i\leq \omega-\epsilon_i\overline{\eta}_i$. By \cref{Cor on socle filtration}$[n-1]$, we have $\mu_i^{\epsilon_i}(\sigma)\in \JH(I(\sigma, F(\mathfrak{t}_{\mu}(\omega'))))$. Therefore, $\mu_i^{\epsilon_i}(\sigma)\in \JH(\ker(f))$, which is a contradiction as $[V\colon \mu_i^{\epsilon_i}(\sigma)]=1$, and $f$ is nonzero. Therefore, $\omega'_i=0$ and $\omega'\leq \omega^{(i)}$. Hence, $F(\mathfrak{t}_{\mu}(\omega'))$ is a subquotient of $I(\sigma, \tau^{(i)})$ by \cref{Cor on socle filtration}$[n-n_i]$.
\end{proof}
\begin{proposition}\label{m=n multiplicity free} 
Assume that \cref{Main theorem on generalization} holds for all pairs $<(n,n)$. Suppose $V$ as in \cref{Main theorem on generalization} with $m=n$, then $V$ is multiplicity free. Moreover, for all $0\leq k\leq n-1$, $V^{k+1}/V^k$ is exactly as described in \cref{Main theorem on generalization}.
 \end{proposition}
 \begin{proof}
 By \cref{cor on multiplicities}$ [n-1]$, $V^{n-1}$ is multiplicity free. By \cref{soc(V/Vn-1)}, $\soc(V/V^{n-1})\cong \widetilde{\tau}$. If we show that $[V/V^{n-1}\colon \widetilde{\tau}]=1$, then by \cref{Main theorem on generalization}$ [(1,1)]$, $V/V^{n-1}\cong I(\widetilde{\tau}, \tau)$, which is multiplicity free. As the theorem holds for all $m<n$; therefore, by \cref{lemma on k-weight in Vk}, $V^{n-1}$ and $V/V^{n-1}$ do not share common Jordan--H\"{o}lder factors. Therefore, $V$ is multiplicity free. Moreover, as $V/V^{n-1}\cong I(\widetilde{\tau}, \tau) $, together with \cref{Cor on Vk+1/Vk}, we can conclude the second assertion. \par
Assume $|\omega_i|>1$ for some fixed $i$. By \cref{quotienting out 1st time}, $I(\sigma, \tau^{(i)})$ is a subrepresentation of $V$, and $V/I(\sigma, \tau^{(i)})\eqqcolon\widetilde{W}_i$ has socle $\mu_i^{\epsilon_i}(\sigma)=F(\mathfrak{t}_{\mu}(\epsilon_i\overline{\eta}_i))$, which is $(2n-3)$-generic. Then $\widetilde{\tau}= F(\mathfrak{t}_{\mu}(\sum_j2 \lfloor\frac{\omega_j}{2}\rfloor\overline{\eta}_j))$ is an $n-1$-weight with respect to $\mu_i^{\epsilon_i}(\sigma)$. Therefore, applying \cref{cor on multiplicities}$[n-1]$ to $\widetilde{W}_i$, we conclude that $[\widetilde{W}_i\colon \theta]\leq [\widetilde{W}_i\colon \mu_i^{\epsilon_i}(\sigma)]$. On the other hand, $\mu_i^{\epsilon_i}(\sigma)$ is a $1$-weight, by \cref{lemma on k-weight in Vk}, $[V/V^{n-1}\colon \mu_i^{\epsilon_i}(\sigma)]=0$. Furthermore, $V^{n-1}$ is multiplicity free, so $[V\colon \mu_i^{\epsilon_i}(\sigma)]=1$. Therefore, $[\widetilde{W}_i\colon \mu_i^{\epsilon_i}(\sigma)]=1$, and hence $[\widetilde{W}_i\colon \theta]\leq 1$. As $I(\sigma, \tau^{(i)})\subsetneq V^{n-1}$, $V/V^{n-1}$ is a quotient of $\widetilde{W}_i$, $[V/V^{n-1}\colon \theta]\leq [\widetilde{W}_i\colon \theta] \leq 1$. On the other hand, $\theta=\soc(V/V^{n-1})$. This finishes the proof.
\end{proof}

\begin{lemma}\label{quotienting out special case}
Assume that \cref{Main theorem on generalization} holds for all pairs $<(n, n)$. Suppose $V$ as in \cref{Main theorem on generalization} with $m=n$. Assume $\tau' \colonequals F(\mathfrak{t}_{\mu}(\omega'))\in \JH(V)$ with $0\leq \omega'_i< \widetilde{\omega}_i$ or $0\geq \omega'_i>\widetilde{\omega}_i$ and $\omega'_j=\omega_j$ for all $j\neq i$. Assume $|\omega_i-\omega'_i|<2(n-1) $. Then $I(\sigma, \tau')$ is a subrepresentation of $V$ and
    $$V/I(\sigma, \tau')\cong I(F(\mathfrak{t}_{\mu}((\omega'_i+\epsilon_i 1) \overline{\eta}_i), \tau).$$
\end{lemma}
\begin{proof}
Let $\ell=\sum_j\lfloor\frac{|\omega'_j|}{2}\rfloor<n-1$ and $\ell'=\lfloor\frac{|\omega_i-\omega'_i|}{2}\rfloor<n-1$. By \cref{m=n multiplicity free}, $V$ is multiplicity free. As $\tau'\leq \tau$, by \cref{JH factors when m=n}, $V$ admits a unique subrepresentation $W^i$ with cosocle $\tau'$. By \cref{Main theorem on generalization}$ [(n,\ell)]$, $W^i\cong I(\sigma, \tau')$. Similarly, $V$ admits a quotient $\widetilde{W}^i$ with socle $F(\mathfrak{t}_{\mu}((\omega'_i+\epsilon_i 1) \overline{\eta}_i))$ which is $(2n-2\ell-3)$-generic. By \cref{Main theorem on generalization}$ [(n,\ell')]$, $\widetilde{W}^i\cong I(F(\mathfrak{t}_{\mu}((\omega'_i+\epsilon_i 1) \overline{\eta}_i)), \tau)$.\par
As $\omega'\leq \omega$, $\sgn(\omega'_i)=\epsilon_i$ if $\omega'_i\neq0$. Given $F(\mathfrak{t}_{\mu}(\omega'))\in \JH(V)$, by applying \cref{Cor on socle filtration}$ [(n,\ell)]$ and respectively $ [(n, \ell')]$ to $I(\sigma, \tau')$ and respectively $I(F(\mathfrak{t}_{\mu}((\omega'_i+\epsilon(\omega_i)) \overline{\eta}_i)), \tau)$, we deduce that
\begin{equation}\label{quotienting out equation}
    F(\mathfrak{t}_{\mu}(\omega''))\in\begin{cases} 
        \JH(I(\sigma, \tau')) \text{ if and only if } |\omega''_i|\leq |\omega'_i|\\
    \JH(I(F(\mathfrak{t}_{\mu}((\omega'_i+\epsilon(\omega_i) 1) \overline{\eta}_i)), \tau)) \text{ if and only if } |\omega''_i|\geq|\omega'_i|+ 1.
    \end{cases} 
\end{equation}
Consider the map $f\colon V\to \Inj_n F(\mathfrak{t}_{\mu}((\omega'_i+\epsilon(\omega_i) 1) \overline{\eta}_i))$, as $V$ is multiplicity free, so is the image of $f$. Moreover, $\cosoc(f(V))\cong \tau$. By \cref{Main theorem on generalization}$ [(n,\ell')]$, we deduce that $f(V)\cong I(F(\mathfrak{t}_{\mu}((\omega'_i+\epsilon(\omega_i) 1) \overline{\eta}_i)), \tau)$. It follows from \cref{quotienting out equation} that $\tau'\in \JH(\ker(f))$. Therefore, $\ker(f)$ admits a subrepresentation with cosocle $\tau'$. By \cref{Main theorem on generalization}$ [(n,\ell)]$, such a representation is isomorphic to $I(\sigma,\tau') $. On the other hand, by \cref{quotienting out equation}, we have $\JH(\ker(f))=\JH(I(\sigma,\tau'))$. Therefore, $\ker(f)\cong I(\sigma,\tau')$.
\end{proof}
\begin{proposition}\label{uniqueness when m=n}
Assume \cref{Main theorem on generalization} holds for all pairs $<(n,n)$. Suppose $V$ as in \cref{Main theorem on generalization} with $m=n$, then $V$ is uniquely determined by $\sigma$ and $\tau$ up to multiplication by scalar.
\end{proposition}
\begin{proof}
Pick an $i$ such that $|\omega_i|>1$. By \cref{JH factors when m=n}, $F(\mathfrak{t}_\mu(\omega_k-\epsilon_i(\widetilde{\omega}_i-\omega_i+1)\overline{\eta}_i))$, $ F(\mathfrak{t}_{\mu}(\widetilde{\omega}_i\overline{\eta}_i))$ are in $\JH(V)$. Therefore, the assumption of \cref{quotienting out special case} holds and we deduce that $V$ admits $W^i\colon \cong I(\sigma,F(\mathfrak{t}_\mu(\omega_k-\epsilon_i(\widetilde{\omega}_i-\omega_i+1)\overline{\eta}_i)))$ as a subrepresentation, and $\widetilde{W}^i \colonequals  I(F(\mathfrak{t}_{\mu}(\widetilde{\omega}_i\overline{\eta}_i)), \tau)$ as a quotient; and $V/W^i=\widetilde{W}^i$.\par
Hence, $V$ represents a nontrivial class in $\Ext^1_{K/Z_1}(\widetilde{W}^i, W^i)$. Therefore, to prove the uniqueness of this representation, 
it is sufficient to show $$\dim_\F(\Ext^1_{K/Z_1}(\widetilde{W}^i, W^i))\leq 1.$$
First, we reduce to the case where $\widetilde{\omega}_i=\omega_i$. Assume $|\omega_i|=|\widetilde{\omega}_i|+1$. By the proof of \cref{quotienting out special case}, we see that if $ F(\mathfrak{t}_{\mu}(\beta'))\in \JH(V)$, $ F(\mathfrak{t}_{\mu}(\beta'))\in \JH(W^i)$ if and only if $|\beta'_i|<|\widetilde{\omega}_i|$.
$\widetilde{W}^i$ admits a quotient $\widetilde{W}'$ with socle $F(\mathfrak{t}_{\mu}(\omega_i\overline{\eta}_i))$, which is $(2n-1-|\omega_i|)$-generic.
As $|\omega_i|=|\widetilde{\omega}_i|+1$, we can apply \cref{Cor on socle filtration}$[n-|\frac{\widetilde{\omega}_i}{2}|]$ to $I(F(\mathfrak{t}_{\mu}(\omega_i\overline{\eta}_i)), \tau)$ and deduce that $F(\mathfrak{t}_{\mu}(\omega'))\in \JH(I(F(\mathfrak{t}_{\mu}(\omega_i\overline{\eta}_i)), \tau))$ only if $\omega'_i=\omega_i=\widetilde{\omega}_i+\epsilon_i1$. Therefore, if $\tau'\in \JH(I(F(\mathfrak{t}_{\mu}(\omega_i\overline{\eta}_i)), \tau))$ and $\sigma'\in \JH(W^i)$, then by \cite[Lemma~2.4.6]{BHHMS}, $\Ext^1_{K/Z_1}(\tau', \sigma')=0$. Therefore, by d\'evissage,
$$\Ext^1_{K/Z_1}(I(F(\mathfrak{t}_{\mu}(\omega_i\overline{\eta}_i)), \tau), W^i)=0.$$
Moreover, by applying \cref{quotienting out 1st time} to $\widetilde{W}^i\cong I( F(\mathfrak{t}_{\mu}(\widetilde{\omega}_i\overline{\eta}_i)), \tau)$ with $n=\sum_{j\neq i}\frac{|\omega_j|}{2}$ here and the multiplicity free condition, we have the following short exact sequence:
\begin{equation}\label{equation for uniqueness in general 1}
\begin{split}
    0\to I(F(\mathfrak{t}_{\mu}(\widetilde{\omega}_i\overline{\eta}_i)), F(\mathfrak{t}_\mu(\omega_k-\epsilon_i\overline{\eta}_j)))\to \widetilde{W}^i\to I(F(\mathfrak{t}_{\mu}(\omega_i\overline{\eta}_i), \tau) \to 0.
    \end{split}
\end{equation}
Applying the $\Hom_{K/Z_1}(\mbox{---},W^i)$ functor to \cref{equation for uniqueness in general 1} and observing that the Jordan--H\"{o}lder factors of $W^i$ and $\widetilde{W}^i$ are disjoint, we deduce that the first 3 terms vanish and we have 
\begin{equation*}
\begin{split}
    0\to \Ext^1_{k/Z_1}( I(F(\mathfrak{t}_{\mu}(\omega_i\overline{\eta}_i)), \tau), W^i)
    \to \Ext^1_{K/Z_1}(\widetilde{W}^i, W^i)\\
\to \Ext^1_{K/Z_1}(I(F(\mathfrak{t}_{\mu}(\widetilde{\omega}_i\overline{\eta}_i)), F(\mathfrak{t}_\mu(\omega_k-\epsilon_i\overline{\eta}_j))), W^i).
\end{split}
\end{equation*}
As $\Ext^1_{K/Z_1}(I(F(\mathfrak{t}_{\mu}(\omega_i\overline{\eta}_i)), \tau), W^i)=0$, 
\begin{equation*}
\dim_\F(\Ext^1_{K/Z_1}(\widetilde{W}^i, W^i))\leq\dim_\F(\Ext^1_{k/Z_1}(I(F(\mathfrak{t}_{\mu}(\widetilde{\omega}_i\overline{\eta}_i)), F(\mathfrak{t}_\mu(\omega_k-\epsilon_i\overline{\eta}_j))), W^i)).
\end{equation*}
Hence, it is sufficient to show that the latter is 1. Therefore, we can assume $\widetilde{\omega}_i=\omega_i$.\par
By \cref{quotienting out 1st time}, we know that $I(\sigma,\tau^{(i)})$ is a subrepresentation of $V^{n-1}$. Applying \cref{cor on multiplicities}$ [n-\frac{\omega_i}{2}]$ to $I(\sigma,\tau^{(i)})$, we deduce that $F(\mathfrak{t}_{\mu}(\omega'))\in \JH(I(\sigma,\tau^{(i)}))$ only if $\omega'_i=0$. Recall from the second paragraph that $F(\mathfrak{t}_\mu(\beta'))\in \JH(\widetilde{W}^{(i)})$ only if $\beta'_i=\omega_i$, in particular $|\beta_i|\geq2$. By \cite[Lemma~2.4.6]{BHHMS}, if $\sigma' \in I(\sigma,\tau^{(i)})$ and $\tau'\in \widetilde{W}^{(i)}$, we have
$\Ext^1_{K/Z_1}(\tau', \sigma')=0$. Therefore, by d\'evissage, 
$$\Ext^1_{K/Z_1}(\widetilde{W}^{(i)},I(\sigma,\tau^{(i)}))=0.$$
Again, by the result of \cref{quotienting out 1st time} for $V=W^{i}$, we have a short exact sequence:
\begin{equation}\label{equation on quotienting out}
    0\to I(\sigma, \tau^{(i)}) \to W^i \to I(\mu_i^{\epsilon_i}(\sigma),F(\mathfrak{t}_\mu(\omega-\epsilon_i\overline{\eta}_i)))\to 0.
\end{equation}
Therefore, we again apply the functor $\Hom_{K/Z_1}(\widetilde{W}^i, \mbox{---})$ to \cref{equation on quotienting out}, and observe that the Jordan--H\"{o}lder factors of $W^i$ and $\widetilde{W}^i$ are disjoint, and hence the first 3 terms vanish, and we have an exact sequence
\begin{equation*}
0\to \Ext^1_{K/Z_1}(\widetilde{W}^i, I(\sigma, \tau^{(i)}))\to \Ext^1_{K/Z_1}(\widetilde{W}^i,W^i)
\to \Ext^1_{K/Z_1}(\widetilde{W}^i, I(\mu_i^{\epsilon_i}(\sigma),F(\mathfrak{t}_\mu(\omega-\epsilon_i\overline{\eta}_i)))).
\end{equation*}
As $\Ext^1_{K/Z_1}(\widetilde{W}^i,I(\sigma,\tau^{(i)}))=0$, 
\begin{equation*}
        \dim_\F(\Ext^1_{K/Z_1}(\widetilde{W}^i,W^i))
        \leq \dim_\F(\Ext^1_{K/Z_1}(\widetilde{W}^i, I(\mu_i^{\epsilon_i}(\sigma),F(\mathfrak{t}_\mu(\omega-\epsilon_i\overline{\eta}_i))))).
\end{equation*}
As $\omega_i\in 2\Z$, $\sum_k\lfloor\frac{|\omega_k-\epsilon_i\delta_{ik}|}{2}\rfloor=n-1$, $\tau$ is an $n-1$ weight with respect to $\mu_i^{\epsilon_i}(\sigma)$ according to \cref{lemma on k-weight in Vk}. Therefore, by \cref{Main theorem on generalization}$[(n,n-1)]$, there is a unique $m_{K/Z_1}^{n-1}$-torsion representation with socle $\mu_i^{\epsilon_i}(\sigma)$, which is $(2n-2)$-generic, and $\tau$, and so
\[\dim_\F(\Ext^1_{K/Z_1}(\widetilde{W}^i, I(\mu_i^{\epsilon_i}(\sigma),F(\mathfrak{t}_\mu(\omega_k-\epsilon_i\overline{\eta}_i)))))=1. \qedhere\] 
\end{proof}
\end{proof}
\begin{corollary}\label{Cor on condition of being mn rep}
    If $V$ has cosocle $\tau$, $[V\colon \tau]=1$, $\tau$ is $(2n+1)$-generic and $\JH(V)\subset \JH(\Proj_n \tau)$, then $\Proj_n \tau \twoheadrightarrow V$, in particular $V$ is $\m_{K_1}^n$-torsion.
\end{corollary}
\begin{proof}
Assume that this does not hold, let $M$ be the counterexample with minimal length. Then by definition, $V/\soc(V)$ has cosocle $\tau$ and $\text{length}(V/\soc(V))<\text{length}(V)$; therefore, $V/\soc(V)$ is $\m_{K_1}^n$-torsion. We have
    $$0 \to \soc(V)\to V \to V/\soc(V) \to 0$$
    and $\soc(V)$ is semisimple and therefore $K_1$-invariant. Therefore, $V$ is $\m_{K_1}^{n+1}$-torsion. By the dual version of \cref{cor on multiplicities}$ [n+1]$, we deduce that $V$ is multiplicity free. We have $V\hookrightarrow \bigoplus_{\sigma\subset\soc(V)}\Inj\sigma$, which by \cref{Main theorem on generalization}, factors through $I(\sigma, \tau)$. Since $\sigma\in \JH(V)\subset\JH(\Proj_n \tau)$, $I(\sigma, \tau)$ is $\m_{K_1}^n$-torsion, so is V.
\end{proof} 
\section{Galois Deformation rings}\label{ch3: Galois deformation ring}
In this section, our goal is to use ``local models'' to compute the Galois deformation ring $R^{\lambda,\tau}_{\overline{\rho}}$ for sufficiently generic $\overline{\rho}$. On the one hand, we will compute the ring $R_{poly}$, which approximates $R^{\leq(\ell_j,0)_j,\tau}_{\overline{\rho}}$, up to an explicit tail. On the other hand, we will calculate the integer $k$ such that $p^k$ lies in a certain ideal. We can then apply Elkik's approximation theorem and compute $R^{\leq(\ell_j,0)_j,\tau}_{\overline{\rho}}$ with an explicit genericity condition. For any Hodge--Tate weights $\lambda$ and tame inertial type $\tau$ under some explicit genericity conditions, we will give an explicit description of $R^{\lambda,\tau}_{\overline{\rho}}$. In particular, we show such a ring to be a normal domain and a complete intersection ring. Moreover, we give a bijection between the irreducible components of the special fibre of $R^{\lambda,\tau}_{\overline{\rho}}$ and $W(\overline{\rho})\cap\JH(\overline{\sigma}(\lambda, \tau))$. We follow the approach and notation of \cite{BHHMS} with modification from \cite{Yitong} for the non-semisimple cases, which in turn, use the methods and results of \cite{potentiallycrystalline}, \cite{weightelimination}, \cite{LatticeforGL3} and \cite{localmodel}.
\subsection{Notations}\label{notation for Gal def ring}
An \emph{inertial type} is a representation of $I_K$ with open kernel which can be extended to $G_K$. Given $\lambda\in X_+^{*}(\underline{T})$, we can define the $\lambda$-admissible set relative to the Bruhat order, $\Adm^{\vee}(\mathfrak{t}_{\lambda})$, which will be described explicitly below. Given $\widetilde{w}\in \underline{\widetilde{W}}^{\vee}$, we can associate $\widetilde{w}^* \in \underline{\widetilde{W}}$, such that $((s\mathfrak{t}_\mu)^*)_j=\mathfrak{t}_{\mu_{f-1-j}}s^{-1}_{f-1-j}$. 
Let $\lambda=(\lambda'_{j,1}, \lambda_{j,2})$ with $\lambda_{j,1}\geq \lambda_{j,2}$. We write $\lambda'\leq\lambda$ if for all $j$, $\lambda_{j,1}+\lambda_{j,2}=\lambda'_{j,1}+\lambda'_{j,2}$ and $\lambda_{j,1}\geq\lambda'_{j,1}\geq \lambda_{j,2}$. It can be shown that 
\[\Adm^{\vee}(\mathfrak{t}_{\lambda})=\{
\widetilde{w}\colon \widetilde{w}_{f-1-j}=\mathfrak{t}_{\lambda_{j}'} \text{ or }\mathfrak{w}\mathfrak{t}_{\lambda_{j}'}, \text{ with }\lambda_{j}'\leq \lambda_{j} \text{ and } \widetilde{w}_{f-1-j}\neq\mathfrak{w}\mathfrak{t}_{(\lambda_{j,2},\lambda_{j,1})}\}.
\]
Given $(s, \mu)\in \underline{W}\times X^*(\underline{T})$, we can associate a tame inertial type $\tau(s, \mu)$. (For more details, see \cite[\S~2.3]{BHHMS})
We say that $\tau$ is \emph{$N$-generic} for some $N\in \Z_{\geq0}$ if $\tau\cong \tau(s, \mu+\eta)$ for some $\mu$ which is $N$-deep in $\underline{C_0}$.\par
We let $\overline{\rho}\colon G_K\to \GL_n(\F)$ be a Galois representation. We say $\overline{\rho}$ is \emph{$N$-generic} if $\overline{\rho}^{ss}|_{I_K}\cong \overline{\tau}(s, \mu) $ for some $s\in \underline{W}$ and $\mu-\eta\in X^*(\underline{T})$ which is $N$-deep in $\underline{C_0}$.
 Let $\rho$ be a two-dimensional de Rham representation of $G_K$ over $\overline{\Q}_p$, with regular Hodge--Tate weights. If there is a unique $\lambda=(\lambda_{\kappa,i})\in (\Z^2)^f$ such that for each $\sigma_i\colon K\hookrightarrow \overline{\Q}_p$, 
$$HT_i(\rho)=\{\lambda_{i,1}, \lambda_{i,2}\},$$
with $\lambda_{i,1}> \lambda_{i,2}$, then we say $\rho$ is \emph{regular} of Hodge type $\lambda$. For two Hodge--Tate weights $\lambda, \lambda'$, we write $\lambda'\geq\lambda$ if for all $j$, $\lambda_{j,1}+\lambda_{j,2}=\lambda'_{j,1}+\lambda'_{j,2}$ and $\lambda'_{j,1}\geq\lambda_{j,1}\geq 0$. We normalize Hodge-–Tate weights so that $\varepsilon$ has Hodge–Tate weight $1$ at every embedding.\par
Let $R_{\overline{\rho}}^\Box$ be the local $\mathcal{O}$-algebra parameterizing framed deformation of $\overline{\rho}$. 
For each dominant weight $\lambda\in X^*_{+}(\underline{T})$, let $R_{\overline{\rho}}^{\lambda,\tau}$ (resp. $R_{\overline{\rho}}^{\leq \lambda,\tau}$) be the maximal reduced, $\mathcal{O}$-flat quotient of $R^{\Box}_{\overline{\rho}}$, which parametrizes potentially crystalline lifts of $\rho$ with Hodge--Tate weights $\lambda$ (resp. $\lambda'\leq \lambda$) and tame inertial type $\tau$ (its existence follows from \cite{Kisin}). If $\overline{\sigma}=\otimes_{j=0}^{f-1}(\Sym^{a_j})^{\text{Fr}}\otimes\det^{\sum_{j=0}^{f-1}p^jb_j}$, where $0\leq a_j, b_j\leq p-1$, is a Serre weight, then we say $\rho:G_{F_v}\to \GL_2(E)$ has \emph{Hodge type} $\overline{\sigma}$ if $\HT(\rho)_{-j}=(a_j+b_j+1, b_j)$. We define $R^{\overline{\sigma}}_{\overline{\rho}}$ to be the reduced, $p$-torsion free quotient of $R^\square_{\overline{\rho}}$ corresponding to the crystalline deformation of Hodge type $\overline{\sigma}$.\par
Given a tame inertial type $\tau$, by the inertial local Langlands correspondence given in the appendix of \cite{Breuil-Mezard}, we have a finite-dimensional irreducible $E$-representation $\sigma(\tau)$ of $\GL_2(\mathcal{O}_K)$, which by extending scalar is defined over $E$. We write $\overline{\sigma}(\tau)$ for the mod $p$ semisimplification of $\sigma(\tau)$. Then the action of $\GL_2(\mathcal{O}_K)$ on $\overline{\sigma}(\tau)$ factors
through $\GL_2(k)$, so that the Jordan--H\"{o}lder factors of $\overline{\sigma}(\tau)$ are Serre weights. More precisely, the Jordan--H\"{o}lder factors of $\overline{\sigma}(\tau)$ are described as follows. \par
Recall that $\eta_J\colonequals \sum_{j\in J}\eta_j\in X^*(T)$, $\overline{\eta}_J$ is the image of $\eta_J$ in $\Lambda\colonequals X^*(T)/X^0(T)$, and $\Sigma\colonequals\{\overline{\eta}_j:J\subset \mathcal{J}\}$.
\begin{proposition}\label{prop on JH} \cite[Prop. 2.4.3]{BHHMS}
    Suppose $\tau=\tau(sw^{-1}, \mu-sw^{-1}(\nu))$ for some $(s, \mu), (w, \nu)\in \underline{W}\times X^*(\underline{T})$ such that $\mu-sw^{-1}(\nu)-\eta$ is $1$-deep in $\underline{C}_0$. If $\nu\in \eta+\Lambda_R$, then
\[\JH\left(\overline{\sigma}(\tau)\right)=\left\{F(\mathfrak{t}_{\mu-\eta}(sw^{-1}(\omega-\overline{\nu})))\colon \omega\in \Sigma\right\}.\]
\end{proposition}
We let $V(\lambda)$ be the irreducible algebraic representation with the highest weight $\lambda$.
We write $V(\lambda-\eta)\colonequals\otimes_{E,j} V(\lambda_j-\eta)^{(j)}$ and $\sigma(\lambda, \tau)\colonequals \sigma(\tau) \otimes_{E,j} V(\lambda_j-\eta)^{(j)}$ for the $\GL_2(\mathcal{O}_K)$ representation over $E$. We write $\overline{\sigma}(\lambda, \tau)$ for the mod $p$
semisimplification of $\sigma(\lambda,\tau)$.
\begin{lemma}\label{genericity condition}
    Assume $\tau=\tau(sw^{-1}, \mu-sw^{-1}(\nu))$ is $N$-generic where $N\geq 1$, then for all $\sigma\in \JH(\overline{\sigma}(\tau))$, $\sigma$ is $N-1$-generic (\textit{cf.} \cref{definition in general}). If $\sigma\in \JH(\overline{\sigma}(\lambda, \tau))$ where $\lambda\leq (\ell_j,0)$, then $\sigma$ is $N-\ell$-generic.
\end{lemma}
\begin{proof}
    As $\tau$ is $N$-generic, by \cref{prop on JH}, we let $\sigma=F(\mathfrak{t}_{\mu-\eta}(sw^{-1}(\omega-\overline{\nu})))$ for some $\omega\in \Sigma$. We have 
    \[N<\mu-sw^{-1}(\nu)-\eta<p-N.\]
    Since
    \[\langle \mathfrak{t}_{\mu-\eta}(sw^{-1}(\omega-\overline{\nu})),\alpha_j^\vee\rangle=\langle \mu-\eta+sw^{-1}(\overline{\nu}), \alpha_j^\vee\rangle\pm 1\text{ (for all $j$)},\]
    \[N-1<\langle \mathfrak{t}_{\mu-\eta}(sw^{-1}(\omega-\overline{\nu})),\alpha^\vee\rangle<p-N+1.\]
Therefore, $\sigma$ is $(N-1)$-generic. We can deduce the last assertion from the fact that
\begin{equation}\label{tensor alg rep}
    L(a, b)\otimes_\F L(m-1, n) = L(a+m-1, b+n)\oplus L(a+m-2, b+n+1)\oplus\dots\oplus L(a+n, b+m-1).\qedhere
\end{equation}
\end{proof}
\subsection{Kisin modules}
We will use without explanation the notation of \cite[\S~3]{BHHMS} and \cite[\S~3]{Yitong}. Let $R$ be a $p$-adically complete Noetherian local $\mathcal{O}$-algebra and $h\in \Z_{\geq 0}$. We denote the category of Kisin modules over $R$ of $E(u')$-height $\leq h$ and type $\tau$ by $Y^{[0,h],\tau}(R)$ as in \cite[Definition~3.1.3]{LatticeforGL3}. Given an eigenbasis $\beta$ for $\mathfrak{M}\in Y^{[0,h],\tau}(R)$ (\textit{cf.} \cite[Definition~3.1.6]{LatticeforGL3}), we have a matrix $A^{j}_{\mathfrak{M}, \beta}$. Given a dominant weight $\lambda$, we can then define a subcategory of Kisin modules of height $\leq \lambda$, denoted by $Y^{\leq \lambda,\tau}(R)\subset Y^{[0,h],\tau}(R)$ where $h=\max_j\{\lambda_{j,1}+\lambda_{j,2}\}$. Let $I(\F)$ be the Iwahori subgroup of $\GL_2(\F\llbracket v\rrbracket)$ consisting of the matrices which are upper triangular modulo $v$. Then $\overline{\mathfrak{M}}$ has shape $\widetilde{w}$ if $A^{j}_{\mathfrak{M}, \beta}\in I(\F)\widetilde{w}_jI(\F) $ for any choice of eigenbasis $\beta$, we have for each $0 \leq j\leq f-1$. In order to account for non-semisimple Galois representation, we need to use $\widetilde{w}$-gauge basis \cite[Definition~3.1]{Yitong} instead of Gauge basis defined in \cite[Definition~3.2.23]{weightelimination}. As noted in \cite{Yitong}, $\overline{\mathfrak{M}}$ has a unique shape, but it could have $\widetilde{w}$-gauge for many choices of $\widetilde{w}$.
\begin{example}\label{example}(\textit{cf.} \cite[Example~3.3]{Yitong})
    Let $\alpha,\beta\in \F^\times$ and $a\in \F$. We list the gauges and shapes of some matrices in $\GL_2\left(\F(\!(v)\!)\right)$ that will be considered in \cref{kisin module}.
   \[ \begin{array}{|c|c|c|c|}
    \hline
       &\text{Matrix}&\text{One choice of gauge}  & \text{Shape}  \\
       \hline
        m>n&\begin{pmatrix}
            \alpha v^m&0\\av^m &\beta v^n
        \end{pmatrix} & \mathfrak{t}_{(m,n)}\text{-gauge}&\mathfrak{t}_{(m,n)}\\
        \hline
            m\leq n&\begin{pmatrix}
            \alpha v^m&0\\ av^m&\beta v^n
        \end{pmatrix} & \mathfrak{t}_{(m,n)}\text{-gauge}&\begin{array}{c}\mathfrak{t}_{(m,n)} \text{ if }a=0\\\mathfrak{w}\mathfrak{t}_{(m,n)} \text{ if }a\neq0\end{array}\\
        \hline
        m>n&\begin{pmatrix}
            0&\beta v^n\\\alpha v^m &a v^n
        \end{pmatrix} & \mathfrak{w}\mathfrak{t}_{(m,n)}\text{-gauge}&\begin{array}{c}\mathfrak{w}\mathfrak{t}_{(m,n)} \text{ if }a=0\\\mathfrak{t}_{(m,n)} \text{ if }a\neq0\end{array}\\
        \hline
        m\leq n&\begin{pmatrix}
            0&\beta v^n\\\alpha v^m &a v^n
        \end{pmatrix} & \mathfrak{w}\mathfrak{t}_{(m,n)}\text{-gauge}&\mathfrak{w}\mathfrak{t}_{(m,n)}\\
        \hline
    \end{array}\]
\end{example}
Let $\mathcal{O}_{\mathcal{E},K}$ be the $p$-adic completion of $W(k)\llbracket v\rrbracket[1/v]$. We write $\Phi \Mod^{\et}(R)$ for the category of \'etale $\varphi$-modules over $\mathcal{O}_{\mathcal{E},K}\widehat{\otimes}R$. We have an equivalence of categories $\mathbb{V}_K^*\colon \Phi\Mod^{\et}(R)\xrightarrow[]{\sim}\Rep_{G_{K_\infty}}(R)$. By post-composing it with the functor $\epsilon_\tau\colon Y^{[0,h],\tau}(R)\to \Phi\Mod^{\et}(R)$ (\textit{cf.} \cite[\S~5.4]{localmodel}), we have a functor $T_{dd}^*\colon Y^{[0,h],\tau}(R) \to \Rep_{G_{K_\infty}}(R)$.

Fix $(\ell_1,\dots, \ell_f)\in \Z_+^f$, and let $\ell=\max\{\ell_j\}$. We fix a Galois representation
$\overline{\rho}\colon  G_K \to \GL_2(\F)$ and 
$(s, \mu)\in \underline{W}\times X^{*}(\underline{T}) $ such that $\overline{\rho}^{ss}|_{I_K}\cong \overline{\tau}(s, \mu)$ (here $ss$ denotes the semisimplification of $\overline{\rho}$), where 
\hfill\begin{enumerate}
\item $s_j = \mathfrak{w}$ precisely when $j = 0$ and $\overline{\rho}$ is irreducible;
\item $\mu-\eta$ is $N$-deep in $\underline{C_0}$.
\end{enumerate}
Twisting $\overline{\rho}$ with a power of $\omega_f$ if necessary, we further assume that 
$\mu_j=(r_j+1,0)\in \Z^2$ with $N<r_j+1 <p-N$ for all $j$ so that
\begin{equation}\label{condition on rho bar}
\overline{\rho}|_{I_k}\cong\begin{cases}
    \begin{pmatrix}
       \omega_{2f}^{\sum_{j=0}^{f-1}(r_j+1)p^j} &*\\
       0&1
    \end{pmatrix}&\text{if $\overline{\rho}$ is reducible;}\\
    \begin{pmatrix}
       \omega_{2f}^{\sum_{j=0}^{f-1}(r_j+1)p^j} &0\\
       0&\omega_{2f}^{\sum_{j=0}^{f-1}(r_j+1)p^{j+f}}
    \end{pmatrix} &\text{if $\overline{\rho}$ is irreducible.}
\end{cases}
\end{equation}
(The pair $(s, \mu)$ is not uniquely determined by $\overline{\rho}|_{I_K}$, however if $\overline{\rho}$ is $(N+1)$-generic, then by choosing an appropriate choice of $s$, $1.,2.$ always hold \cite[Propositions~2.2.15, 2.2.16]{weightelimination}.) We will assume $N\geq 4\ell$ in the following.\par
Let $\overline{\mathcal{M}}$ be the \'etale $\varphi$-module over $k(\!(v)\!)\otimes_{\F_p}\F$ such that $\mathbb{V}^*_K(\overline{\mathcal{M}})\cong \overline{\rho}|_{G_{K_\infty}}$. By \cite{multone}, we have a decomposition $\overline{\mathcal{M}}\cong \bigoplus_{i \in \mathcal{J}}\overline{\mathcal{M}}^{(j)}$ with $\overline{\mathcal{M}}^{(j)}=F(\!(v)\!)e_1^{(j)}\oplus F(\!(v)\!)e_2^{(j)}$ such that the matrices of the Frobenius map $\phi_{\overline{\mathcal{M}}}^{(j)}\colon \overline{\mathcal{M}}^{(j)}\to \overline{\mathcal{M}}^{(j+1)}$ with respect to the basis $\{(e_1^{(j)}, e_2^{(j)})\}$ have the following form
\begin{equation}\label{phiM}
    \Mat(\phi_{\overline{\mathcal{M}}}^{(f-1-j)})=\begin{cases}
    \begin{pmatrix}
    \alpha_{j}v^{r_j+1}&0\\
    \alpha_{j}\gamma_{f-1-j}v^{r_j+1}&\beta_{j}
    \end{pmatrix} &\text{if $\overline{\rho}$ is reducible;}\\ 
    \begin{pmatrix}
    \alpha_{j}v^{r_j+1}&0\\
    0&\beta_{j}
    \end{pmatrix} &\text{if $\overline{\rho}$ is irreducible and $j\neq0$;}\\
    \begin{pmatrix}
    0&-\beta_{j}\\
    \alpha_{j}v^{r_j+1}&0
    \end{pmatrix} &\text{if $\overline{\rho}$ is irreducible and $j=0$.}
\end{cases}
\end{equation}
where $\alpha_{j}, \beta_{j}\in \F^\times$ and $\gamma_j\in \F$.
When $\overline{\rho}$ is irreducible, we define $\gamma_j=0$ for all $j$. From now on, we fix a choice of $\alpha_j, \beta_j, \gamma_j$. 
\begin{proposition}\label{prop on W(r)} \cite[Proposition~3.2]{multone}, \cite[Proposition~3.5]{Diagram}
Given a Galois representation $\overline{\rho}\colon G_K\to \GL_2(\F)$, there is a set of Serre weights $W(\overline{\rho})$ associated to $\overline{\rho}$, described as follows:
\[W(\overline{\rho})=\left\{F(\mathfrak{t}_{\mu-\eta}(b_0, \dots, b_{f-1})\colon  b_j\in \{0, \sgn(s_j)\} \text{ if } \gamma_{f-1-j}=0 \text{ and } b_j=0 \text{ if } \gamma_{f-1-j}\neq0\right\}.\]
Alternatively, $W(\overline{\rho})=\{\sigma_J|J\subset J_{\overline{\rho}}\}$, and we can associate $\sigma\mapsto J_\sigma=\{j\colon b_j\neq0\}$. 
\end{proposition}
For $\widetilde{w}\in \Adm^{\vee}(\mathfrak{t}_{\lambda})$, where $\widetilde{w}^*=\mathfrak{t}_\nu w$ for some unique $(w,\nu)\in\underline{W}\times X^*(\underline{T})$, we associate type
\[\tau_{\widetilde{w}}\colonequals \tau(sw^{-1}, \mu-sw^{-1}(\nu))\]
with a lowest alcove representation $(s(\tau),\mu(\tau))=(sw^{-1}, \mu-sw^{-1}(\nu)-\eta)$, in particular $\tau_{\widetilde{w}}$ is $(N-\ell)$-generic, where $\ell=\max_j\{\lambda_{j,1}+\lambda_{j,2}\}$. Explicitly, $s(\tau)_j=w_j^{-1}$ except when $j=0$ and $\overline{\rho}$ is irreducible, in which case, we have $s(\tau)_0=\mathfrak{w}w_0^{-1}$. We have
\begin{equation}\label{mutau}
    \mu(\tau)_j+\eta_j=\begin{cases}
    (r_j+1-m, -n) &\text{ if } (\mathfrak{t}_{\nu_j}w_j,s_j)=(\mathfrak{t}_{(m,n)},1) \text{ or }\mathfrak({t}_{(m,n)}\mathfrak{w},\mathfrak{w});\\
    (r_j+1-n, -m) &\text{ if } (\mathfrak{t}_{\nu_j}w_j,s_j)=\mathfrak({t}_{(m,n)}\mathfrak{w},1) \text{ or }\mathfrak({t}_{(m,n)},\mathfrak{w}).
\end{cases}
\end{equation}
\begin{definition}
    Given a dominant weight $\lambda$, we define
    \[X(\overline{\rho}, \lambda)\colonequals \{\widetilde{w}\in \Adm^\vee(\mathfrak{t}_{\lambda})\colon \JH(\overline{\sigma}(\tau_{\widetilde{w}}, \lambda)\cap W(\overline{\rho})\neq \varnothing\}.\]
\end{definition}
\begin{lemma}\label{Xrholambda}
    If $\lambda=(\lambda_{j,1}, \lambda_{j,2})$ is a dominant weight, 
    \[X(\overline{\rho}, \lambda)=\{\widetilde{w}\in \Adm^\vee(\mathfrak{t}_{\lambda})\colon \widetilde{w}_{f-1-j}\neq \mathfrak{t}_{(\lambda_{j,2},\lambda_{j,1}) }\text{ if } \gamma_{f-1-j}\neq0\}.\]
\end{lemma}
\begin{proof}
    The proof is similar to \cite[\S~4]{Yitong}. By \cref{prop on JH} and \cref{tensor alg rep},
we can deduce that $F(\mathfrak{t}_{\mu-\eta}(b_0,\dots, b_{f-1}))\in \JH(\overline{\sigma}(\lambda,\tau_{\widetilde{w}}))$ if and only if $b_j\in$
\begin{equation} \label{intersection}
\{\sgn(s_j)\sgn(w_{j})(r-1)+\lambda_{j,1}+\lambda_{j,2}+1-2k\colon k\in \Z,\lambda_{j,2}< k\leq \lambda_{j,1}, r\in\{0,1\}\}.
\end{equation} 
Assume $\widetilde{w}_{f-1-j}=\omega_j\mathfrak{t}_{(a,b)}$. If $\lambda_{j,2}<a\leq \lambda_{j,1}$, by taking $r=1$ and $k=\frac{(1+\sgn(s_j)\sgn(w_j)1)}{2}(\lambda_{j,1}+\lambda_{j,2}+1)-\sgn(s_j)\sgn(w_j)a$, we deduce that $0$ is in \cref{intersection}. If $a=\lambda_{j,2}$, then $0$ is not in \cref{intersection}, but $\sgn(s_j)$ is if $\sgn(\omega_j)=-1$, $r=0$ and $k=\frac{1}{2}((\sgn(s_j)+1)(\lambda_{j,1})+(1-\sgn(s_j))\lambda_{j,2}))$. These are all the possibilities for $\widetilde{w}_{f-1-j}$.
\end{proof}
\begin{definition}\label{number of layers} 
Given $\widetilde{w}\in \Adm^{\vee}(\mathfrak{t}_{(\ell_j,0)_j})$. Let $Y(\widetilde{w}, \lambda)$ be the set of all regular Hodge--Tate weights $\lambda'$, such that $\lambda'\leq \lambda$ and $\widetilde{w}\in X(\overline{\rho},\lambda')$. And let $S(\widetilde{w}, \lambda)$ be the cardinality of $Y(\widetilde{w}, \lambda)$.
    We define
    \begin{equation*}
  S(\widetilde{w}_j)=\begin{cases} \min(m,n)+1 &\text{ if } \widetilde{w}_j=\mathfrak{t}_{(m, n)}\text{ with } m> n \text{ or } (\gamma_j= 0 \text{ and } m< n);\\
  \min(m,n) &\text{ if } \widetilde{w}_j=\mathfrak{t}_{(m, n)}\text{ with } (m< n \text{ and } \gamma_j\neq 0) \text{ or }m=n;\\
          \min(m,n+1) &\text{ if } \widetilde{w}_j=\mathfrak{m}\mathfrak{t}_{(m, n)}.
    \end{cases}\end{equation*}
\end{definition}
By \cref{Xrholambda}, we can deduce that
\begin{equation}\label{S(tau)}
S({\widetilde{w}},\lambda )=\prod_j \max\{0,S(\widetilde{w}_{f-1-j})-\lambda_{2,j}\}.
\end{equation}
We recall the results for the geometric Breuil--M\'ezard Conjecture for $\GL_2$. Following the notations from \cite{EG14}: Given a closed subscheme $\mathcal{Z}$ of $\mathcal{X}$, the cycles $Z(\mathcal{Z})\colonequals \sum_{\mathfrak{a}}e(\mathcal{Z},\mathfrak{a})$ are well-defined, where $e(\mathcal{Z},\mathfrak{a})$ is the Hilbert--Samuel multiplicity of $\mathcal{Z}$ at $\mathfrak{a}$ and the sum is over the points of $\mathcal{X}$ with the same dimension as $\mathcal{Z}$. 
\begin{lemma}\label{BreuilMezard}
    Fix $\tau$ which is $(2n+2)$-generic and $\lambda\leq(\ell_j,0)$ with $\ell_j\leq n$. Let $a_\sigma (\lambda,\tau)\in \{0,1\}$ such that $\overline{\sigma}^{ss}(\lambda,\tau)=\sum a_\sigma (\lambda,\tau) \sigma$ where the sum is over all Serre weights. Given a Serre weight $\sigma$, let $C_{\sigma}\colonequals Z(\Spec R_{\overline{\rho}}^{\sigma}/\varpi)$. We have the following equality of cycles: 
    \[Z(\Spec R_{\overline{\rho}}^{\lambda,\tau}/\varpi)=\sum_{\sigma\in \JH(\overline{\sigma}(\lambda,\tau))}a_\sigma (\lambda,\tau) C_\sigma.\]
\end{lemma}
\begin{proof}
    This can be deduced from \cite[Theorem~1.3.1 (1)]{BMandMS}, where we take $G$ in the theorem as $G=\Res_{K/\Q_p}\GL_2$. Specifically, there exists cycles $\mathcal{Z}(\sigma)$ such that
    \[[\mathcal{X}^{\lambda,\tau}_{2,\F_p}]=\sum_{\sigma\in \JH(\overline{\sigma}(\lambda,\tau))}a_\sigma (\lambda,\tau)\mathcal{Z}(\sigma),\]
    where $[\cdot]$ denotes the cycle class and $\mathcal{X}^{\lambda,\tau}_{2,\F}$ is the special fibre of the Emerton--Gee stack which parametrizes 2-dimensional potentially crystalline representations of $G_K$ with Hodge--Tate weights $\lambda$ and inertial type $\tau$. Moreover, as pointed out by Daniel Le, by comparing the result with \cite[Theorem~1.5]{CEG} (cf.\cite[\S~1.4]{BMandMS}), we know $C_{F(\lambda)}=[\mathcal{X}^{\lambda,\triv}|_{\F_p}]$.
    By the discussion of \cite[\S~8.3]{EmertonGeeStack}, we can recover the version of geometric Breuil--M\'ezard conjecture in terms of algebraic cycles, as in \cite{EG14}.
\end{proof}
\begin{lemma}\label{tame type criterion}
Assume $\lambda=(\lambda_j)_{j\in \mathcal{J}}\in X^*(T^\vee)^\mathcal{J}$ satisfies $\lambda_j\leq ({\ell_j,0})$. Given that $\tau$ is $(2\ell+2)$-generic. Then $R^{\lambda,\tau}_{\overline{\rho}}\neq0$ if and only if $\tau=\tau_{\widetilde{w}}$ with $\widetilde{w}\in X(\overline{\rho},\lambda)$.
    For each fixed $(2\ell+2)$-generic tame type $\tau_{\widetilde{w}}$, there are $ S({\widetilde{w}}, \lambda)$ regular Hodge--Tate weights $\lambda'\leq \lambda$, such that $\overline{\rho}$ admits a potentially crystalline lift $\rho$ of inertial type $\tau$ with $\HT_j(\rho) = \lambda'_j$ for all $j$.
\end{lemma}
\begin{proof}
The first statement follows from Breuil--M\'ezard conjecture in \cref{BreuilMezard}.
By \cite[Theorem~5.4.4]{EG14}, the mod $\varpi_E$-fibre of the deformation space $\overline{X}(\lambda, \tau_{\widetilde{w}})$ is the union of $\varpi_E$-fibres $\overline{X}(\overline{\sigma})$ where $\sigma$ runs over the Jordan--H\"{o}lder factors of $\overline{\sigma}(\lambda, \tau_{\widetilde{w}})$. Therefore, $R^{\lambda, \tau_{\widetilde{w}}}_{\overline{\rho}}\neq 0$ if and only if there exists $\overline{\sigma}\in \JH(\overline{\sigma}(\lambda, \tau))$ with $X^{\overline{\sigma}}_{\overline{\rho}}\neq0$. Moreover, $\overline{X}(\overline{\sigma})$ is nonempty if and only if $\overline{\sigma}\in W(\overline{\rho})$, by \cite[Theorem~A]{GLS} (\textit{cf.} \cite[Theorem~7.1.1]{EGS}). The last statement follows from the first one, together with \cref{S(tau)}.
\end{proof}
\begin{remark}\label{remark on genericity of tame type}
    We assume $\tau$ to be $(2\ell+2)$-generic a priori, however if $R^{\lambda,\tau}_{\overline{\rho}}\neq 0$, by \cref{tame type criterion} and \cref{Xrholambda}, we deduce that $\tau$ is actually $3\ell$-generic
\end{remark}
\begin{lemma}\label{kisin module}
    Let $\widetilde{w}\in X(\overline{\rho}, \lambda)$. Up to isomorphism there exists a unique Kisin module $\overline{\mathfrak{M}}\in Y^{\leq\lambda,\tau_{\widetilde{w}}}(\F)\subseteq Y^{\leq (\ell_j,0)_j, \tau_{\widetilde{w}}}$ such that $T^*_{dd}(\overline{\mathfrak{M}})\cong \overline{\rho}|_{G_{K_\infty}}$
\end{lemma}
\begin{proof}
    If $\overline{\rho}$ is irreducible, the proof goes exactly as in \cite[Lemma~4.1.1]{BHHMS}, provided that $\ell<\langle \mu(\tau)_j+\eta_j, \alpha_j^\vee\rangle<p-\ell-1$. If $\overline{\rho}$ is reducible, the proof goes exactly the same way as in \cite[Lemma~4.1]{Yitong}, we will simply comment on the changes required. Define a Kisin module $\overline{\mathfrak{M}}$ over $\F$ of type $\tau_{\widetilde{w}}$ by imposing the matrix of the partial Frobenius map to be $\overline{A}^{(f-1-j)}=\Mat(\phi_{\overline{\mathcal{M}}}^{(f-1-j)})v^{-(\mu(\tau)_j+\eta_j)}\dot{s}(\tau)_j$, where $\Mat(\phi_{\overline{\mathcal{M}}}^{(f-1-j)})$ and $\mu(\tau)_j+\eta_j$ are computed in \cref{phiM} and \cref{mutau} respectively. Therefore, we have
    \begin{equation}\label{mod A}
        \overline{A}^{(f-1-j)}=\begin{cases}
    \begin{pmatrix}
    \alpha_{j}v^{m}&0\\
    \alpha_{j}\gamma_{f-1-j}v^{m}&\beta_{j}v^n
    \end{pmatrix} & \text{if }  \widetilde{w}_{f-1-j}=\mathfrak{t}_{(m,n)};\\
     \begin{pmatrix}
    0&\alpha_{j}v^{n}\\
    \beta_{j}v^m&\alpha_{j}\gamma_{f-1-j}v^{n}
    \end{pmatrix} &\text{if }  \widetilde{w}_{f-1-j}=\mathfrak{w}\mathfrak{t}_{(m,n)}.
    \end{cases}
    \end{equation}
    In general, $\overline{\mathfrak{M}}$ has $\widetilde{w}$-gauge basis, but may not have shape $\widetilde{w}$. If $\widetilde{w}_{f-1-j}=\mathfrak{t}_{(m,n)}$ with $m\leq n$, then \cref{example} shows that $\overline{A}^{(f-1-j)}$ has shape contained $\Adm^{\vee}(\mathfrak{t}_{(n,m)})$ if and only if $\gamma_{f-1-j}=0$. By \cref{Xrholambda}, we deduce that $\widetilde{w}\in X(\overline{\rho}, \lambda)$ if and only if $\overline{A}^{(f-1-j)}$ has shape contained $\Adm^{\vee}(\mathfrak{t}_{\lambda_j})$ for all $j$. Therefore, $\overline{\mathfrak{M}}\in Y^{\leq\lambda,\tau_{\widetilde{w}}}(\F)\subseteq Y^{\leq(\ell_j,0))_j,\tau_{\widetilde{w}}}(\F)$ for all $\lambda\in X({\widetilde{w}}, (\ell_j,0)_j)$ by \cite[Proposition~5.4]{Kisinmod}. The rest of the proof goes through the same way given our genericity assumption.
\end{proof}
\subsection{Galois Deformation ring}
Given a complete Noetherian local $\mathcal{O}$-algebra $R$ with residue field $\F$ and $(\ell_1, \dots, \ell_f)$ a $f$-tuple of positive integers, we define $D^{\leq (\ell_j,0)_j,\tau}_{\overline{\mathfrak{M}}, \overline{\beta}}(R)$ to be the groupoid of the triplet $(\mathfrak{M}, \beta, \jmath)$, where $\mathfrak{M}\in Y^{\leq (\ell_j,0)_j,\tau_{\widetilde{w}}}(R)$, $\beta$ a $\widetilde{w}$-gauge basis of $\mathfrak{M}$ and $\jmath\colon\mathfrak{M}\otimes_R\F\xrightarrow{\sim}\overline{\mathfrak{M}}$ sending $\beta$ to $\overline{\beta}$. Then for any $(\mathfrak{M},\beta,\jmath)\in D^{\leq (\ell_j,0)_j,\tau}_{\overline{\mathfrak{M}}, \overline{\beta}}(R)$, we have a corresponding matrix $A^{(f-1-j)}$ such that $A^{(f-1-j)}\mod m_R\equiv \overline{A}^{(f-1-j)}$. We will compute $A^{(f-1-j)}$ using the monodromy and height conditions as in \cite[Proposition~4.2.1]{BHHMS}, \textit{cf.} \cite[Proposition~4.18]{potentiallycrystalline}.\par
If $\widetilde{w}_{f-1-j}=\mathfrak{t}_{(m,n)}$, i.e.\ , $\overline{A}^{(f-1-j)}=\begin{pmatrix}
\alpha_jv^m& 0\\
\alpha_j \gamma_{f-1-j}v^m&\beta_j v^{n}
\end{pmatrix}.$
Then 
$$A^{(f-1-j)}=\begin{pmatrix}
    \sum_{0\leq i\leq m} a_i^{(j)} (v+p)^{i} &\sum_{0\leq i\leq n-1} b_i^{(j)} (v+p)^{i}\\
    v(\sum_{0\leq i\leq m-1} c_i^{(j)} (v+p)^{i}) &\sum_{0\leq i\leq n} d_i^{(j)} (v+p)^{i}
\end{pmatrix}.$$
Given Shape $\widetilde{w}_{f-1-j}=\mathfrak{w}\mathfrak{t}_{(m,n)}$, i.e.\ , $\overline{A}^{(f-1-j)}=\begin{pmatrix}
0 & \beta_j v^{n}\\
\alpha_j v^m & \alpha_j\gamma_{f-1-j}v^n
\end{pmatrix}.$
Then 
$$A^{(f-1-j)}=\begin{pmatrix}
    \sum_{0\leq i\leq m-1} a_i^{(j)} (v+p)^{i} &\sum_{0\leq i\leq n} b_i^{(j)}  (v+p)^{i}\\
    v(\sum_{0\leq i\leq m-1} c_i^{(j)}  (v+p)^{i}) &\sum_{0\leq i\leq n} d_i^{(j)}  (v+p)^{i}
\end{pmatrix}.$$
We will suppress the superscript ${j}$ when it is clear from the context.
Recall that the finite height condition is given by 
\[\det A^{(f-1-j)}\in R^\times (v+p)^{\ell_{f-1-j}}\text{ for all } j.\]
For $0\leq k\leq \ell_{f-1-j}-1$ the $k$th height condition is given by
$$H(k)=\sum_{i+j=k}(a_id_j+pb_jc_i)-\sum_{i+j=k-1}b_jc_i.$$ 
Since $\overline{\rho}$ is $N=4\ell$-generic, if $R^{\lambda,\tau}_{\overline{\rho}}\neq 0$, by \cref{remark on genericity of tame type}, $\tau$ is $3\ell$-generic. We can then apply \cite[Proposition~3.19]{BHHMS} with $h=\ell$ and obtain the monodromy condition given as follows: 
\begin{equation}\label{monodromy condition}
    \left(\frac{d}{dv}\right)^{t}\bigg\vert_{v=-p}\left\{\left[v\frac{d}{dv}A^{(f-1-j)}-A^{(f-1-j)}\begin{pmatrix}
    \mathfrak{a}& 0\\ 0&0
\end{pmatrix} \right](v+p)^{h}(A^{(f-1-j)})^{-1}\right\}+O(p^{N-h-t}) \end{equation}
for all $0\leq t\leq h-2$, $0\leq j\leq f-1$.
As $\det A^{(f-1-j)}\in R^\times (v+p)^{\ell_{f-1-j}}$, \cref{monodromy condition} is $0$ for $0\leq t\leq h-\ell_{f-1-j}$. Therefore, using the Leibniz rule, we can reduce it to the following equation:
\[\left(\frac{d}{dv}\right)^{t}\bigg\vert_{v=-p}\left\{\left[v\frac{d}{dv}A^{(f-1-j)}-A^{(f-1-j)}\begin{pmatrix}
    \mathfrak{a}& 0\\ 0&0
\end{pmatrix} \right](v+p)^{\ell_{f-1-j}}(A^{(f-1-j)})^{-1}\right\}+O(p^{N-\ell_{f-1-j}-t}) \]
for all $0\leq t\leq \ell_{f-1-j}-1$, $0\leq j\leq f-1$.
Here $O(p^{N-\ell_{f-1-j}-t})\eqqcolon O(p^{u_j-t})$ is a specific but inexplicit element of $p^{N-\ell_{f-1-j}-t}M_2(R)$ and \begin{equation}\label{value of a}
    \mathfrak{a}\equiv-\langle(ws^{-1}(\mu)-\nu)_j,\alpha^{\vee}_j\rangle \;\;(\text{mod } p).
\end{equation}
For $0\leq k\leq \ell_{f-1-j}-2$, we label the entry $(s,t)$ of the $k$th monodromy as $A(k,s,t)$. Then we have the following.
\begin{align*}
    A(k,1,1)=&k!\biggl\{\sum_{i+j=k+1}-p(ia_id_j-jb_jc_i) +\sum_{i+j=k} \bigl[(i-\mathfrak{a})a_id_j+2pjb_jc_i\bigr]+\sum_{i+j=k-1}jb_jc_i\biggr \}+O(p^{u_j-k});\\
    A(k,1,2)=&k!\biggl\{\sum_{i+j=k}(\mathfrak{a}+j-i)a_ib_j+p\sum_{i+j=k+1}(i-j)a_ib_j\biggr \}+O(p^{u_j-k});\\
A(k,2,1)=&k!\biggl\{\sum_{i+j=k+1}p^2(i-j)c_id_j+\sum_{i+j=k}p(\mathfrak{a}+2j-2i+1)c_id_j+\sum_{i+j=k-1}-(\mathfrak{a}-i+j-1)c_id_j\biggr \}\\
&+O(p^{u_j-k});\\
A(k,2,2)=&k!\biggl\{\sum_{i+j=k+1}-p(pic_ib_j+ja_id_j)+\sum_{i+j=k}\bigl[ja_id_j-p(\mathfrak{a}-2i-1)c_ib_j\bigr]\\
&+\sum_{i+j=k-1}(\mathfrak{a}-i-1)c_ib_j\biggr \}+O(p^{u_j-k}).
\end{align*}
Let $M(-1,s,t)=0$ and $\widetilde{A}(k,i,j)=A(k,i,j)-O(p^{u_j-k})$. Define $M_k(s,t)$ for $0\leq k\leq \ell_{f-1-j}-2$, $1\leq s,t\leq 2$ recursively as follows:
\begin{align*}
&M_k(1,1)=\left(\frac{\widetilde{A}(k,1,1)}{k!}+\mathfrak{a} H_k+M(k-1,1,1)\right)/p;
M_k(1,2)=\frac{\widetilde{A}(k,1,2)}{k!};\\
&M_k(2,1)=\left(\frac{\widetilde{A}(k,2,1))}{k!}+M(k-1,2,1)\right)/p;
M_k(2,2)=\left(\frac{\widetilde{A}(k,2,2)}{k!}+M(k-1,2,2)\right)/p.
\end{align*}
Then for $0\leq k\leq \ell_{f-1-j}-2$, we have
\begin{equation}\label{monodromy equations}
\begin{split}
M_k(1,1)&=\sum_{i+j=k}(\mathfrak{a} +j)c_ib_j-\sum_{i+j=k+1}ia_id_j+jpc_ib_j;\\
M_k(1,2)&=\sum_{i+j=k}(\mathfrak{a} +j-i)a_ib_j+p\sum_{i+j=k+1}(i-j)a_ib_j;\\
M_k(2,1)&=\sum_{i+j=k}(\mathfrak{a}-i+j-1)c_id_j+p\sum_{i+j=k+1}(i-j)c_id_j;\\
M_k(2,2)&=\sum_{i+j=k}(i+1-\mathfrak{a})b_jc_i-\sum_{i+j=k+1}ja_id_j+ipc_ib_j.
\end{split}
\end{equation}
\begin{definition}
Let $R=\widehat{\otimes}_j R^{(j)}$ where $R^{(j)}$ is defined in \cref{Table 1} and \cref{Table 2}.
Let $I^{(j), \leq(\ell_{f-1-j},0)}$ be the ideal of $R$ generated by the equations given by the height conditions $H(k)$, $0\leq k\leq\ell_{f-1-j}-1$.
And we let $R^{\leq (\ell_{f-1-j},0)_j, \tau}_{\overline{\mathfrak{M}},\overline{\beta}}$ be the maximal reduced $p$-flat quotient of $\widehat{\otimes}R^{(j)}/I^{(j), \leq(\ell_{f-1-j},0)}$.\par
Let $I^{(j), \nabla}$ be the ideal generated by the monodromy condition $A(k,s,t)$ for $0\leq k \leq \ell_{f-1-j}-2$, $1\leq s,t\leq 2$. Let $R^{\leq (\ell_{f-1-j},0)_j,\tau,\nabla}_{\overline{\mathfrak{M}}, \overline{\beta}}$ be the maximal, reduced $\mathcal{O}$-flat quotient of the ring $R/\sum_{j}(I^{(j),\leq(\ell_{f-1-j},0)}+I^{(j),\nabla})$.\par
We further define $R_{\overline{\rho}}^{\leq(\ell_j,0)_j, \tau, \reg}$ as the quotient of $R_{\overline{\rho}}^{\leq(\ell_j,0)_j, \tau}$ such that each component is of maximal dimension, i.e.\ , $R_{\overline{\rho}}^{\leq(\ell_j,0)_j, \tau, \reg}=R_{\overline{\rho}}^{\leq(\ell,0)_j, \tau}/(\cap_i \mathfrak{p}_i)$, where the intersection is over $\mathfrak{p}_i$, such that $R_{\overline{\rho}}^{\leq(\ell_j,0)_j, \tau}/\mathfrak{p}_i$ is of maximal dimension. We define $R^{\leq (\ell_j,0)_j,\tau,\nabla, \reg}_{\overline{\mathfrak{M}}, \overline{\beta}}$ as the quotient such that every component is of the same maximal dimension analogously.
\end{definition}
By \cite[Theorem 3.3.8]{Kisin}, $R_{\overline{\rho}}^{\leq(\ell_j,0)_j, \tau, \reg}$ corresponds to the quotient of $R_{\overline{\rho}}^{\leq(\ell_j,0)_j, \tau}$ consisting of only regular Hodge--Tate weights, since the component with irregular Hodge--Tate weights has a positive codimension (\textit{cf.} \cref{dim drops}). \par
As in \cite[5]{potentiallycrystalline}, we have an isomorphism
\[R_{\overline{\rho}}^{\leq(\ell_j,0)_j, \tau, \reg}\llbracket X_1, \dots, X_{2f}\rrbracket\cong R^{\leq (\ell_j,0)_j,\tau,\nabla, \reg}_{\overline{\mathfrak{M}}, \overline{\beta}}\llbracket Y_1, \dots, Y_4\rrbracket.\]
We will compute the generators of $I_\infty^{\reg}\colonequals \ker(R\twoheadrightarrow R^{\leq (\ell_j,0)_j,\tau,\nabla, \reg}_{\overline{\mathfrak{M}}, \overline{\beta}})$
and show that $I^{\reg}_{\infty}=\sum_j I^{(j),\reg}$ where $I^{(j),\reg}$ is given in \cref{Table 1} and \cref{Table 2}. In general, $I^{(j), \reg}$ is not an ideal of $R^{(j)}$, as $O(p^{u_j-k})$ is an element of $M_2(R)$ rather than $M_2(R^{(j)})$. \par
Since $M_k(2,2)+M_k(1,1)=-(k+1)H(k+1)$, $(I^{(j),\leq(\ell_{f-1-j},0)}+I^{(j),\nabla}, p^{N-2\ell_{f-1-j}+1})$ is generated by $H(0)$ and $M_k(s,t)$ for $0\leq k\leq \ell_{f-1-j}-2$, $1\leq s,t \leq 2$. 
We will find the solutions to equations arising from $H(0)$ and $M_k(s,t)$ for $0\leq k\leq \ell_{f-1-j}-2$, $1\leq s,t \leq 2$. If $\rho$ is of Hodge--Tate weight $(m,n)$, then $\rho\otimes \epsilon^k$ is of Hodge--Tate weight $(m+k,n+k)$. On the representation side, this corresponds to twisting $\sigma(\tau)$ by $(N_{k/\F_p} \circ \det)^k$. Assume $\widetilde{w}_{f-1-j}=\mathfrak{t}_{(m,n)}$ (respectively $\mathfrak{w}\mathfrak{t}_{(m,n)}$), we let $\widetilde{w}_{f-1-j}+(k,k)=\mathfrak{t}_{(m+k,n+k)}$ (respectively $\mathfrak{w}\mathfrak{t}_{(m+k,n+k)}$). Hence, $\tau_{\widetilde{w}}\otimes \epsilon^k=\tau_{\widetilde{w}+(k,k)}$. Moreover, we have \[R_{\overline{\rho}}^{(m,n),\tau_{\widetilde{w}}}\hookrightarrow R_{\overline{\rho}\otimes\omega^k}^{(m+k,n+k)\tau_{\widetilde{w}+(k,k)}}.\] Therefore, in order to compute the monodromy conditions arising from the Galois deformation space $R^{\leq(\ell_j,0)_j, \tau_{\widetilde{w}}}$, we can instead consider the monodromy conditions from $R^{\leq(\ell_j+2k,0)_j, \tau_{\widetilde{w}+(k,k)}}$, which has more variables. If we relabel the solutions to the height and monodromy equations for $\widetilde{w}_{f-1-j}$ as $a_{k}=\bm{a}_{-m+k}$, $b_{k}=\bm{b}_{ -n+1+k}$, $c_{k}=\bm{c}_{-m+1+k}$, $d_{k}=\bm{d}_{-n+k}$, we expect them to be the same as the solutions $\{\bm{a'}_k, \bm{b'}_k, \bm{c'}_k, \bm{d'}_k\}$ (relabelled analogously) for $\widetilde{w}_{f-1-j}+(k,k)$ when both are well-defined. (We will use the superscript $*$ to indicate that it is a unit.) We will show this is the case below. Moreover, by cancelling the extra variables introduced, we can compute the Galois deformation space of higher Hodge--Tate weights.\par 
\textbf{Assume $\widetilde{w}_{f-1-j}=\mathfrak{t}_{(m,n)}$ with $m\geq n$ or $\mathfrak{w}\mathfrak{t}_{(m,n)}$ with $m> n$}. Denote $M^{a}_k(s,t)$ the monodromy equations for $\rho\otimes\epsilon^a$.
By induction, we can relate $M^{m-n}_{3m-n-1-k}(s,t)$ with $M^0_{m+n-1-k}(s,t)$ as follows. We use the notation $i+j=^*(k-1)/k$ to denote $i+j=k-1$ if $\widetilde{w}=\mathfrak{t}_{(m,n)}$ and $i+j=k$ if $\widetilde{w}=\mathfrak{w}\mathfrak{t}_{(m,n)}$. Similarly, we use $j\geq^* n/(n+1)$ to denote $j\geq n$ if $\widetilde{w}=\mathfrak{t}_{(m,n)}$ and $j\geq n+1$ if $\widetilde{w}=\mathfrak{w}\mathfrak{t}_{(m,n)}$.\par
\begin{equation}\label{raising weight for monodromy eq}
{\allowdisplaybreaks
\begin{aligned}
&M^{m-n}_{3m-n-1-k}(1,1)=M^0_{m+n-1-k}(1,1)+\sum_{\substack{i+j=^*k-1/k,\\ i\geq m\text{ or}\\j\geq^* n/n+1}}(\mathfrak{a} +m-1-j)\bm{c}_{-i}\bm{b}_{-j}\\
&-\sum_{\substack{i+j=^*k/k-1,\\ i\geq^* m+1/m\text{ or}\\j\geq n+1}}(2m-n-i)\bm{a}_{-i}\bm{d}_{-j}+\sum_{\substack{i+j=^*k-2/k-1,\\ i\geq m\text{ or}\\j\geq^* n/n+1}}(m-1-j)p\bm{c}_{-i}\bm{b}_{-j},\\
&M^{m-n}_{3m-n-1-k}(1,2)=M^0_{m+n-1-k}(1,2)+\\
&\sum_{\substack{i+j=k,\\ i\geq^* m+1/m\text{ or}\\j\geq^* n/n+1}}(\mathfrak{a} +n-m-1-j+i)\bm{a}_{-i}\bm{b}_{-j}-\sum_{\substack{i+j=k-1,\\ i\geq^* m+1/m\text{ or}\\j\geq^* n/n+1}}p(m-n+1-i+j)\bm{a}_{-i}\bm{b}_{-j},\\
&M^{m-n}_{3m-n-1-k}(2,1)=M^0_{m+n-1-k}(2,1)+\\
&\sum_{\substack{i+j=k,\\ i\geq m\text{ or}\\j\geq n+1}}(\mathfrak{a}-m+n+i-j)\bm{c}_{-i}\bm{d}_{-j}+\sum_{\substack{i+j=k-1,\\ i\geq m\text{ or}\\j\geq n+1}}p(n-m-1-i+j)\bm{c}_{-i}\bm{d}_{-j},\\
&M^{m-n}_{3m-n-1-k}(2,2)=M^0_{m+n-1-k}(2,2)+\\
&\sum_{\substack{i+j=^*k-1/k,\\ i\geq m\text{ or}\\j\geq^* n/n+1}}(m-i-\mathfrak{a})\bm{c}_{-i}\bm{b}_{-j}-\sum_{\substack{i+j=^*k/k-1,\\ i\geq^* m+1/m\text{ or}\\j\geq n+1}}(m-j)\bm{a}_{-i}\bm{d}_{-j}+\sum_{\substack{i+j=^*k-2/k-1,\\ i\geq m\text{ or}\\j\geq^* n/n+1}}(m-1-i)p\bm{c}_{-i}\bm{b}_{-j}.
\end{aligned}}
\end{equation}
The inequalities for $i$ and $j$ come from the fact that they are introduced as extra variables. Let $I^{(j),\mathrm{extra}}$ be the ideal of $R^{(j)}$ generated by the extra terms on the right-hand side of \eqref{raising weight for monodromy eq}. If we can find $\bm{a}_{-k},\bm{b}_{-k},\bm{c}_{-k}, \bm{d}_{-k}$ for $0\leq k\leq m+n-1$ and $1\leq i,j\leq2$ from $M^{m-n}_{3m-n-1-k}(s,t)$ where $0\leq k\leq m+n-1, \leq s,t\leq 2$, then they are also a solution to $M^0_{m+n-1-k}(s,t)$ for $0\leq k\leq m+n-1, \leq s,t\leq 2$ up to modulo by the ideal $I^{(j),\mathrm{extra}}$. As $j\geq0$ and $i+j=k, k-1$ or $k-2$, we must have $i\leq k$.\par
Assume $\widetilde{w}_{f-1-j}=\mathfrak{t}_{(m,n)}$, note that $\overline{\bm{a}_0^*}=\alpha_j$, $\overline{\bm{d}_0^*}=\beta_j$ and $\overline{\bm{b}_0}=\gamma_{f-1-j}\alpha_j$ when $n\geq 1$. In particular, if $\gamma_{f-1-j}\neq0$, then $\bm{b}_0\neq0$. Also, we have $\ell_{f-1-j}=m+n$.\par
\begin{theorem}\label{conjectured solution}
Assume $\widetilde{w}_{f-1-j}=\mathfrak{t}_{(m,n)}$,\[(I^{(j),\leq(\ell_{f-1-j},0)}+I^{(j),\nabla}, p^{N-2\ell_{f-1-j}+1})= (I^{(j)}_{poly}, p^{N-2\ell_{f-1-j}+1}),\] where $I^{(j)}_{poly}$ is given by row 5 of \cref{Table 1} without the $O(p^{k_j})$ term.
\end{theorem}
\begin{proof}
We first deal with the case where with $m\geq n$. The value of $\mathfrak{a}$ follows from \cref{value of a} and \cref{mutau}. As $p>2l$, $\mathfrak{a}\pm k\not\equiv 0\mod p$, so $\mathfrak{a}\pm k$ is a unit for all $0\leq k\leq \ell_{f-1-j}$. Moreover, the monodromy equations are given by $M^k(j,s,t)$ up to modulo $p^{N-2\ell_{f-1-j}+1}$.
We first prove that the solution to $M^{m-n}_{3m-n-1-k}(s,t)$ for $0\leq k\leq m+n-1$, $\leq s,t\leq 2$ is $\bm{a}_{-k},\bm{b}_{-k}, \bm{c}_{-k}, \bm{d}_{-k}$ as given by the $I^{(j),\reg}$ row without the $O(p^{k_j})$ tail in \cref{Table 1}, up to modulo $I^{(j),\mathrm{extra}}$. For $i=0$, it is just the definition. We proceed by induction that given we have verified $\bm{a}_{-k}$, $\bm{d}_{-k}$, $\bm{b}_{-k}$ and $\bm{c}_{-k}$ for $k\leq i-1$, we can deduce $\bm{a}_{-i}$ and $\bm{d}_{-i}$ from $M^{m-n}_{3m-n-1-i}(1,1)$ and $M^{m-n}_{3m-n-1-i}(2,2)$. From the combinations of the indices, we deduce that $\bm{a}_{-i}$ and $\bm{d}_{-i}$ are the only indeterminate. From $M^{m-n}_{3m-n-1-i}(1,1)$, we have
\begin{align*}
    &(2m-n-i)\bm{a}_{-i}\bm{d}_0^*+(2m-n)\bm{a}_0^*\bm{d}_{-i}\\
    =&\sum_{1\leq j\leq i}(\mathfrak{a}+m-j)\bm{c}_{-i+j}\bm{b}_{-j+1}-\sum_{1 \leq j\leq i-1}[\bigl(2m-n-i+j)\bm{a}_{-i+j}\bm{d}_{-j}+(m-j)p\bm{c}_{-i+j+1}\bm{b}_{-j+1}\bigr].
\end{align*}
And from $M^{m-n}_{3m-n-i}(2,2)$, we have
\begin{align*}
    m\bm{a}_{-i}\bm{d}_0^*+(m-i)\bm{a}_0^*\bm{d}_{-i}=&\sum_{1\leq j \leq i}(2m-n-i+j-\mathfrak{a})\bm{c}_{-i+j}\bm{b}_{-j+1}\\
    &-\sum_{1\leq j \leq i-1}\bigl[(m-j)\bm{a}_{-i+j}\bm{d}_{-j}+(2m-n-i+j)p\bm{c}_{-i+j+1}\bm{b}_{-j+1}\bigr].
\end{align*}
Hence,
\begin{equation}\label {equation for a}
\begin{split}
    \bm{a}_{-i}=&\dfrac{-1}{i\bm{d}_0^*}\biggl[\sum_{1 \leq j \leq i}(\mathfrak{a}-m+n-j)\bm{c}_{-i+j}\bm{b}_{1-j}+\sum_{1 \leq j \leq i-1}(m-n+j)p\bm{c}_{1-i+j}\bm{b}_{1-j}+(i-j)\bm{a}_{-i+j}\bm{d}_{-j}\biggr],
\end{split}
\end{equation}
\begin{equation*}\label {equation for d}
    \bm{d}_{-i}=\dfrac{1}{i\bm{a}_0^*}\biggl[\sum_{1\leq j\leq i}(\mathfrak{a}+i-j-m+n)\bm{c}_{-i+j}\bm{b}_{1-j}-\sum_{1 \leq j \leq i-1}(i-j-m+n) p\bm{c}_{1-i+j}\bm{b}_{1-j}+j\bm{a}_{-i+j}\bm{d}_{-j}\biggr].
\end{equation*}
By a simple calculation, we can show the following:
\begin{lemma}\label{relations 1}
Assume $\bm{a}_{-i+j},\bm{b}_{1-j}, \bm{c}_{1-i+j},\bm{d}_{-j}$ for $0<i,j,i-j$ are as given in the $I^{(j),\reg}$ row without the $O(p^{k_j})$ tail in \cref{Table 1}, then
$$i\bm{a}_{-i}\bm{d}_0^*=-(\mathfrak{a}-m+n-1)\bm{c}_{1-i}\bm{b}_0,$$
$$i\bm{a}_0^*\bm{d}_{-i}=(\mathfrak{a}-m+n)\bm{c}_0\bm{b}_{1-i},$$
$$\bm{a}_{-i+j}\bm{d}_{-j}=\dfrac{-(\mathfrak{a}-m+n)(\mathfrak{a}-m+n-1)\bm{b}_0\bm{c}_0\bm{b}_{1-j}\bm{c}_{1-i+j}}{\bm{a}_0^*\bm{d}_0^*(i-j)j}.$$
\end{lemma}
Therefore, the right-hand side of \cref{equation for a} is
\begin{align*}
    &\dfrac{-(\mathfrak{a}-m+n-1)\bm{c}_{1-i}\bm{b}_0}{i\bm{d}_0^*}+\\
    &\sum_{1 \leq j \leq i-1}\bigl[(\mathfrak{a}-m+n-1-j)\bm{c}_{-i+j+1}\bm{b}_{-j}+(i-j)\bm{a}_{-i+j}\bm{d}_{-j}+(m-n+j)p\bm{c}_{1-i+j}\bm{b}_{1-j}\bigr].
\end{align*}
From the expression for $\bm{b}_{-j}$ and \cref{relations 1}, we have the terms in the summand cancelling each other out, and $\bm{a}_{-i}$ is indeed as in the $I^{(j),\reg}$ row without the $O(p^{k_j})$ tail in \cref{Table 1}. The proof for $\bm{d}_{-i}$ is analogous. \par
We now show that given the solutions $\bm{a}_{-k}$, $\bm{d}_{-k}$ for $k\leq i$ and $\bm{b}_{-k}$, $\bm{c}_{-k}$ for $k\leq i-1$, we can deduce $\bm{b}_{-i}$ and $\bm{c}_{-i}$ from $M^{m-n}_{3m-n-1-i}(1,2)$ and $M^{m-n}_{3m-n-1-i}(2,1)$, respectively.
\begin{equation}\label{equation for b}
\begin{split}
    \bm{b}_{-i}=&\frac{-1}{(\mathfrak{a}-m+n-1-i)\bm{a}_0^*}\\
    &\biggl\{\sum_{0\leq j\leq i-1}\bigl[(\mathfrak{a}-m+n-1+i-2j)\bm{a}_{-i+j}\bm{b}_{-j}+
    p(m-n+2-i+2j)\bm{a}_{-i+1+j}\bm{b}_{-j}\bigr]\biggr\};
\end{split}
\end{equation}
\begin{equation*}\label{equation for c}
\begin{split}
    \bm{c}_{-i}=&\frac{-1}{(\mathfrak{a}-m+n+i)\bm{d}_0^*}\\
    &\biggl\{\sum_{0\leq j\leq i-1} \bigl[(\mathfrak{a}-m+n-i+2j)\bm{c}_{-j}\bm{d}_{-i+j}+
    p(m-n-2+i-2j)\bm{c}_{-j}\bm{d}_{1-i+j}\bigr]\biggr\}.
    \end{split}
\end{equation*}
\begin{lemma} \label{relations for bc}
  Assume $\bm{a}_{-i+j},\bm{b}_{-j}, \bm{c}_{-j}, \bm{d}_{-i+j}$ for $0\leq j\leq i$ are as given in the $I^{(j),\reg}$ row without the $O(p^{k_j})$ tail in \cref{Table 1}. Let $$T_j=(\mathfrak{a}-m+n-1+i-2j)\bm{a}_{-i+j}\bm{b}_{-j}+p(m-n+2-i+2j)\bm{a}_{+1-i+j}\bm{b}_{-j},$$ $$R_j=\frac{-j}{i}(\mathfrak{a}-m+n-1-j)\bm{a}_{-i+j}\bm{b}_{-j}.$$ Then we have for $0\leq j \leq i-1$, 
$T_{j}+R_j=R_{j+1}$. 
Similarly, let $$T'_j=(\mathfrak{a}-m+n-i+2j)\bm{c}_{-j}\bm{d}_{-i+j}+p(m-n-2+i-2j)\bm{c}_{-j}\bm{d}_{1-i+j},$$ $$R'_j=\frac{-j}{i}(\mathfrak{a}-m+n+j)\bm{c}_{-j}\bm{d}_{-i+j}.$$ Then we have that $T'_{j}+R'_j=R'_{j+1}$ for $0\leq j \leq i-1$.
\end{lemma}
By \cref{relations for bc} and the fact that $R_0=0$, the right-hand side of \cref{equation for b} is 
\[\frac{-1}{(\mathfrak{a}-m+n-1-i)\bm{a}_0^*}\sum_{0\leq j \leq i}T_j=\frac{-1}{(\mathfrak{a}-m+n-1-i)\bm{a}_0^*}R_{i},\]
which is precisely the conjectured $\bm{b}_{-i}$ in the $I^{(j),\reg}$ row without the $O(p^{k_j})$ tail in \cref{Table 1}, again by \cref{relations for bc}. The proof for $\bm{c}_{1-n-i}$ goes exactly the same way, using $T'_j$, $R'_j$ instead of $T_j$, $R_j$. Therefore, we finish the induction step and prove that the conjectured solution to $M^{m-n}_{3m-n-1-k}(1,1)$ for $0\leq k\leq m+n-1$ for all $i,j$ are $\{\bm{a}_{-k}, \bm{b}_{-k},\bm{c}_{-k}, \bm{d}_{-k}\}_{0\leq k\leq m+n-1}$.\par
Now by \eqref{raising weight for monodromy eq}, we know that the solutions, which is given by the $I^{(j),\reg}$ row without the $O(p^{k_j})$ tail in \cref{Table 1}, are also solutions to the monodromy equations $M_k(i,j)$ modulo $I^{\mathrm{extra}}$. 
We claim that the last equation in $I^{(j),\reg}$ when $m>n$ (resp. equations in $I^{(j)}$ when $m=n$) in \cref{Table 1} without the $O(p^{k_j})$ tail, generates $I^{\mathrm{extra}}$. On the one hand, by \eqref{raising weight for monodromy eq}
\[M^{m-n}_{3m-2n-1}(1,2)-(\mathfrak{a}-m-1)\bm{a}_0^*\bm{b}_{-n}=M^0_{m-1}(1,2),\] and $\bm{a}_0^*$ is a unit, from the formula for $\bm{b}_{-n}$ we deduce that the term with the last equation in $I^{(j),\reg}$ when $m>n$ (resp. first term in $I^{(j)}$ when $m=n$), without the $O(p^{k_j})$ tail, in \cref{Table 1} is contained in $I^{\mathrm{extra}}$. Furthermore, for $m=n$, by \eqref{raising weight for monodromy eq}, $M^{m-n}_{3m-2n-1}(2,1)-(\mathfrak{a}+m)\bm{d}_0^*\bm{c}_{-n}=M^0_{m-1}(1,2)$, we deduce that the terms with the last equation in $I^{(j)}$ without the $O(p^{k_j})$ tail is contained in $I^{\mathrm{extra}}$ in this case. On the other hand, all the terms that generate $I^{\mathrm{extra}}$ are divisible by $\bm{a}_{-j}$ where $j\geq m+1$, $\bm{b}_{-j}$ where $j\geq n$, $\bm{c}_{-j}$ where $j\geq m$ or $\bm{d}_{-j}$ where $j\geq n+1$. These are all computed to be according to the $I^{(j),\reg}$ row without the $O(p^{k_j})$ tail in \cref{Table 1}, hence they are all divisible by
\begin{equation}\label{eq. for comp 3}
\prod_{m\geq j\geq 1}\left(\dfrac{(\mathfrak{a}-m+n)(\mathfrak{a}-m+n-1)\bm{b}_0\bm{c}_0}{\bm{a}_0^*\bm{d}_0^*}-(m-n-j)jp\right).
\end{equation}
Furthermore, all $\bm{a}_{-j}, \bm{d}_{-j}$ with $j>0$ and all $\bm{b}_{-j}$ are divisible by $\bm{b}_0$; and all $\bm{a}_{-j}, \bm{d}_{-j}$ with $j>0$, and all $\bm{c}_{-j}$ are divisible by $\bm{c}_0$. Therefore, all the generators of $I^{\mathrm{extra}}$ are divisible by the last equation in $I^{(j),\reg}$ when $m>n$ (resp. equations in $I^{(j)}$ when $m=n$)in \cref{Table 1} without the $O(p^{k_j})$ tail. Therefore, $(I^{(j), \nabla},p^{N-2\ell_{f-1-j}+1})$ is generated by the terms in $I^{(j), \reg}$ in \cref{Table 1} without the $O(p^{k_j})$ tail, except the last equation in $I^{(j),\reg}$ is replaced by two equations in $I^{(j)}$ if $m=n$.\par
 Since $-(k+1)H(k+1)=M_k(2,2)+M_k(1,1)$ for $0\leq k\leq m+n-1$. To finish the proof, we substitute the conjectured solution in the equation $H(0)=\bm{a}_{-m}\bm{d}_{-n}+p\bm{b}_{-n+1}\bm{c}_{-m+1}$, and a direct calculation shows that it is divisible by the last equation in $I^{(j),\reg}$ when $m>n$ (resp. equations in $I^{(j)}$ when $m=n$), without the $O(p^{k_j})$ tail, in \cref{Table 1}.
\end{proof}
\begin{lemma}\label{m less than n}
    \cref{conjectured solution} holds also for $m<n$.
\end{lemma}
\begin{proof}
We will reduce it to the case in \cref{conjectured solution} where we swap $a_j$ with $d_{j}$, $c_{j}$ with $-b_j$, and $\mathfrak{a}$ with $-\mathfrak{a}+1$.
    The value of $\mathfrak{a}$ follows from \cref{value of a} and \cref{mutau}. Assume $m<n$, let $A$ be the $A^{(f-1-j)}$ for $\widetilde{w}_{f-1-j}=\mathfrak{t}_{(m,n)}$ and $A$ be the $A^{(f-1-j)}$ for $\widetilde{w}_{f-1-j}=\mathfrak{t}_{(n,m)}$. Also, let $A=\begin{pmatrix}
        \alpha &\beta\\ \gamma&\delta
    \end{pmatrix}$. Then 
    \[\inv(A)\colonequals \begin{pmatrix}
      0&-\frac{1}{v}\\ 1 &0  
    \end{pmatrix}A\begin{pmatrix}
        0&1\\-v&0
    \end{pmatrix}=\begin{pmatrix}
    \sum_{0\leq i\leq m} \bm{d}_{-i}^{(j)} (v+p)^{i} &\sum_{0\leq i\leq n-1} -\bm{c}_{-i}^{(j)} (v+p)^{i}\\
    v(\sum_{0\leq i\leq m-1} -\bm{b}_{-i}^{(j)} (v+p)^{i}) &\sum_{0\leq i\leq n} \bm{a}_{-i}^{(j)} (v+p)^{i}
\end{pmatrix}.\]
    The monodromy equation is given by
    \begin{equation}\label{monodromy}
        \left(\frac{d}{dv}\right)^t\bigg|_{v=-p}\left\{\left[v\frac{d}{dv}A-A\begin{pmatrix}
            \mathfrak{a} &0\\ 0&0
        \end{pmatrix}\right](A)^{\adj}\right\}+O(p^{N-\ell_{f-1-j}-t}) 
    \end{equation}
    for all $0\leq t\leq 1$, $0\leq j\leq f-1$, where $\adj$ stands for adjugate. \\
   We apply $\inv$ to \cref{monodromy}, after simplification, we have the following
    \[\left(\frac{d}{dv}\right)^t\bigg|_{v=-p}\left\{\left[v\frac{d}{dv}\inv(A)-\begin{pmatrix}
            0&\frac{1}{v}\gamma\\-v\beta&0
        \end{pmatrix}-\inv(A)\begin{pmatrix}
            0 &0\\ 0&\mathfrak{a}
        \end{pmatrix}\right](\inv(A))^{\adj}\right\}+O(p^{N-\ell_{f-1-j}-t}).\]
Since $\inv(A)\begin{pmatrix}
    1&0\\0&0
\end{pmatrix}-\begin{pmatrix}
    1&0\\0&0
\end{pmatrix}\inv(A)=\begin{pmatrix}
            0&\frac{1}{v}\gamma\\-v\beta&0
        \end{pmatrix}$,
The leading term up to modulo $(v+p)^\ell$ is equivalent to 
\[\left(\frac{d}{dv}\right)^t\bigg|_{v=-p}\left\{\left[v\frac{d}{dv}\inv(A)-\inv(A)\begin{pmatrix}
            1-\mathfrak{a} &0\\ 0&0
        \end{pmatrix}\right]\inv(A)^{\adj}\right\}.\]
Therefore, we can apply \cref{conjectured solution} to $\inv(A)$.
\end{proof}
Now assume $\widetilde{w}_{f-1-j}=\mathfrak{w}\mathfrak{t}_{(m,n)}$. 
Similar to the case where $\widetilde{w}_{f-1-j}=\mathfrak{t}_{(m,n)}$, note that $\overline{\bm{b}_0^*}=\alpha$, $\overline{\bm{c}_0^*}$ and $\overline{\bm{d}_0}=\gamma$. In particular, if $\gamma_{f-1-j}\neq0$, then $\bm{d}_0\neq0$.
\begin{theorem}\label{conjectured solution 2}
Assume $\widetilde{w}_{f-1-j}=\mathfrak{w}\mathfrak{t}_{(m,n)}$,
\[(I^{(j),\leq(\ell_{f-1-j},0)}+I^{(j),\nabla}, p^{N-2\ell_{f-1-j}+1})= (I^{(j)}_{poly}, p^{N-2\ell_{f-1-j}+1}),\]
where $I^{(j)}_{poly}$ is given by row 5 of \cref{Table 2} without the $p^{k_j}$ tail.
\end{theorem}
\begin{proof}
We first deal with the case where $m\geq n$. The value of $\mathfrak{a}$ follows from \cref{value of a} and \cref{mutau}.
As $p>2\ell_{f-1-j}$, we have $\mathfrak{a}\pm k\not\equiv 0\mod p$, so $\mathfrak{a}\pm k$ is a unit for all $0\leq k\leq \ell_{f-1-j}$. We first prove that the solution to $M^{m-n}_{3m-n-1-k}(s,t)$ for $0\leq k\leq m+n-1$, $\leq s,t\leq 2$ is $\bm{a}_{-k},\bm{b}_{-k}, \bm{c}_{-k}, \bm{d}_{-k}$ as conjectured by the equations in \cref{Table 2} without the $O(p^{k_j})$ term. For $k=0$, it is just the definition. We then proceed as in the case where $\widetilde{w}_{f-1-j}=\mathfrak{t}_{(m,n)}$. We will prove by induction. If we have verified $\bm{a}_{-k}$, $\bm{d}_{-k}$, $\bm{b}_{-k}$ and $\bm{c}_{-k}$ for $k\leq i-1$, we will deduce $\bm{b}_{-i}$ and $\bm{c}_{-i}$ from $M^{m-n}_{3m-n-1-i}(1,1)$ and $M^{m-n}_{3m-n-1-i}(2,2)$. From $M^{m-n}_{3m-n-i}(1,1)$, we have
\begin{align*}
    &(\mathfrak{a}+m-i)\bm{b}_{-i}\bm{c}_0^*+(\mathfrak{a}+m)\bm{b}_0^*\bm{c}_{-i}\\
    =&-\sum_{1\leq j\leq i-1}(\mathfrak{a}+m-j)\bm{c}_{-i+j}\bm{b}_{-j}+\sum_{0 \leq j\leq i-1}(2m-n-i+j)\bm{a}_{1-i+j}\bm{d}_{-j}+(m-j)p\bm{c}_{-i+j}\bm{b}_{1-j}.
\end{align*}
And from $M^{m-n}_{3m-n-i}(2,2)$, we have
\begin{equation*}
\begin{split}
    (2m-n-\mathfrak{a})\bm{b}_{-i}\bm{c}_0^*+(2m-n-\mathfrak{a}-i)\bm{b}_0^*\bm{c}_{-i}=-\sum_{1\leq j \leq i-1}(2m-n-i+j+\mathfrak{a})\bm{c}_{-i+j}\bm{b}_{-j}\\
    +\sum_{0\leq j \leq i-1}(m-j)\bm{a}_{1-i+j}\bm{d}_{-j}+(2m-n-i+j)p\bm{c}_{1-i+j}\bm{b}_{-j}.
\end{split}
\end{equation*}
Hence,
\begin{equation}\label {equation for b 2}
\begin{aligned}
    \bm{b}_{-i}=\dfrac{-1}{i\bm{c}_0^*}\biggl(&\sum_{0\leq j\leq i-2}(j+1)\bm{c}_{1-i+j}\bm{b}_{-1-j}\\
    &-\sum_{0\leq j \leq i-1}\bigl[(\mathfrak{a}+i-j-m+n)\bm{a}_{1-i+j}\bm{d}_{-j}+(\mathfrak{a}+j)p\bm{c}_{1-i+j}\bm{b}_{-j}\big]\biggr).
\end{aligned}
\end{equation}
\begin{equation*}\label {equation for c 2}
\begin{split}
    \bm{c}_{-i}=&\dfrac{-1}{i\bm{b}_0^*}\biggl(\sum_{1\leq j \leq i-1}(i-j)\bm{c}_{-i+j}\bm{b}_{-j}+\sum_{0\leq j \leq i-1}\bigl[(\mathfrak{a}-j-m+n)\bm{a}_{1-i+j}\bm{d}_{-j}+(\mathfrak{a}-i+j)p\bm{c}_{1-i+j}\bm{b}_{-j}\bigr]\biggr).
    \end{split}
\end{equation*}
\begin{lemma}\label{relations 2}
Assume $\bm{a}_{1-i+j},b_{m-j}, \bm{c}_{1-i+j},\bm{d}_{-j}$ are as given in the $I^{(j),\reg}$ row without the $O(p^{k_j})$ tail in \cref{Table 2}, then for $j, i-j-1\geq 0$, we have equalities:
\begin{align*}
    (j+1)\bm{b}_{-j-1}\bm{c}_{1-i+j}&=p(\mathfrak{a}+j)\bm{b}_{-j}\bm{c}_{1-i+j}+(\mathfrak{a}-m+n+i-j)\bm{a}_{1-i+j}\bm{d}_{-j};\\
    -(i-j)\bm{b}_{-j}\bm{c}_{-i+j}&=p(\mathfrak{a}-i+j)\bm{b}_{-j}\bm{c}_{-i+j+1}+(\mathfrak{a}-m+n-j)\bm{a}_{-i+j+1}\bm{d}_{-j}
\end{align*}
\end{lemma}
By \cref{relations 2}, the right-hand side of \cref{equation for b 2} all cancel out except the term $(\mathfrak{a}+i-1)p\bm{c}_0^*b_{-i+1}+(\mathfrak{a}+1-m+n)\bm{a}_0\bm{d}_{-i+1}$, which is equal to $\bm{b}_{-i}$ again by \cref{relations 2}.
we can similarly verify that $\bm{c}_{ -i}$ is as in the $I^{(j),\reg}$ row without the $O(p^{k_j})$ tail in \cref{Table 2}.\par
We now show that given solutions $\bm{b}_{-k}$ and $\bm{c}_{-k}$ for $k\leq i$ and $\bm{a}_{-k}$ and $\bm{d}_{-k}$ for $k\leq i-1$, we can deduce $\bm{a}_{-i}$ and $\bm{d}_{-i}$ from $M^{m-n}_{3m-n-1-i}(1,2)$ and $M^{m-n}_{3m-n-1-i}(2,1)$ respectively that:
\begin{equation}\label{equation for a 2}
\begin{aligned}
    \bm{a}_{-i}=&\dfrac{-1}{(\mathfrak{a}-m+n+1+i)\bm{b}_0^*}\\&(\sum_{0\leq j\leq i-1}(\mathfrak{a}-m+n+1-i+2j)\bm{a}_{-j}\bm{b}_{-i+j}+
    p(m-n-2+i-2j)\bm{a}_{-j}\bm{b}_{1-i+j});
    \end{aligned}
\end{equation}
\begin{equation*}\label{equation for d 2}
    \bm{d}_{-i}=\dfrac{-1}{(\mathfrak{a}-m+n-i)\bm{c}_0^*}(\sum_{0\leq j\leq i-1} (\mathfrak{a}-m+n+i-2j)\bm{c}_{-i+j}\bm{d}_{-j}+
    p(m-n-i+2j)\bm{c}_{1-i+j}\bm{d}_{-j}).
\end{equation*}
\begin{lemma}\label{relations for cd}
  Assume $\bm{a}_{+j},\bm{b}_{-i+j}, \bm{c}_{+j}, \bm{d}_{-i+j}$ for $0\leq j\leq i$ are as given in the $I^{(j),\reg}$ row without the $O(p^{k_j})$ tail in \cref{Table 2}. Let $$T_j=(\mathfrak{a}-m+n+1-i+2j)\bm{a}_{-j}\bm{b}_{-i+j}+p(m-n-2+i-2j)\bm{a}_{-1-j}\bm{b}_{+1-i+j},$$ $$R_j=\frac{-j}{i}(\mathfrak{a}-m+n+1+j)\bm{a}_{-j}\bm{b}_{-i+j}.$$ Then we have for $0\leq j \leq i-1$, 
$T_{j}+R_j=R_{j+1}$. Similarly, let $$T'_j=(\mathfrak{a}-m+n+i-2j)\bm{c}_{-i+j}\bm{d}_{-j}+p(m-n-i+2j)\bm{c}_{1-i+j}\bm{d}_{-j},$$ $$R'_j=\dfrac{-j}{i}(\mathfrak{a}-m+n-j)\bm{c}_{-i+j}\bm{d}_{-j}.$$ Then, for $0\leq j \leq i-1$, 
$T'_{j}+R'_j=R'_{j}$.
\end{lemma}
By \cref{relations for cd} and the fact that $R_0=0$, the right-hand side of \cref{equation for a 2} is
\[\frac{-1}{(\mathfrak{a}-m+n-1-i)\bm{a}_0^*}\sum_{0\leq j \leq i}T_j=\frac{-R_{i}}{(\mathfrak{a}-m+n-1-i)\bm{a}_0^*},\]
which is precisely the conjectured $\bm{a}_{-i}$ in the $I^{(j),\reg}$ row in \cref{Table 2} without the $O(p^{k_j})$ tail, again by \cref{relations for cd}. The proof for $\bm{d}_{-i}$ goes exactly the same way, using $T'_j$, $R'_j$ instead of $T_j$, $R_j$. Therefore, we finished the induction step and proved that the solution to $M^{m-n}_{3m-n-1-k}(s,t)$ for $0\leq k\leq m+n-1$, $\leq s,t\leq 2$ is as conjectured in the $I^{(j),\reg}$ row in \cref{Table 2} for all $0\leq k\leq m+n$ and all $i,j$.\par
As in the case of $\mathfrak{t}_{(m,n)}$, by \eqref{raising weight for monodromy eq}, we know that the solutions, given in the $I^{(j),\reg}$ row in \cref{Table 2} without the $O(p^{k_j})$ tail, satisfy the monodromy equations up to modulo $I^{\mathrm{extra}}$. By \cref{raising weight for monodromy eq}, we have
\begin{equation*}
\begin{split}
    M^{m-n}_{3m-2n-2}(1,1)=M^0_{m-2}(1,1)+(\mathfrak{a}+m-n-2)\bm{c}_0^*\bm{b}_{-n-1}+\delta_{m,n+1}(\mathfrak{a}+n)\bm{c}_{-n-1}\bm{b}_{0}^*;\\
    M^{m-n}_{3m-2n-2}(2,2)=M^0_{m-2}(2,2)+(m-\mathfrak{a})\bm{c}_0^*\bm{b}_{-n-1}-\delta_{m,n+1}\mathfrak{a}\bm{c}_{-n-1}\bm{b}_{0}^*.
\end{split}
\end{equation*}
As $\bm{b}_0^*, \bm{c}_0^*$ are units, we deduce that $\bm{b}_{-n-1}\in I^{\mathrm{extra}} $, and hence the terms with $*$ in \cref{Table 2}, without the term $O(p^{k_j})$, are in $I^{\mathrm{extra}}$, following the description of $\bm{b}_{-n-1}$ according to \cref{Table 2}. Conversely, by \eqref{raising weight for monodromy eq}, all the generators of $I^{\mathrm{extra}}$ are divisible by $\bm{a}_{-j}$ or $ \bm{c}_{-j}$ where $j\geq m$ or $\bm{b}_{-j}$ or $ \bm{d}_{-j}$ where $j\geq n+1$. By the computation above, they in turn are divisible by terms with $*$ in \cref{Table 2} without the $O(p^{k_j})$ tail.
We then verify that if we substitute the conjectured solution in $H(0)$, it is divisible by the terms with $*$ in \cref{Table 2} without the $O(p^{k_j})$ tail, which is straightforward. 
\end{proof}
\begin{lemma}\label{m less than n 2}
    \cref{conjectured solution 2} holds for $m<n$.
\end{lemma}
\begin{proof}
   It goes exactly the same as the proof of \cref{m less than n}
\end{proof}
\begin{definition}
    Let $R^{(j)}_{\poly}$ be the polynomial ring generated over $\mathcal{O}$ as the variables generating $R^{(j)}$ in the $4$th row of \cref{Table 1} and \cref{Table 2}, except replacing $x_{12}$, $x_{21}$ by $\frac{x_{12}}{x_{11}}$ and $\frac{x_{21}}{x_{22}}$ respectively in \cref{Table 1} and replacing $x_{11}$, $x_{22}$ by $\frac{x_{11}}{x_{12}}$ and $\frac{x_{22}}{x_{21}}$ respectively in \cref{Table 2}. We define $R_{\poly}\colonequals \otimes_{\mathcal{O},j} R^{(j)}_{\poly}$. We let $I^{(j)}$ be defined by the elements in the row $I^{(j),\reg}$ in \cref{Table 1} and \cref{Table 2}, where the last equation in $I^{(j),\reg}$ is replaced by the equations in $I^{(j)}$ if $\widetilde{w}=\mathfrak{t}_{(m,m)}$. We define $I^{(j)}_{\poly}$ as the ideal of $R^{(j)}_{\poly}$ generated in the same way but without the $O(p^{k_j})$ tail.
\end{definition}
\begin{lemma}\label{dim drops}
    In the case where $\widetilde{w}_{f-1-j}=\mathfrak{t}_{(m,m)}$, $R^{(j)}_{\poly}/(I^{(j)}_{\poly},\bm{b}_0)=R^{(j)}_{\poly}/(I^{(j)}_{\poly},\bm{c}_0)=\mathcal{O}[x_{11},x_{22}]$.
\end{lemma}
\begin{proof}
If $\bm{b}_0=0$, by row for $I^{(j)}$ of \cref{Table 1}, $\bm{c}_0=0$. Moreover, by the $I^{(j),\reg}$ row without the $O(p^{k_j})$ tail in \cref{Table 1}, $a_i,b_i,c_i,d_i=0$ except for $a_m, d_m$. By symmetry, the same happens if $\bm{c}_0=0$.
\end{proof}
\begin{definition}\label{definition of Ij}
The ideal $I^{(j),\reg}$ is given in \cref{Table 1} and \cref{Table 2}. We let $I^{(j),\reg}_{\poly}$ be the ideal of $R^{(j)}_{\poly}$ generated in the same way but without the tail $O(p^{k_j})$. We let $I^{\reg}_{\poly}\colonequals \sum_jI^{(j), \reg}_{\poly}$.
\end{definition}
\begin{corollary}\label{pth power}
    We have $p^{\ell_{f-1-j}}\in H^{(j)}+I^{(j),reg}_{\poly}$.
\end{corollary}
\begin{proof}

We use $G$ to denote the term with $*$ in row for $I^{(j),\reg}$ of \cref{Table 1} and \cref{Table 2} without the term $O(p^{k_j})$. Note that $H^{(j)}+I^{(j)}_{\poly}$ contains $G$ and the partial derivatives of $G$. If $\widetilde{w}_{f-1-j}=\mathfrak{t}_{(m,n)}$ let $x=\frac{\bm{b}_0}{\bm{a}_0^*}$ and $y=\frac{\bm{d}_0}{\bm{c}_0^*}$. If $m> n$ and $\gamma_{f-1-j}=0$, then $G(x,y)=x\prod_{i=1}^{n} (xy-\alpha_i)$. Otherwise, $G(x,y)=\prod_{i=1}^{n} (xy-\alpha_i)$. If $\widetilde{w}_{f-1-j}=\mathfrak{w}\mathfrak{t}_{(m,n)}$, we let $x=\frac{\bm{a}_0}{\bm{b}_0^*}$, $y=\frac{\bm{d}_0}{\bm{c}_0^*}$, then $G(x,y)=\prod_{i=1}^{n+1} (xy-\alpha_i)$.
In any case, $\alpha_i$ are distinct.
    \begin{equation}\label{condition}
        v_p(\alpha_i)=1 \text{ for all } i  \text{ and } v_p(\alpha_i-\alpha_j)=1 \text{ for all } i\neq j.
    \end{equation}
    (Here, $v_p$ denotes the $p$-adic valuation, and we define $v_p(0)=\infty$.) We first consider the case where $G(x,y)=\prod_{i=1}^{n+1} (xy-\alpha_i)$ (respectively, $\prod_{i=1}^{n} (xy-\alpha_i)$).
    If we let $xy=z$, then $G$ is a polynomial in $z$ of degree $n+1 $ (respectively $n)$ and $x\frac{\partial G}{\partial x}=z\frac{\partial G}{\partial z}$. We will show below that $p^{2n+1} (\text{resp. } p^{2n-1})\in (G(z),z\frac{\partial G}{\partial z})$ which will imply that $p^{\ell_{f-1-j}}\in (G,x\frac{\partial G}{\partial x})$.\par
    Given $G=z^{n+1}+a_{n}z^{n}+\dots + a_0$ with roots satisfying \cref{condition}. Let $p_{-1}=\frac{\partial G}{\partial z}=(n+1)z^{n}+na_{n}z^{n-1}+\dots +a_1$. For $-1\leq i\leq n-1$, given $p_i=b_{i,n}z^{n}+\dots+b_{i,0}$, we obtain $p_{i+1}=b_{i,n}G-zp_i$.
\begin{lemma}
    $v_p(b_{i,k})\geq i+1+n-k$.
\end{lemma}    
\begin{proof}
     We will prove this by inducting on $i$. For $1\leq k \leq n$, $b_{-1,k}=(k+1)a_{k+1}=(k+1)(-1)^k\sum_{i_1<\dots<i_{n-k}}\alpha_{i_1}\dots\alpha_{i_{n-k}}$, where the sum is over any $k$-tuple, and $b_{-1,0}=0$, the lemma holds for $i=-1$. Assume that it is true for $i$, then $$v_p(b_{i+1,0})=v_p(b_{i,n}a_0)=v_p(b_{i, n})+v_p(a_0)\geq i+n+2.$$
    For $k>0$,
    \begin{align*}
        v_p(b_{i+1,k})&=v_p(b_{i,n}a_k+b_{i, k-1})\\
        &\geq \max\{v_p(b_{i,n})+v_p(a_k), v_p(b_{i, k-1})\}\\
        &\geq i+n+2-k. \qedhere
    \end{align*}
\end{proof}
\begin{lemma}
The determinant $D$ of the matrix given by the coefficients of the polynomials $p_i$ is determined as follows:
    $$D\colonequals \begin{vmatrix}
        b_{0, 0}& \cdots &b_{0, n}\\
        \vdots &\ddots &\vdots\\
        b_{n,0} &\cdots &b_{n,n}
    \end{vmatrix}=\pm\prod_{i\neq j}(\alpha_i-\alpha_j)\prod_i \alpha_i.$$
\end{lemma}
\begin{proof}
First, we show that $D$ is equal to res$(G, z\frac{\partial G}{\partial z})$. Recall that $\res(G, z\frac{\partial G}{\partial z})=$
\begin{gather*}
\begin{vmatrix}
    1 &0&\cdots&0 &n+1 &0&\cdots &0\\
    a_{n}&1&\cdots& &na_{n}&n+1&\cdots &0\\
    \vdots&\vdots &\ddots& &\vdots&\vdots&\ddots &\vdots \\
    a_0 &a_1& \cdots&1&0&a_1&\cdots &n+1 \\
    \vdots&\vdots &\ddots&\vdots &\vdots&\ddots& &\vdots \\
    0&0 &\cdots &a_1& 0 &\cdots & 0& a_1\\
    0&0 &\cdots &a_0 &0 &\cdots &0 &0 
\end{vmatrix}=\begin{vmatrix}
    1 &0&\cdots&0 &0 &0&\cdots &0\\
    a_{n}&1&\cdots& &b_{0,n}&0&\cdots &0\\
    \vdots& &\ddots& &\vdots&\ddots& &\vdots \\
    a_0 &a_1& \cdots&1&b_{0,0}&b_{0,1}&\cdots &0 \\
    \vdots&\vdots &\ddots&\vdots &\vdots&\ddots& &\vdots \\
    0&0 &\cdots &a_1& 0 &\cdots & 0& b_{0,1}\\
    0&0 &\cdots &a_0 &0 &\cdots &0 &b_{0,0}
\end{vmatrix}.
\end{gather*}
Then the process of producing $p_{i+1}$ from $p_{i}$ is equivalent to recursively subtracting the $2n+2-i-k$th column by multiples $n+1-k$th column, for all $0\leq k\leq n-i$, to reduce it to an upper triangular matrix. Therefore, by applying the column reduction corresponding to generating $p_2$ to $p_{n+1}$, we obtain
$$\begin{vmatrix}
    1 &\cdots&0 &0 &\cdots &0\\
    \vdots &\ddots& &\vdots&\ddots&\vdots \\
    a_0 & \cdots&1&0&\cdots &0 \\
    0&\cdots&a_{n} &  b_{n, n}& \cdots &b_{0, n}\\
    \vdots &\ddots&\vdots & \vdots &\ddots &\vdots\\
    0 &\cdots &a_0& b_{n,0} &\cdots &b_{0,0}
\end{vmatrix}.$$
Therefore, we show that $D$ is equal to $\res(G, z\frac{\partial G}{\partial z})$. On the other hand,
\begin{align*}
    \res(G, z\frac{\partial G}{\partial z})&=\res(G, \frac{\partial G}{\partial z})\res(G,z)\\
    &=(-1)^{\frac{n(n-1)}{2}}\Disc(G)(a_0)\\
    &= a_0\prod_{i\neq j}(\alpha_i-\alpha_j). \qedhere
\end{align*}
\end{proof}
In particular, assume that $v_p(\alpha_i)=1$ for all i, and $v_p(\alpha_i-\alpha_j)=1$ for all $i\neq j$, then $v_p(D)=(n+1)^2$.\par
Now we would like to calculate $D$ in another way. We would like to perform row reduction to reduce the matrix to a shape such that exactly one entry on each row and each column is nonzero (since the determinant is nonzero). As row reduction corresponds to adding scalar multiple of polynomials together, the nonzero entry appearing on the first column after the reduction, corresponding to the scalar term, is in $(G(z),z\frac{\partial G}{\partial z})$.\par
As we know that $D$ is nonzero, not all $b_{j,n}$ are $0$, so we can find one $j$ such that $v_p(b_{j,n})$ is the smallest (and finite), let $x_n\colonequals b_{j,n}$. Then $\frac{b_{k,n}}{b_{j,n}}\in \mathcal{O}$ for all $k$. Then we can subtract the $k$th row by $\frac{b_{k,n}}{b_{j,n}}\times$ $ j$th row, corresponding to $p_k-\frac{b_{k,n}}{b_{j,n}}\times p_j$. We then perform row reductions recursively. We set $b^0_{s,t}=b_{s,t}$. After the $i$th row reduction, we relabel the $(s,t)$th entry as $b^i_{s,t}$. We choose a $j$ such that $v_p(b_{j, n-i}^i)$ is the lowest and $b_{j,n-k}\neq x_{n-k}$ for $k<i$ (it is possible as the determinant is nonzero), and define $x_{n-i}\colonequals b_{j,n-i}$. By repeating the process, we reduce the matrix to a shape such that exactly one entry on each row and each column is nonzero. Moreover, all nonzero entries are given by $x_i=b^{n}_{\sigma(i), i}$, for some $\sigma\in \mathfrak{S}_{n+1}$.
\begin{lemma}
If $x_i=b^n_{j, i}$, $v_p(x_i)=j+1+n-i$.
\end{lemma}
\begin{proof}
We consider the $(n+1)\times (n+1)$ matrix given by the lower bound of the valuation $v_p$ of the entry of $b_{s,t}$:
    $$\begin{array}{|c|c|c|c|}
         \hline
         n+1&n&\cdots& 1\\
         \hline
         n+2&n+1& \cdots &2\\
         \hline
         \vdots&\vdots&\ddots&\vdots\\
         \hline
         2n+1&2n &\cdots&n+1\\
         \hline
    \end{array}.$$
    As $b^{i+1}_{s,t}=b^i_{s,t}-\frac{b_{s,n-i}}{b_{j,n-i}}b^i_{j,t}$ for some $j$, $v_p(b^{i+1}_{s,t})\geq v_p(b^{i}_{s,t})$. Therefore, the grid remains unchanged if we replace $b_{s,t}$ by $b^i_{s,t}$ for any $i$. On the other hand, $D=\prod x_i$, hence $\sum v_p(x_i)=v_p(D)=(n+1)^2$. It is a simple calculation to show that if we choose one element from each row and column, then they add up to exactly $(n+1)^2$. Therefore, the inequality $v_p(x_i)=v_p(b_{j,i})\geq j+1+n-i$ must be an equality. 
\end{proof}
 Therefore, we deduce $v_p(x_0)\leq 2n+1$. Since $x_0\in (G(z),z\frac{\partial G}{\partial z})$, we finish the proof for the first case where $G(x,y)=\prod_{i=1}^{n+1} (xy-\alpha_i)$ or $\prod_{i=1}^{n} (xy-\alpha_i)$.\par
Now assume $G(x,y)=x\prod_{i=1}^n (xy-\alpha_i)$, with $\alpha_i$ satisfying \cref{condition}. As in the other case, we will show that $p^{2n}\in (G,\frac{\partial G}{\partial x})$. Let $F=yG$, and $xy=z$, then $F=z\prod_{i=1}^n(z-\alpha_i)$, then $F$ is a polynomial of degree $n+1$ in $z$. Moreover, $ \frac{\partial G}{\partial x}=\frac{\partial F}{\partial z}$. 
    Let $p_{1}=\frac{\partial F}{\partial z}$, and given $p_i$, we obtain $p_{i+1}=(b_{i,n})F-z(p_i)$ where $b_{i,k}$ is the coefficient of $z^{k}$ in $p_i$.
    Again, we consider the determinant of a $n+1\times n+1$ matrix
    $$D'=\begin{vmatrix}
        (p_{n+1}) &\cdots & (p_1)
    \end{vmatrix},$$ where the column with $(p_{i})$ means that the entries are given by the coefficients of $p_i$ in the descending power of $z$.
    We can apply the same argument about column reduction, except that the resultant matrix is now an $(2n+1) \times (2n+1)$ matrix, with the last column fixed in the column reduction. We obtain 
    \begin{align*}
    D'&=\pm\res(F,F')\\
    &=\pm\Disc(F)\\
    &=\pm \prod_{i\neq j}(\alpha_i-\alpha_j)(\prod_{i=1}^n -\alpha_i^2). 
    \end{align*}
    Therefore, $v_p(D')=n^2+n$. \par
    We then similarly consider the $n+1\times n+1$ grid given by the lower bound of the valuation $v_p$ of the entry of $b_{s,t}$, except in this case, all are shifted by $1$:
    $$\begin{array}{|c|c|c|c|}
         \hline
         2n&2n-1&\cdots& n\\
         \hline
         2n-1&2n& \cdots &n-1\\
         \hline
         \vdots&\vdots&\ddots&\vdots\\
         \hline
         n&n-1 &\cdots&0\\
         \hline
    \end{array}.$$ 
    Again, the same kind of simple calculation shows that if we choose one element from each row and column, then they add up to exactly $n^2+n$. Therefore, by the same argument as above, we show that $p^{2n}\in (F, \frac{\partial F}{\partial z})=(G, \frac{\partial G}{\partial x})$.
\end{proof}
\begin{definition}\label{defining ideals}
We let $I_{\infty}=\ker (R\twoheadrightarrow R^{\leq (\ell_j,0)_j,\tau,\nabla}_{\overline{\mathfrak{M}}, \overline{\beta}})$ and $I_{\infty}^{\reg}=\ker (R\twoheadrightarrow R^{\leq (\ell_j,0)_j,\tau,\nabla, \reg}_{\overline{\mathfrak{M}}, \overline{\beta}})$.
    Fix $\lambda=(\lambda_{j,1}, \lambda_{j,2})_j\leq (\ell_j,0)_j$.
\end{definition}
\begin{theorem}\label{deformation ring}
    Assume $\overline{\rho}$ is of the form in \cref{condition on rho bar} and is $4\ell$-generic and $\tau$ is $(2\ell+2)$-generic where $\ell=\max\{\ell_j\}$. If $W(\overline{\rho})\cap\JH(\overline{\sigma}((\ell_j,0), \tau))\neq\emptyset$, we have an isomorphism
    \[R_{\overline{\rho}}^{\leq(\ell_j,0)_j, \tau, \reg}\llbracket X_1, \dots,X_{2f}\rrbracket \cong(R/\sum_j I^{(j),\reg})\llbracket Y_1,\dots,Y_4\rrbracket.\]
    The irreducible components of $\Spec R_{\overline{\rho}}^{\leq (\ell_j,0)_j, \tau, \reg}$ are given by $\Spec R_{\overline{\rho}}^{\lambda, \tau}$ where $\lambda_j\leq (\ell_j,0)_j$ and are regular for all $j$ and $\JH\left(\overline{\sigma}(\lambda, \tau)\right)\cap W(\overline{\rho})\neq\emptyset$.
    Moreover, via the isomorphism above, the kernel of the natural isomorphism $R_{\overline{\rho}}^{\leq (\ell_j,0)_j, \tau, \reg}\llbracket X_1, \dots,X_{2f}\rrbracket \to R_{\overline{\rho}}^{\lambda, \tau, \reg}\llbracket X_1, \dots,X_{2f}\rrbracket $ is given by $\mathfrak{p}^\lambda\colonequals \sum_j \mathfrak{p}^{(j), \lambda_{f-1-j}} $ where $\mathfrak{p}^{(j),\lambda_{f-1-j}}$ is given by row $6$ in \cref{Table 1} and \cref{Table 2}.\par
\end{theorem}
 \begin{proof}  
We follow the proof of \cite[Proposition~4.2.1]{BHHMS}. Instead of $h=3$, we allow $h=\ell$. 
Moreover, to account for non-semisimple $\overline{\rho}$, we follow the proof of \cite[Theorem~4.2]{Yitong}. We use $\widetilde{w}$-gauge bases instead of gauge bases. If $W(\overline{\rho})\cap\JH(\overline{\sigma}((\ell_j,0), \tau))\neq\varnothing$, by \cref{tame type criterion}, $R^{\leq(\ell_j,0),\tau,\reg}_{\overline{\rho}}\neq0$. Moreover, by \cref{tame type criterion}, we have $\tau=\tau_{\widetilde{w}}$ for some $\tau_{\widetilde{w}}\in X(\overline{\rho},(\ell_j,0))$. We modify the definition of $D^{\leq (\ell_j,0)_j,\tau_{\widetilde{w}}}_{\overline{\mathfrak{M}}, \overline{\beta}}(R)$ by requiring $\beta$ to be a $\widetilde{w}$-gauge basis instead of a basis. Then for any $(\mathfrak{M},\beta,\jmath)\in D^{\leq (\ell_j,0)_j,\tau_{\widetilde{w}}}_{\overline{\mathfrak{M}}, \overline{\beta}}(R)$, we have a corresponding matrix $A^{(f-1-j)}$ where $A^{(f-1-j)}\mod m_R\equiv \overline{A}^{(f-1-j)}$. Note that $A^{(f-1-j)}$ may have a $\widetilde{w}$-gauge basis, but may not have shape $\widetilde{w}$ (\textit{cf.} \cref{example}). \par 
By the same argument as in \cite[Theorem~4.2]{Yitong}, we have an isomorphism
\[R^{\leq(\ell_j,0)_j,\tau_{\widetilde{w}}}_{\overline{\rho}}\llbracket x_1,\dots, x_{2f}\rrbracket \cong R^{\leq (\ell_j,0)_j,\tau_{\widetilde{w}},\nabla}_{\overline{\mathfrak{M}}, \overline{\beta}}\llbracket Y_1,\dots Y_4\rrbracket \]
Recalls that $$I_{\infty}^{\reg}=\ker (R\twoheadrightarrow R^{\leq (\ell_j,0)_j,\tau_{\widetilde{w}},\nabla, \reg}_{\overline{\mathfrak{M}}, \overline{\beta}}).$$
And we will show that $I_\infty^{\reg}=\sum_j I^{(j), \reg}$, where $I^{(j),\reg}_{\poly}$ is defined in \cref{definition of Ij}. By construction, $I^{(j)}\subset (I^{(j), \leq(\ell_j,0)}, I^{(j),\nabla})\subset \ker (R\twoheadrightarrow R^{\leq (\ell_j,0)_j,\tau_{\widetilde{w}},\nabla}_{\overline{\mathfrak{M}}, \overline{\beta}})$.
Since $R^{\leq (\ell_j,0)_j,\tau_{\widetilde{w}},\nabla, \reg}_{\overline{\mathfrak{M}}, \overline{\beta}}$ is the quotient of $R^{\leq (\ell_j,0)_j,\tau_{\widetilde{w}},\nabla}_{\overline{\mathfrak{M}}, \overline{\beta}}$ for which every component has the maximal dimension, by \cref{dim drops}, we deduce that all the equations generating $I_{\infty}^{\reg}$ do not contain $\bm{b}_{0}$ or $\bm{c}_{0}$ as a factor when $\widetilde{w}_{f-1-j}=\mathfrak{t}_{(m,m)}$. It follows that
$I_{\poly}^{\reg}\colonequals \sum_j I^{(j),\reg}_{\poly}\subset(I_{\infty}^{\reg}, p^{N-2\ell+1})$. 
From \cref{pth power}, we know that $p^\ell_j\in H^{(j)}+I^{(j), \reg}_{\poly}$. Since $N-(2\ell-1)>2\times \ell$, by applying Elkik's approximation (\cite[Lemme~1]{Elkik}) in the same way as in \cite[Proposition~4.2.1]{BHHMS}, we obtain an $\mathcal{O}$-algebra homomorphism, $\widetilde{\phi}^{(j),\reg}\colon R^{(j)}_{\poly} /I_{\poly}^{(j),\reg}\to R/I_{\infty}^{\reg} $ such that $\widetilde{\phi}^{(j),\reg}$ agrees with the natural map modulo $p^{N-3\ell+1}$. We let $\widetilde{\phi}^{\reg}\colonequals \otimes_j\widetilde{\phi}^{(j),\reg}$. As $N> 4\ell-1$, we have the following surjection:
\begin{equation}\label{computing ker}    \widetilde{\phi}^{\reg}\colon {R/I^{\reg}_{\poly}} \twoheadrightarrow {R^{\leq(\ell_j,0)_j,\tau_{\widetilde{w}}, \nabla, \reg}_{\overline{\mathfrak{M}},\overline{\beta}}}.
\end{equation}
We will show that $\widetilde{\phi}^{\reg}$ is an isomorphism.
Note that $\left(\frac{(\mathfrak{a}-m+n)(\mathfrak{a}-m+n-1)\bm{b}_{0}\bm{c}_{0}}{\bm{a}_{0}^*\bm{d}_{0}^*}-(m-n-k)kp\right)$ and $\left(\frac{(\mathfrak{a}-m+n)(\mathfrak{a}-m+n+1)\bm{a}_{0}\bm{d}_{0}}{\bm{b}_{0}^*\bm{c}_{0}^*}- (\mathfrak{a}-m+n-k)(\mathfrak{a}+k) p\right)$ are irreducible for all $k$ by \cite[Lemma~3.3.1]{BHHMS}. Therefore, $R^{(j)}/I_{\poly}^{(j, \reg) } $ is reduced, $\mathcal{O}$-flat, with $S(\widetilde{w}_{f-1-j})$ (\textit{cf.} \cref{number of layers}) irreducible components which are
geometrically integral and of relative dimension $3$ over $\mathcal{O}$. By \cite[Lemma~2.6]{Cal} and \cite[Lemma~3.3]{BLGHT}, $R/I_{\poly}^{\reg}$ is reduced, $\mathcal{O}$-flat with $S(\tau_{\widetilde{w}})$ irreducible components, each with dimension $3f$ over $\mathcal{O}$. Hence, to show $\widetilde{\phi}^{\reg}$ is an isomorphism, it remains to show that $R^{\leq(\ell_j,0)_j,\tau, \nabla}_{\overline{\mathfrak{M}},\overline{\beta}}$, equivalently, $R^{\leq(\ell_j,0)_j,\tau, \nabla}_{\overline{\rho}}$ has at least $S(\tau_{\widetilde{w}})$ components, which follows from \cref{tame type criterion}. Therefore, $\widetilde{\phi}^{\reg}$ is an isomorphism and induces the natural map modulo $p^{N-3\ell+1}$, we show that $(I_{\poly}^{\reg}, p^{N-3\ell+1})=(I_{\infty}^{\reg}, p^{N-3\ell+1})$.\par
By the same argument as in \cite[Lemma~4.2.4]{BHHMS}, we can show that there exists an automorphism of local $\mathcal{O}$-algebra $\psi\colon  R\to R$ such that 
\[\begin{tikzcd}
	R \arrow [swap]{r}{\sim}[swap]{\psi}\arrow{d}& R\arrow{d} \\
	{R/I_{\poly}^{\reg}} \arrow[swap]{r}{\sim} [swap]{\widetilde{\phi}^{\reg}} & {R/I_\infty^{\reg}}
\end{tikzcd}.\]
commutes and such that $\psi$ induces the identity modulo $p^{N-3\ell+1}$. Hence, $\psi$ identifies $I_{\poly}^{\reg}$ with $I_{\infty}^{\reg}$, and $I_{\infty}^{\reg}=\sum_jI^{(j),\reg}$. Moreover, it follows that $\mathfrak{p}^{\lambda}$ are distinct minimal primes containing $I_\infty$. As \cref{computing ker} is an isomorphism, we have the irreducible components of $R^{\leq(\ell_j,0)_j, \tau_{\widetilde{w}}}_{\overline{\rho}}$ in bijection with the set $\lambda$ such that $\JH(\overline{\sigma}(\lambda, \tau_{\widetilde{w}}))\cap W(\overline{\rho})\neq 0$. As in \cite[Proposition~4.2.1]{BHHMS}, this is given explicitly by sending
a component $\mathcal{C}$ to the labelled Hodge--Tate weights of the framed deformation corresponding to
any closed point of the generic fibre of $\mathcal{C}$. Hence, the components are given by $R^{\lambda, \tau_{\widetilde{w}}}_{\overline{\rho}}$ for some $\lambda\leq (\ell_j,0)_j$.\par
It remains to identify the components. We consider the kernel of the composition
\[\phi_{\lambda}\colon R\twoheadrightarrow R/I_{\infty}^{\reg}\cong R^{\leq(\ell_j,0)_j, \tau, \nabla}_{\overline{\mathfrak{M}},\overline{\beta}} \twoheadrightarrow R^{\leq\lambda, \tau, \nabla}_{\overline{\mathfrak{M}},\overline{\beta}}.\]
By \cref{tame type criterion} and the discussion above, we know that $\ker(\phi_\lambda)$ is of the form $\cap_{\lambda'\in X}\mathfrak{p}^{\lambda'}$ for some subset $X\subset X(\tau_{\widetilde{w}}, \underline{(\ell_j,0)_j})$ of cardinality $S(\tau_{\widetilde{w}}, \lambda)$ (\textit{cf.} \cref{number of layers}). We would like to show that $X=X(\tau_{\widetilde{w}}, \lambda)$. We will show that $\lambda'_{f-1-j}\leq\lambda_{f-1-j}$ for all $\lambda'\in X$.
If $\lambda_{f-1-j}=(\ell_j,0)$ then there is nothing to prove. 
Otherwise, because of the finite height condition, $A^{(f-1-j)}$ is divisible by $(v+p)^{\lambda_{f-1-j,2}}$ we conclude that $a_{k}^{(j)},b_{k}^{(j)}, c_{k}^{(j)},d_k^{(j)}=0$ for all $0\leq k<\lambda_{f-1-j,2}$. Assume $\widetilde{w}_{f-1-j}=\mathfrak{t}_{(m,n)}$. As in \eqref{raising weight for monodromy eq}, we obtain the equation that appears in the case where the weight is $(\ell-2\lambda_{f-1-j,2},0)$ and $\widetilde{w}_{f-1-j}=\mathfrak{t}_{m-\lambda_{f-1-j,2},n-\lambda_{f-1-j,2}}$. More precisely, we have
\begin{equation}\label{p lambda}
    \bm{x}_{0}^{(j)}\prod_{n-1\geq j\geq \lambda_{f-1-j,2}}\left(\dfrac{(\mathfrak{a}-m+n)(\mathfrak{a}-m+n-1)\bm{b}_{0}^{(j)}\bm{c}_{0}^{(j)}}{\bm{a}_{0}^{(j)*}\bm{d}_{0}^{(j)*}}-(m-j)(n-j)p\right)+O(p^{N-3\ell+1})\in \mathfrak{p}^{\lambda'}
\end{equation}
 for all $\lambda'\in X$; where $\bm{x}_0^{(j)}=\bm{b}_{0}^{(j)}$ if $\widetilde{w}_{f-1-j}=\mathfrak{t}_{(m,n)}$ with $m>n$; $1$ if $m=n$ or $\alpha_j=0$ and $\bm{c}_{0}^{(j)}$ if $\widetilde{w}_{f-1-j}=\mathfrak{t}_{(m,n)}$ with $m<n$.
Assume for the sake of contradiction that $\lambda'_{f-1-j}>\lambda_{f-1-j}$ for some $\lambda'\in X$ and $0\leq j\leq f-1$. Then $\mathfrak{p}^{(j),\lambda'_{ f-1-j}}=$ 
\begin{equation}\label{p lambda 2}
\left(\frac{(\mathfrak{a}-m+n)(\mathfrak{a}-m+n-1)\bm{b}_{0}^{(j)}\bm{c}_{0}^{(j)}}{\bm{a}_{0}^{(j)*}\bm{d}_{0}^{(j)*}}-(m-\lambda'_{f-1-j,2)})(n-\lambda'_{f-1-j,2})p+O(p^{N-3\ell+1})\right)+I^{(j), \reg}.
\end{equation}
As $\lambda'_{f-1-j}>\lambda_{f-1-j}$, $\lambda'_{f-1-j,2}<\lambda_{f-1-j,2}$. Considering \cref{p lambda} and \cref{p lambda 2}, since $N\geq 4l$, $\mathfrak{a}-m+n, a-m+n-1, m-n+j, \bm{a}_{0}^*, \bm{d}_{0}^*$ are units for all $\ell\geq m,n,j\geq0$, we deduce that $p^{k}\in \mathfrak{p}^{\lambda'}$ for some $k$, and therefore $p\in \mathfrak{p}^{\lambda'}$, which is a contradiction. The proof for the case where $\widetilde{w}_{f-1-j}=\mathfrak{w}\mathfrak{t}_{(m,n)}$ is analogous. Therefore, this completes the proof.
\end{proof}
\begin{corollary}\label{normal domain}
    Assume that $\overline{\rho}$, up to twisting by a power of $\omega_f$, is of the form in \cref{condition on rho bar} and $4\ell$-generic, $\lambda$ are Hodge--Tate weights with $0<\lambda_{j,1}-\lambda_{j,2}\leq \ell$ and $\tau$ is a $(2\ell+2)$-generic inertial type. If $W(\overline{\rho})\cap \JH(\overline{\sigma}(\lambda,\tau))\neq \emptyset$, then
    \[R^{\lambda,\tau}_{\overline{\rho}}\cong \mathcal{O}\llbracket(x_j,y_j)_{j\in \mathcal{K}}, Z_1,\ldots, Z_{f-m+4}\rrbracket/(x_jy_j-p)_{j\in \mathcal{W}},\]
    where $x_j$(resp. $y_j$) corresponds to $\bm{b}_0^{(j)}\in R$ (resp. $ \bm{c}_0^{(j)}\in R$) if $\widetilde{w}_{f-1-j}=\mathfrak{t}_{(m,n)}$ and $\bm{c}^{(j)}_0-[\overline{\bm{c}}_0^{(j)}]\in R$ (resp. $ \bm{d}^{(j)}_0-[\overline{\bm{d}}_0^{(j)}]\in R$) if $\widetilde{w}_{f-1-j}=\mathfrak{w}\mathfrak{t}_{(m,n)}$ up to multiplication by an element in $R^\times$. (Recall $R=\widehat{\otimes}R^{(j)}$ where $R^{(j)}$ are given in \cref{Table 1} and \cref{Table 2}). In particular, $R^{\lambda,\tau}_{\overline{\rho}}$ is a normal domain and a complete intersection ring. Moreover, the special fibre $\overline{R}^{\lambda,\tau}_{\overline{\rho}}$ is reduced. \par
    Furthermore, the irreducible components of the special fibre of $\overline{R}^{\lambda,\tau}_{\overline{\rho}}$ are given by \[\overline{R}^{\sigma}_{\overline{\rho}}=\overline{R}^{\lambda,\tau}_{\overline{\rho}}/(z(\sigma)_{j\in \mathcal{W}})\cong \F\llbracket (\widetilde{z}_j)_{j\in \mathcal{W}}, \dots, Z_1, \dots, Z_{f-m-4}\rrbracket\]
    for all $\sigma\in W(\overline{\rho})\cap \JH(\overline{\sigma}(\lambda, \tau))$, and $R^{\lambda,\tau}_{\overline{\rho}}$ in the middle term is identified via the isomorphism above.
    In particular, all the irreducible components of $\overline{R}^{\lambda,\tau}_{\overline{\rho}}$ are formally smooth over $\F$.
\end{corollary}
\begin{proof}
We will first reduce it to the case where $\lambda_{j,2}>0$ such that $\lambda_j\leq (\ell_j,0)$ for all $j$ where $\ell_j$ is some positive integer. By \cite[Lemma~5.1.6]{GeneralSerreWeight}, this can always be achieved via twisting by a crystalline character $\psi$ where $\overline{\psi}|_{I_K}$ is a power of $\omega_f$.\par
By \cref{tame type criterion}, if $W(\overline{\rho})\cap \JH(\overline{\sigma}(\lambda,\tau))\neq \emptyset$, then $R^{\lambda,\tau}_{\overline{\rho}}\neq0$.
By \cref{deformation ring}, we know that 
\[{R}_{\overline\rho}^{\lambda,\tau }\llbracket X_1, \ldots, X_{2f}\rrbracket\cong(\mathcal{O}\llbracket(x'_j,y'_j, u_j, v_j, (Z^{k}_j)_{k=1}^{M_j})_{j=1}^f\rrbracket/(\mathfrak{p}^{\lambda}, I^{\reg}_\infty))\llbracket Y_1, \ldots, Y_4\rrbracket\]
for some positive integer $ M_j$. Here $x'_j, y'_j$ corresponds to $\bm{b}_0^{(j)}, \bm{c}_0^{(j)}-[\overline{\bm{c}}_0^{(j)}]\in R^{(j)}$ (resp. $\bm{a}_0^{(j)}, \bm{d}_0^{(j)}-[\overline{\bm{d}}_0^{(j)}]\in R^{(j)}$) and $u_j, v_j$ corresponds to $x_{11}^{(j)}, x_{22}^{(j)}$ (resp. $x_{12}^{(j)}, x_{21}^{(j)}$) if $\widetilde{w}=\mathfrak{t}_{(m,n)}$ (resp. if $\widetilde{w}=\mathfrak{w}\mathfrak{t}_{(m,n)}$), and $Z_{j,k}$ corresponds to the rest of $\bm{a}_{-k}^{(j)}, \bm{b}_{-k}^{(j)}, \bm{c}_{-k}^{(j)}, \bm{d}_{-k}^{(j)}$ for $k>1$. Note that $\mathfrak{p}^{(j),\lambda_{f-1-j}}+I^{(j),\reg}$ are generated by equations of the form $Z_{j,k}+\beta_k+\gamma_k p$, where $\beta_k$ is divisible by $x'_j, y'_j, u_j$ or $v_j$ and $W_j-\alpha_j^*p$ where $W_j=x'_j, y'_j$ or $x'_jy'_j$ and $\alpha^*_j$ is a unit. Let $A=\mathcal{O}\llbracket(x_j,y_j, u_j, v_j)_{j=1}^f, (Z_{j,i})_{1\leq i\leq M, i\neq k}\rrbracket$ with maximal ideal $\mathfrak{m}$, then $Z_{j,k}+\beta_k+\gamma_k p\in A\llbracket Z_k\rrbracket$. By the Weierstrass preparation theorem, we have $Z_{j,k}+\beta_k+\gamma_k p=uf(Z_{j,k})$ where $u$ is a unit in $A\llbracket Z_{j,k}\rrbracket$ and $f$ is a distinguished polynomial in $A[t]$ of some degree $d$. Reducing modulo $\mathfrak{m}$, we have $Z_{j,k}^d \equiv \overline{u} Z_{j,k}^d \mod \mathfrak{m}$. Therefore, we must have $d=1$, and $(Z_{j,k}+\beta_k+\gamma_k p)=(Z_{j,k}+\delta_k )$ in $A\llbracket Z_j^k\rrbracket$ with $\delta_j\in\mathfrak{m}$. We can therefore eliminate the variables $Z_{j,k}$. Similarly, we can eliminate $W_j$ if $W_j=x'_j, y'_j$. If $W_j=x_jy_j$, then let $x_j=x'_j$ and $y_j=y'_j(\alpha^*_j)^{-1}$, and hence we have $x_jy_j=p$. By \cite[Theorem~4]{power-invariant}, if $R, S$ are quasi-local, then $R\llbracket x\rrbracket\cong S\llbracket x\rrbracket$ implies $R\cong S$; hence we can cancel the variables $X_1, \ldots X_{2f}$ with $\{u_j, v_j\}_{j=1}^f$. 
By definition, $R^{\lambda,\tau}_{\overline{\rho}}$ is reduced and irreducible by \cref{deformation ring}, and hence a domain. By \cite[Lemma~3.3.1]{BHHMS}, $x_jy_j-p$ are irreducible; therefore $(x_jy_j-p)_{j=1}^N$ is a regular sequence and $R^{\lambda,\tau}_{\overline{\rho}}$ is a complete intersection ring. Therefore, it is Cohen--Macaulay and $s_k$ holds for all $k$. Given a height-$1$ prime $\mathfrak{p}$, if $\varpi\notin \mathfrak{p}$, then by the description of $R^{\lambda,\tau}_{\overline{\rho}}$ in \cref{deformation ring}, $R^{\lambda,\tau}_{\overline{\rho}}$ is nonsingular at $\mathfrak{p}$. Moreover, $R^{\lambda,\tau}_{\overline{\rho}}[\frac{1}{p}]$ is regular, by \cite[Theorem~3.3.8]{Kisin}. Therefore, $R_1$ is satisfied and $R^{\lambda,\tau}_{\overline{\rho}}$is normal. 
The last statement follows from taking the modulo $\varpi$ that
\[\overline{R}^{\lambda,\tau}_{\overline{\rho}}\cong\F\llbracket(x_i,y_i)_{i=1}^{m}, z_1, \dots, z_{f-m+4}\rrbracket/(x_iy_i)_{i=1}^m=\prod_{\substack{1\leq i\leq m\\ \widetilde{x}_i=x_i \text{ or }y_i}}\F\llbracket \widetilde{x}_1,\dots \widetilde{x}_m, z_1, \dots, z_{f-m+4}\rrbracket.\]
For the identification of components, we closely follow \cite[\S~3.6]{LatticeforGL3}. We have the same canonical diagram as \cite[Diagram (3.9)]{LatticeforGL3} with appropriate modification, for example, the rank of $\mathfrak{M}$ is $2$ instead of $3$, we have $\lambda$ instead of $\eta$ etc. In particular, \cite[Corollary~3.6.3]{LatticeforGL3} is still valid. We have $$\overline{R}_{\overline{\rho}}^{\lambda,\tau}\llbracket X_1, \dots, X_{2f}\rrbracket\cong \bigotimes_{j}R^{(j)}/(p^{(j),\lambda_{f-1-j}},p)\llbracket Y_1, \dots, Y_4\rrbracket.$$ Therefore, it suffices to match the components of $\bigotimes_{j}R^{(j)}/(p^{(j),\lambda_{f-1-j}},p)$ with $W(\overline{\rho})\cap\JH(\overline{\sigma}(\lambda,\tau))$. Notice that $x_j=0$ (resp. $y_j=0$) if and only if $x'_j=0$ (resp. $y'_j=0$) in the notation above.
As explained in \cite[\S~3.6]{LatticeforGL3}, given a Serre weight $\sigma\in W(\overline{\rho})\cap \JH(\overline{\sigma}(\lambda, \tau))$, we first find a minimal type $\tau'$, such that $W(\overline{\rho})\cap \JH(\overline{\sigma}(\lambda, \tau'))=\{\sigma\}$, then by \cref{BreuilMezard}, $\overline{R}_{\overline{\rho}}^{\lambda,\tau}=\overline{R}^{\sigma}$. Using the same calculation as in \cref{Xrholambda} and \cref{tame type criterion}, assume $\lambda_j=(m,n)$ and $\sigma=F(\mathfrak{t}_{\mu-\eta}(b_j))$ where $b_j\in\{0, \sgn(s_j)\}$, we see that $\tau'=\tau_{\widetilde{w}}$ where $\widetilde{w}'_{f-1-j}=\mathfrak{t}_{(m,n)}$ if $b_j=0$ and $\widetilde{w}'_{f-1-j}=\mathfrak{t_{(n,m)}}$ if $b_j=\sgn(s_j)$. In this case, 
\[R^{(j)}/(p^{(j),\lambda_{f-1-j}},\varpi)\cong\F\llbracket z^{(j)}, \bm{a}_{0}^{*(j)}, \bm{d}_{0}^{*(j)}\rrbracket,\]
where $z^{(j)}=y'_j=\bm{c}_0^{(j)}$ if $b_j=0$ and $z_j=x'_j=\bm{b}_0^{(j)}$ if $b_j\neq 0$. We now carry out a similar calculation as in \cite[\S~3.6]{LatticeforGL3}. By \cref{tame type criterion}, we can assume $\tau=\tau_{\widetilde{w}}$ with $\widetilde{w}=\widetilde{w}'\widetilde{z}$. Assume $\widetilde{w}_{f-1-j}=\mathfrak{w}\mathfrak{t}_{(a,b)}$. If $\widetilde{w}'_{f-1-j}=\mathfrak{t}_{(m,n)}$ (i.e.\ , $b_j=0$), then $\widetilde{z}_{f-1-j}=\mathfrak{w}\mathfrak{t}_{(a-n, b-m)}$. Note that $A^{(f-1-j)}$ now has entries in characteristic $p$. By \cref{conjectured solution} and \cref{deformation ring}, ${A'}^{(f-1-j)}\mod \bm{b}_o'=\begin{pmatrix}
    v^m\bm{a}_{0}^{*'}&0\\v^m\bm{c}_{0}'&v^n\bm{d}_{0}^{*'}
\end{pmatrix}$. Similarly, by \cref{conjectured solution 2} and \cref{deformation ring}, $A^{(f-1-j)}\mod \bm{a}_{0}=\begin{pmatrix}
    0&v^b\bm{b}_{0}^*\\v^a\bm{c}_{0}^*& v^b \bm{d}_{0}
\end{pmatrix}$. (Here $'$ is used to denote those in ${A'}^{(f-1-j)}$ given by $\widetilde{w}'$ and we omit the index $(j)$ for legibility.) By setting
\[\begin{pmatrix}
    0&v^b\bm{b}_{0}^*\\v^a\bm{c}_{0}^*& v^b \bm{d}_{0}
\end{pmatrix}=\begin{pmatrix}
    v^m\bm{a}_{0}^{*'}&0\\v^mx'_{21}&v^n\bm{d}_{0}^{*'}
\end{pmatrix}\begin{pmatrix}
    0&v^{b-m}\\v^{a-n}&0
\end{pmatrix},\]
we deduce the following identification:
\[\bm{b}_0^*={\bm{a}_0^*}'\;\;\bm{c}_{0}^*={\bm{d}_{0}^*}'\;\;\bm{c}_{0}'=\bm{d}_{0}.\]
Similarly, if $\widetilde{w}'_{f-1-j}=\mathfrak{t}_{(n,m)}$ (i.e.\ , $b_j\neq0$), then $\widetilde{z}_{f-1-j}=\mathfrak{w}\mathfrak{t}_{(a-m, b-n)}$ and ${A'}^{(f-1-j)}\mod \bm{c}'_{0}=\begin{pmatrix}
    v^n\bm{a}_{0}^{*'}&v^{m-1}x'_{12}\\0&v^m\bm{d}_{0}^{*'}
\end{pmatrix}$ and $A^{(f-1-j)}\mod \bm{d}_{0}=\begin{pmatrix}
   v^{a-1} \bm{a}_{0}&v^b\bm{b}_{0}^*\\v^a\bm{c}_{0}^*& 0
\end{pmatrix}$. By setting
\[\begin{pmatrix}
     v^{a-1} \bm{a}_{0}&v^b\bm{b}_{0}^*\\v^a\bm{c}_{0}^*& 0
\end{pmatrix}=\begin{pmatrix}
    v^n\bm{a}_{0}^{*'}&v^{m-1}\bm{b}'_0\\0&v^m\bm{d}_{0}^{*'}
\end{pmatrix}\begin{pmatrix}
    0&v^{b-n}\\v^{a-m}&0
\end{pmatrix},\]
we deduce the following identification:
\[\bm{b}_{0}^*={\bm{a}_{0}^*}'\;\;\bm{c}_{0}^*={\bm{d}_{0}^*}'\;\;\bm{b}_{0}'=\bm{a}_{0}.\]
By the same argument, it is straightforward to see that when $\widetilde{w}_{f-1-j}=\mathfrak{t}_{(a,b)}$, ${\bm{a}_{0}^*}'=\bm{a}_{0}^*$, ${\bm{d}_{0}^*}'=\bm{d}_{0}^*$, $\bm{b}_{0}'=\bm{b}_{0}/\bm{c}_{0}'=\bm{c}_{0}$ where appropriate. Therefore, the last statement follows.
\end{proof}
We write $\mathcal{W}=\{j\in \mathcal{J}\colon F(\mathfrak{t}_{\mu-\eta}(0,\dots,0, \sgn(s_j),0 \dots,0))\in W(\overline{\rho})\cap \JH(\overline{\sigma}(\lambda, \tau))\}$ (cf. \cref{prop on W(r)}). Then $m=|\mathcal{W}|$ is the positive integer such that $2^m=|W(\overline{\rho})\cap \JH(\overline{\sigma}(\lambda, \tau))|$.
Given a Serre weight $\sigma=F(\mathfrak{t}_{\mu-\eta}(b_j))\in W(\overline{\rho})\cap \JH(\overline{\sigma}(\lambda, \tau))$ where $b_j\in\{0, \sgn(s_j)\}$ (cf. \cref{prop on W(r)}), under the isomorphism of \cref{normal domain}, we define 
    \[z(\sigma)_{j}=\begin{cases}
        x_{j} &\text{ if } b_j=0,\\ y_{j} &\text{ if } b_j=\sgn(s_j).
    \end{cases}\]
    And we define $\widetilde{z}(\sigma)_{j}$ analogously, with $x_j$, $y_j$ swapped.
\begin{table}[h!]
\caption{ $\widetilde{w}_{f-1-j}=\mathfrak{t}_{(m,n)}$}

    \[
    \begin{array}{|c|c|c|}
\hline\rule{0mm}{7mm}
\overline{A}^{(f-1-j)}&     \multicolumn{2}{|c|}{\begin{pmatrix}
\alpha_jv^m& 0\\
\alpha_j \gamma_{f-1-j}v^m&\beta_j v^{n}\end{pmatrix}}\\[3mm]
\hline\rule{0mm}{4mm}
 \multirow{2}{*}{shape} &\gamma_{f-1-j}\neq0, m\leq n &\mathfrak{w}\mathfrak{t}_{(m,n)}\\
 \cline{2-3}\rule{0mm}{4mm}
 & \text{otherwise}&\mathfrak{t}_{(m,n)}\\
\hline\rule{0mm}{7mm}
A^{(f-1-j)}      &\multicolumn{2}{|c|}{\begin{pmatrix}
    \sum_{0\leq i\leq m} \bm{a}_{-(m-i)} (v+p)^{i} &\sum_{0\leq i\leq n-1} \bm{b}_{-(n-1-i)}  (v+p)^{i}\\
    v(\sum_{0\leq i\leq m-1} \bm{c}_{-(m-1-i)}  (v+p)^{i}) &\sum_{0\leq i\leq n} \bm{d}_{-(n-i)}  (v+p)^{i}
\end{pmatrix}}\\[4mm]
\hline\rule{0mm}{6mm}
      R^{(j)} &   \multicolumn{2}{|c|}{\mathcal{O}\llbracket x_{11}, x_{12}, x_{22}, x_{21}, (\bm{a}_{-k})_{k=1}^{m},(\bm{b}_{-k})_{k=1}^{n-1}, (\bm{c}_{-k})_{k=1}^{m-1}, (\bm{d}_{-k})_{k=1}^{n}\rrbracket} \\[1.5ex]
      \hline\rule{0mm}{6mm}
      I^{(j),\reg}   &  \multicolumn{2}{|c|}{\text{For $0\leq k\leq m$, }   \bm{a}_{-k}- \dfrac{(-1)^{k} \bm{a}_0}{k!\prod_{i=0}^{k-1}(\mathfrak{a}'+i)}\prod_{i=0}^{k-1}\left(Z+i(m-n-i)p\right)+O(p^{k_j}),}\\
       &  \multicolumn{2}{|c|}{ \text{For $0\leq k\leq n-1$, } \bm{b}_{-k}-\dfrac{\bm{b}_0}{k!\prod_{i=0}^{k-1}(\mathfrak{a}'-i)}\prod_{i=1}^{k}\left(Z-i(m-n+i)p\right)+O(p^{k_j}),} \\
      &  \multicolumn{2}{|c|}{\text{For $0\leq k\leq m-1$, } \bm{c}_{-k}- \dfrac{(-1)^{k} \bm{c}_0}{k!\prod_{i=0}^{k-1}(\mathfrak{a}'+i-1)}\prod_{i=1}^{k}\left(Z-i(n-m+i)p\right)+O(p^{k_j}),}\\
      &  \multicolumn{2}{|c|}{\text{For $0\leq k\leq n$, } \bm{d}_{-k}-\dfrac{\bm{d}_0}{k!\prod_{i=0}^{k-1}(\mathfrak{a}'-1-i)}\prod_{i=0}^{k-1}\left(Z+i(n-m-i)p\right) +O(p^{k_j}),}\\[4mm]
     & \multicolumn{2}{|c|}{ (\bm{x}_0+O(p^{k_j}))\prod_{n\geq j\geq 1}\left(Z-(m-n+j)jp+O(p^{k_j})\right)}\\[2mm]
         \hline\rule{0mm}{5mm}
          I^{(j)}& m=n & (\bm{b}_0+O(p^{k_j}))\prod_{n\geq j\geq 1}(Z-(m-n+j)jp+O(p^{k_j}))\\
          & & (\bm{c}_0+O(p^{k_j}))\prod_{n\geq j\geq 1}(Z-(m-n+j)jp+O(p^{k_j}))\\[2mm]
         \hline\rule{0mm}{5mm}
         \multirow{2}{*} {$\mathfrak{p}^{(j), \lambda_{f-1-j}}$}&\lambda_{f-1-j}=(m,n), n<m& I^{(j), \reg}+\bm{b}_0+O(p^{k_j})\\[1mm]
        \cline{2-3}\rule{0mm}{5mm}
       &\text{ otherwise}  & I^{(j), \reg}+(Z-(m-\lambda_{f-1-j,2})(n-\lambda_{f-1-j,2})p+O(p^{k_j})) \\[1mm]
         \hline\rule{0mm}{4mm}
     \multirow{2}{*}{$z(\sigma)_j$} &b_j=0&  \bm{b}_0 \\
     \cline{2-3}\rule{0mm}{3mm}
    &b_j=\sgn(b_j)&\bm{c}_0\\
    \hline
         \end{array}\]        
    \label{Table 1}
    \raggedright
Here, we define $Z\colonequals \frac{(\mathfrak{a}')(\mathfrak{a}'-1)\bm{b}_0\bm{c}_0}{\bm{a}_0^*\bm{d}_0^*}$ and $\bm{x}_0=\bm{b}_0$ if $m>n$, $\bm{c}_0$ if $m<n$ and $1$ if $m=n$. Recall that $O(p^{k_j})$ denotes a specific but inexplicit element in $p^{N-3\ell_{f-1-_j}+1}M_2(R)$, it depends on the whole tuple $\widetilde{w}$, not just $\widetilde{w}_{f-1-j}$. Moreover, $\mathfrak{a}'=\mathfrak{a}-m+n$, $\mathfrak{a}\in \Z_{(p)}$ and $\mathfrak{a}\equiv -\langle s_j^{-1}(\mu_j)-(m,n), \alpha_j^{\vee} \rangle\equiv -\sgn(s_j)(r_j+1)+(m-n) \mod p$. For readability, we remove the superscript $(j)$. Furthermore, $x_{11}=\bm{a}_{0}^*-[\overline{\bm{a}}_{0}^*]$, $x_{12}=\bm{b}_{0}$, $x_{21}=\bm{c}_{0}-[\overline{\bm{c}}_{0}]$, and $x_{22}=\bm{d}_{0}^*-[\overline{\bm{d}}_{0}^*]$. Here, $\sigma=F(\mathfrak{t}_{\mu-\eta}(b_j))$ where $b_j\in \{0, \sgn (s_j)\}$ as in \cref{prop on W(r)}
\end{table}
\newpage
\begin{table}[h!]
\caption{ $\widetilde{w}_{f-1-j}=\mathfrak{w}\mathfrak{t}_{(m,n)}$}
    \centering
     \[\begin{array}{|c|c|c|}
\hline\rule{0mm}{7mm}
\overline{A}^{(f-1-j)}&     \multicolumn{2}{|c|}{\begin{pmatrix}
    0&\alpha_{j}v^{n}\\
    \beta_{j}v^m&\alpha_{j}\gamma_{f-1-j}v^{n}
    \end{pmatrix}}\\[3mm]
\hline\rule{0mm}{4mm}
 \multirow{2}{*}{shape} &\gamma_{f-1-j}\neq0, m> n &\mathfrak{t}_{(m,n)}\\[2mm]
 \cline{2-3}\rule{0mm}{4mm}
 & \text{otherwise}& \mathfrak{w}\mathfrak{t}_{(m,n)}\\[2mm]
\hline\rule{0mm}{7mm}
A^{(f-1-j)}      &\multicolumn{2}{|c|}{\begin{pmatrix}
    \sum_{0\leq i\leq m-1} \bm{a}_{-(m-i)} (v+p)^{i} &\sum_{0\leq i\leq n} \bm{b}_{-(n-i)}  (v+p)^{i}\\
    v(\sum_{0\leq i\leq m-1} \bm{c}_{-(m-1-i)}  (v+p)^{i}) &\sum_{0\leq i\leq n} \bm{d}_{-(n-i)}  (v+p)^{i}
\end{pmatrix}}\\[4mm]
\hline\rule{0mm}{6mm}
      R^{(j)} &   \multicolumn{2}{|c|}{\mathcal{O}\llbracket x_{11}, x_{12}, x_{22}, x_{21}, (\bm{a}_{-k})_{k=1}^{m-1},(\bm{b}_{-k})_{k=1}^{n}, (\bm{c}_{-k})_{k=1}^{m-1}, (\bm{d}_{-k})_{k=1}^{n-1}\rrbracket} \\[1.5ex]
      \hline\rule{0mm}{6mm}
      I^{(j),\reg}   &  \multicolumn{2}{|c|}{\text{For $0\leq k\leq m-1$, }   \bm{a}_{-k}-\dfrac{(-1)^{k}\bm{a}_0}{k!\prod_{i=0}^{k-1}(\mathfrak{a}'+i)}\prod_{i=1}^{k}\left(Z+(\mathfrak{a}-i)(\mathfrak{a}'+i)p\right)+O(p^{k_j}),}\\
       &  \multicolumn{2}{|c|}{ \text{For $0\leq k\leq n$, } \bm{b}_{-k}-\dfrac{\bm{b}_0}{k!\prod_{i=0}^{k-1}(\mathfrak{a}'-i)}\prod_{i=0}^{k-1}\left(Z+(\mathfrak{a}+i)(\mathfrak{a}'-i)p\right)+O(p^{k_j}),} \\
      &  \multicolumn{2}{|c|}{\text{For $0\leq k\leq m-1$, }  \bm{c}_{-k}-\dfrac{(-1)^{k}\bm{c}_0}{k!\prod_{i=0}^{k-1}(\mathfrak{a}'+1+i)}\prod_{i=1}^{k}\left(Z+(\mathfrak{a}-i)(\mathfrak{a}'+i)p\right)+O(p^{k_j}),}\\
      &  \multicolumn{2}{|c|}{\text{For $0\leq k\leq n$, } \bm{d}_{-k}-\dfrac{\bm{d}_0}{k!\prod_{i=0}^{k-1}(\mathfrak{a}'+1-i)}\prod_{i=0}^{k-1}\left(Z+(\mathfrak{a}+i)(\mathfrak{a}'-i)p\right)+O(p^{k_j}),}\\[4mm]
     & \multicolumn{2}{|c|}{ \prod_{0\leq k\leq n-1}\left(Z- (\mathfrak{a}'-k)(\mathfrak{a}+k) p+O(p^{k_j})\right)}\\[2mm]
         \hline\rule{0mm}{5mm}
         \mathfrak{p}^{(j), \lambda_{f-1-j}} &\multicolumn{2}{|c|}{I^{(j), \reg}+\left(Z- (\mathfrak{a}-m+\lambda_{f-1-j,2})(\mathfrak{a}+n-\lambda_{f-1-j,2}) p+O(p^{k_j})\right)} \\[2mm]
         \hline\rule{0mm}{4mm}
    \multirow{2}{*}{$z(\sigma)_j$} &b_j=0&  \bm{a}_0 \\[1mm]
    \cline{2-3}\rule{0mm}{3mm}
    &b_j=\sgn(b_j)&\bm{d}_0\\[2mm]
    \hline
         \end{array}\]       
    \label{Table 2}
    \raggedright
Here we let $Z\colonequals \frac{(\mathfrak{a}')(\mathfrak{a}'+1)\bm{a}_0\bm{d}_0}{\bm{b}_0^*\bm{c}_0^*}$ and $O(p^{k_j})$ denotes a specific but inexplicit element in $p^{N-3\ell_{f-1-_j}+1}M_2(R)$, it depends on the whole tuple $\widetilde{w}$, not just $\widetilde{w}_{f-1-j}$. Moreover, $\mathfrak{a}'=\mathfrak{a}-m+n$, $\mathfrak{a}\in \Z_{(p)}$ and $\mathfrak{a}\equiv -\langle \mathfrak{w}s_j^{-1}(\mu_j)-(m,n), \alpha_j^{\vee} \rangle\equiv \sgn(s_j)(r_j+1)+(m-n) \mod p$. For readability, we remove the superscript $(j)$. Furthermore, $x_{11}=\bm{a}_0$, $x_{12}=\bm{b}_{0}^*-[\overline{\bm{b}}_{0}^*]$, $x_{21}=\bm{c}_{0}^*-[\overline{\bm{c}}_{0}^*]$, and $x_{22}=\bm{d}_{0}-[\overline{\bm{d}}_{0}]$. Here, $\sigma=F(\mathfrak{t}_{\mu-\eta}(b_j))$ where $b_j\in \{0, \sgn (s_j)\}$ as in \cref{prop on W(r)}
\end{table}
\clearpage
\section{Patching functor}
\label{ch4:patching_functor}
\subsection{Abstract patching functor}
Let $S$ be a finite set, which for our purpose will be places dividing $p$. For each $v\in S$, we fix a local field $L_v$ with ring of integer $\mathcal{O}_{L_v}$ the residue field $k_v$, such that $L_v$ is a finite extension of $\Q_p$. We let $K\colonequals\prod_{v\in S}\GL_2(\mathcal{O}_{L_v})$. We let $\overline{\rho}_v\colon G_{L_v}\to \GL_2(\F)$. We define $\mathcal{C}$ to be the category of finitely generated $\mathcal{O}$-algebras with a continuous action of $K$ and let $C'$ be a Serre subcategory of $C$.\par
We define
\[R_\infty\colonequals   
\widehat{\bigotimes}_{v\in S}R^{\square}_{\overline{\rho}_{v}}\llbracket x_1, x_2, \dots, x_h\rrbracket \]
for some positive integer $h\geq |S|-1$, with maximal ideal $\mathfrak{m}_{\infty}$.
Given $\overline{\sigma}$ which is killed by $\varpi_E$, assume that $\overline{\sigma}|_{K_v}$ is a direct sum of $\overline{\sigma}_v$, which is an irreducible representation of $K_v$ over $\F$. Then we let $R^{\overline{\sigma}}_{\overline{\rho}_v}$ be the crystalline deformation of Hodge type $\overline{\sigma}_v$.
We let 
\begin{equation}\label{def of Rtau}
    R_\infty^{\lambda,\tau}\colonequals R_\infty\widehat{\otimes}_{R^{\square}_{\overline{\rho}_{v}}}\widehat{\bigotimes}_{v\in S} R^{\lambda_v, \tau_{v}}_{\overline{\rho}_{v}};\;\;R^{\overline{\sigma}}_\infty\colonequals R_\infty\widehat{\otimes}_{R^{\square}_{\overline{\rho}_{v}}}\widehat{\bigotimes}_{v\in S}R^{\overline{\sigma}_v}_{\overline{\rho}_v}.
\end{equation}
We write $X_\infty\colonequals \Spf R_{\infty}$ and analogously $X_\infty(\tau)\colonequals \Spf R_\infty( \eta,\tau), X_\infty(\lambda, \tau)\colonequals \Spf R_\infty^{\lambda,\tau}$, ${X}_\infty(\overline{\sigma})\colonequals \Spf R^{\overline{\sigma}}_\infty$ and we write $\overline{X}_{\infty}(\overline{\sigma})$ for the special fibre of the space ${X}_\infty(\overline{\sigma})$.\par
Let $\mathcal{S}$ be a set of tuples of pairs $(\lambda_v,\tau_v)_{v\in S}$ such that $\lambda_v$ is a Hodge type and $\tau_v$ is a inertial type. If $\tau=(\tau_v)_{v\in S}$ is a collection of tame inertial types, then $\sigma(\tau_v)$ is a representation of $\GL_2(\mathcal{O}_{L_v})$ over $\mathcal{O}$ corresponding to $\tau_v$ by the local Langlands correspondence.  Recall that $V(\lambda)$ is the algebraic representation with highest weight $\lambda$. For any $\mathcal{O}$-lattice $\sigma^\circ(\lambda_v, \tau_v)$ in $\sigma(\tau_v)\otimes V(\lambda_v-\eta)$, and we write $\sigma^\circ(\lambda,\tau)\colonequals \otimes_{v\in S} \sigma^\circ(\lambda_v,\tau_v)$. Then, following \cite[Definition 4.1.1]{GeneralSerreWeight}), a \emph{patching functor} $M_{\infty}$ for $\mathcal{S}$ is a nonzero covariant exact functor from $C'$ to the category of coherent sheaves over $\Spf R_{\infty}$, such that for $(\lambda,\tau)\in \mathcal{S}$,
\hfill\begin{enumerate}
    \item  $M_\infty(\sigma^\circ(\lambda,\tau))$ is $p$-torsion free and is a maximal Cohen--Macaulay sheaf on $X_\infty(\lambda,\tau)$.
    \item For all $\overline{\sigma}\in \JH(\overline{\sigma}(\tau))$, $M_\infty(\overline{\sigma})$ is a maximal Cohen--Macaulay sheaf on $\overline{X}_\infty(\overline{\sigma})$.
\end{enumerate}
We say that $M_\infty$ is a \emph{minimal patching functor} if the locally free sheaf $M_\infty(\sigma^\circ(\lambda,\tau))[\frac{1}{p}]$ has rank at most one on each connected component. We say a patching functor is \emph{unramified} if the coefficient field is unramified over $\Q_p$. \par
\subsection{Global setup}\label{patching}
For the global setup, we will construct a minimal patching functor by unitary groups, following closely \cite{CEG} and \cite{extremal}.
Let $F$ be a CM field with maximal totally real subfield $F^+$. We call a place in $F^+$ split (resp. inert) if it splits (resp. is inert) in $F$.
We denote $S_p$ for the set of primes of $F^+$ lying above $p$. Let $\Sigma$ be the set of primes of $F^+$ away from $p$ where $\overline{r}$ ramifies. \par
Let $\mathcal{O}_{F^+}, \mathcal{O}_{F_V^+}$ and $ \mathcal{O}_{F_w}$ denote the ring of integers of $F^+$, $F_v^+$ and $F_w$ respectively, where $v$ is a place of $F^+$ and $w$ a place of $F$. Let $G_{/F^+}$ be a reductive group which is an outer form for $\GL_2$ such that
\hfill\begin{enumerate}
     \item $G_{/F}$ is an inner form of $\GL_2$;
     \item $G_{/F^+}(F_v^+)\cong U_n(\mathbb{R})$ for all $v|\infty$;
     \item $G_{/F^+}$ is quasisplit at all inert finite places.
 \end{enumerate}
 By \cite[\S~7.1]{EGH}, $G$ admits a reductive model $\mathcal{G}$ over $\mathcal{O}_{F^+}[1/N]$ for some $p\nmid N$ and an isomorphism $\iota\colon \mathcal{G}_{/\mathcal{O}_{F}[1/N]}\to \GL_{2/\mathcal{O}_F[1/N]}$ which specializes to $\iota_w\colon  \mathcal{G}(\mathcal{O}_{F_v^+})\xrightarrow[]{\sim}\mathcal{G}(\mathcal{O}_{F_w})\xrightarrow[]{\iota}\GL_n(\mathcal{O}_{F_w})$ for all split finite places $w$ in $F$ prime to $N$ where $w|_{F^+}=v$.
Let $U=U^pU_p\leqslant G(\mathbb{A}_{F^+}^{\infty, p})\times \mathcal{G}(\mathcal{O}_{F^+,p})$ where $\mathcal{G}(\mathcal{O}_{F^+,p})\colonequals \prod_{v|p}\mathcal{G}(\mathcal{O}_{F_v^+})$ be a compact open subgroup. 
If $W$ is a finitely generated $\mathcal{O}$-module endowed with a continuous action of $U_{\Sigma}\colonequals\prod_{v\in \sigma}U_v$, we define the space of \emph{automorphic forms} on $G$ of level $U$ with coefficients in $W$ to be the $\mathcal{O}$-module $S(U,W)\colonequals$
\[\{f\colon  \text{continuous map }G(F^+)\backslash G(\mathbb{A_{F^+}^\infty})\to W|f(gu)=u_{\Sigma}^{-1}f(g)\text{ for all } g\in G(\mathbb{A}_{F^+}^\infty), u\in U\}.\]
 We say $U$ is \emph{unramified} at $v$, if $U=\mathcal{G}(\mathcal{O}_{F_v^+})U^v$. Let $S$ be the set of finite split places of $F^+$, which is composed of $ S_p\sqcup \Sigma$ and all place $v$ such that $U$ is not unramified. Let $\mathcal{P}_U$ be the set of finite places $w$ such that $v:=w|_{F^+}$ is split in $F$, and does not divide any primes in $S$, nor any prime dividing $N$.
For any subset $\mathcal{P}\subset\mathcal{P}_U$ of finite complement that is closed under complex conjugation, we define
\begin{equation}\label{Hecke algebra}
\mathbb{T}_{\mathcal{P}}=\mathcal{O}[T_w^{(i)}, w\in \mathcal{P}, 0\leq i\leq 2]
\end{equation}
to be the \emph{universal Hecke algebra} on $\mathcal{P}$. Then, $S(U,W)$ is endowed with an action of $\mathbb{T}_{\mathcal{P}}$ that $T_w^{(i)}$ acts by the double coset operator 
\[\iota_w^{-1}\left[\GL_2(\mathcal{O}_{F_w})\begin{pmatrix}
    \varpi_w \Id_i &0\\ 0& \varpi_w\Id_{n-i}
\end{pmatrix}\GL_2(\mathcal{O}_{F_w})\right].\]
\begin{definition}\label{automorphicdef}\cite[Definition~7.1]{potentiallycrystalline}
   Let $\overline{r}\colon G_F\to \GL_2(\F)$ be a continuous Galois representation. Let $\mathfrak{m}\subset\mathbb{T}_{\mathcal{P}}$ for some $\mathcal{P} \subset \mathcal{P}_U$ corresponding to the kernel of the system of eigenvalues $\overline{\alpha}\colon \mathbb{T}_{\mathcal{P}}\to\F$ such that
   \[\det(1-\overline{r}^{\vee}(\Frob_w)X)=\sum_{j=0}^2(-1)^j(\mathbf{N}_{F/\Q}(w))^{\binom{j}{2}}\overline{\alpha}(T_{w}^{(j)})X^j\]
for all $w\in \mathcal{P}$. Then we say $\overline{r}$ is \emph{automorphic} if there exists a compact open subgroup $U\leqslant G(\mathbb{A}_{F^+}^{\infty, p})\times \mathcal{G}(\mathcal{O}_{F^+,p})$, a finite $\mathcal{O}$-module $W$ endowed with a continuous action of $U_\Sigma$ and a cofinite subset $\mathcal{P}\subset \mathcal{P}_U$ such that
\[S(U,W)_\mathfrak{m}\neq 0\]
\end{definition}
 Given a continuous absolutely irreducible representation $\overline{r}\colon G_{F}\to \GL_2(\F)$, we will assume it satisfies the following properties:
 \begin{properties}\label{condition for r}
\hfill\begin{enumerate}
    \item $p$ is unramified in $F^+$ and every place of $F^+$ dividing $p$ splits in $F$.
    \item $F/F^+$ is unramified at all finite places, and hence $[F^+\colon Q]$ is even.
    \item $\overline{F}^{\ker \text{ad} \overline{r}|_{G_F} }$ does not contain $F(\zeta_p)$.
    \item $\overline{r}|_{G_{F(\zeta_p)}}$ is absolutely irreducible.
    \item $\overline{r}$ is automorphic as in \cref{automorphicdef}.
    \item $\overline{r}$ is ramified only at split places and with minimal ramification in the sense of \cite[Definition~2.4.14]{CHT}.
    \item $\overline{r}(G_F)$ contains $\GL_2(\F)$ with $|\F|>6$.
    \item $\overline{r}(G_{F(\zeta_p)})$ is adequate (in the sense of \cite[Definition~2.3]{Tho}).
    \item For all places $\widetilde{v}|p$ of $F$, $\overline{r}|_{G_{F_{\widetilde{v}}}}$, up to twisting by a crystalline character, satisfies \cref{condition on rho bar} for some $N_v$.
\end{enumerate}
\end{properties}
By \cite[\S~2.3]{CEG}, we can find a finite place $v_1\notin S$ such that \hfill\begin{enumerate}
    \item $v_1$ splits as $w_1w_1^c$ in $F$;
    \item $v_1$ does not split completely in $F(\zeta_p)$, i.e $(\mathbf{N}v_1)\neq 1(\mathrm{mod}\; p)$;
    \item $\overline{r}|_{G_{F_{v_1}^+}}$ is unramified and the ratio between the eigenvalues of $\overline{r}(\Frob_{{F_{v_1}}})$ is not equal to $(\mathrm{N}v_1)^{\pm 1}$ or $1$.
\end{enumerate}
\begin{remark}
    By \cite[\S~5]{MR3554238}, this ensures that $R^{\square}_{\overline{r}_{\widetilde{v}_1}}$ is formally smooth over $W(\F)$.
\end{remark}
\begin{definition}\label{minimal level}
We construct a compact subgroup $U=\prod_v U_v$ of $G(\mathbb{A}^\infty_{F^+})$ such that each $U_v$ is a compact subgroup of $G(F_v^+)$ with the following properties:
\hfill\begin{enumerate}\label{multiplicity one}
    \item $U_v= \iota_{\widetilde{v}}^{-1}\mathcal{G}(\mathcal{O}_{F_{\widetilde{v}}})$ if $v$ is a split place in $F$ and $v\neq v_1$.
    \item $U_v$ is a hyperspecial maximal compact subgroup of $\mathcal{G}(F^+_v)$ if $v$ is inert in $F$.
    \item $U_{v_1}$ is the preimage of the upper triangular matrices under
    \[\mathcal{G}(\mathcal{O}_{F_{v_1}^+})\xrightarrow[\sim]{{\iota}_{v_1}} \GL_2(\mathcal{O}_{F_{w_1}})\xrightarrow[]{\mathrm{mod}\;p}\GL_2(k_{w_1}).\]
\end{enumerate}
\end{definition}
Because of the condition at $v_1$, $U$ is sufficiently small in the sense that for some $v\in F^+$, the projection of $U$ to $G(F^+_v)$ does not contain nontrivial element of finite order (cf. the discussion in \cite[\S~3.1.2]{GK}).
Hence for any $\mathcal{O}$-algebra $A$ and $\mathcal{O}$-module $W$, we have 
\[S(U, W\otimes_\mathcal{O}A)\cong S(U,W)\otimes_\mathcal{O}A.\]
For each $v\in \Sigma$, we fix an inertial type $\tau_{v}$ which is the restriction to the inertia of a minimally ramified lift of $\overline{r}|_{G_{F_{v}^+}}$. By the inertial local Langlands correspondence, we have a finite-dimensional $\GL_2(\mathcal{O}_{F_{v}^+})$-representation $\sigma(\tau_{v}^\vee)$ over $\mathcal{O}$ corresponding to $\tau_{v}^\vee$, and we fix a $\mathcal{O}$-lattice $\sigma(\tau_{v}^\vee)^\circ\subset\sigma(\tau_{v}^\vee)$. Let $W_\Sigma=\bigotimes_{v\in \Sigma}(\sigma(\tau_{v}^\vee)^\circ \circ \iota_v^{-1})$. Let $V$ be any finite $\F$-module with continuous $\prod_{s\in S_p} G(\mathcal{O}_{F_v^+})$-action. Then, the patching functor $V\mapsto M_\infty(V\otimes W_\Sigma)$ is a patching functor for $S=S_p$.
Let $U=U_vU^v$ where $U_v\leqslant\mathcal{G}(\mathcal{O}_{F^+_v})$ is compact open. We define
\[\widehat{S}(U^v,W)\colonequals \varinjlim_{U_v}S(U^vU_v,W) \text{  and  } \widetilde{S}(U^v,W)\colonequals \varprojlim_{n}\widehat{S}(U^v,W/\varpi^n).\]
Let $w_1\in F$ be such that $w_1|_{F^+}=v_1$. We define 
\[\mathbb{T}'_{\mathcal{P}}=\mathbb{T}_{\mathcal{P}}[T_{w_1}^{(1)}, T_{w_1}^{(2)}].\]
Then, $S(U, W), \widehat{S}(U^v,W), \widetilde{S}(U^v,W)$ is endowed with an action of $\mathbb{T}'_{\mathcal{P}}$ such that
$T_{v_1}^{(i)}$ act by the double coset operator:
\[\left[U_{v_1}\iota_{w_1}^{-1}\begin{pmatrix}
    \varpi_{{w_1}} \Id_i &0\\ 0& \Id_{n-i}
\end{pmatrix}U_{v_1}\right].\] 
Label the eigenvalues of $\overline{r}(\Frob_{w_1})$ as $\delta_1, \delta_2$. Let $\mathfrak{m}'$ be the maximal ideal of $\mathbb{T}'_{\mathcal{P}}$ generated by $\mathfrak{m}$ and the elements $T_{w_1}^{(1)}-\delta_1, T_{w_1}^{(2)}-(\mathbf{N}v_1)^{-1}\delta_1\delta_2$. We define $M_{\mathfrak{m}'}\colonequals M\otimes(\mathbb{T}'_\mathcal{P})_{\mathfrak{m}'}$. 
In particular, $ \widetilde{S}(U^v,W)_{\mathfrak{m}'}\colonequals\varprojlim_{n}\varinjlim_{U_v}S(U^vU_v,W/\varpi^n)_{\mathfrak{m}'}$.
\begin{proposition}\label{patching functor existence}
    Given Galois representation $\overline{r}\colon G_ F\to \GL_2(\F)$ satisfying \cref{condition for r}, there exists a minimal patching functor $M_\infty$ for $\mathcal{S}=\{(\lambda,\tau): \tau_v$ is $(2\ell_v+2)$-generic, $N_v\geq 4\ell_v$, where $\ell_v=\lambda_{v_{j,1}}-\lambda_{v_{j,2}}\}$ where $S= S_p$, $L_v=F^+_v\cong F_{\widetilde{v}}$, $\overline{\rho}_v=\overline{r}|_{F_v^+}\cong \overline{r}|_{F_{\widetilde{v}}}$.
\end{proposition}
\begin{proof}
     This follows from \cite[Lemma~A.1.1]{localmodel}.
     For $\overline{\sigma}\in \JH(\overline{\sigma}(\tau_{S_p}))$, $M_\infty(\overline{\sigma})$ is a priori a maximal Cohen--Macaulay sheaf on $\overline{X}_\infty(\tau_{S_p})$. By \cite[Proposition~3.5.1]{EGS}, we can find a tame type $\tau$ such that $\JH(\overline{\sigma}(\tau))\cap W(\overline{r})=\overline{\sigma}$. Then, by \cite[Theorem~7.2.1(2),(4)]{EGS}, the scheme-theoretic support of $ M_\infty(\overline{\sigma})$ is exactly $\overline{X}_\infty(\tau)=\overline{X}_\infty({\overline{\sigma}})$. (\textit{cf.} \cite[\S~B.1]{EGS}) Alternatively, it follows from \cite[Remark 4.1.2]{GeneralSerreWeight}, and noticing that $R_\infty(\lambda,
     \tau)$ is irreducible by \cref{deformation ring}. For minimality, it follows from \cite[Lemma 4.18]{CEG}
\end{proof}
We can further construct a (not necessarily) minimal patching functor for any $S\subseteq S_p$. For each $v\in S_p\setminus S$, we fix a tame inertial type $\tau_{v}$ and regular weights $\lambda_{v}\in (\Z^2)^{f_{v}}$, such that $\ell_v=\max_j\{(\lambda_{v})_{j,1}-(\lambda_{v})_{j,2}\}\leq N_v/4$. We will choose for each $\lambda_v$, a $(2\ell_v+2)$-generic tame inertial type $\tau_v$ such that $W(\overline{r}_v)\cap \JH(\lambda_v,\tau_v)\neq \emptyset$. For the purpose of \cref{ch6: cyclicity of patched module}, we will also construct a minimal patching functor for $S$. In that case, we choose $\tau_v$ such that $W(\overline{r}_v)\cap \JH(\lambda_v,\tau_v)$ is a singleton, which can be arranged using \cref{S(tau)}. We then fix a $\GL_2(\mathcal{O}_{F_{v'}^+})$-invariant lattice $\sigma^0(\lambda_{v}, \tau_{v})$ in $\sigma(\tau_{v})^\vee\otimes V(\lambda_{v}-\eta)$, which can be assumed to be a free $\mathcal{O}$-module. Then $(\sigma^0(\lambda_{v}, \tau_{v}))^d$ is a lattice inside $\sigma(\lambda_v,\tau_v)$. We define $\sigma^{S}\colonequals \otimes_{v'\in S_p\setminus\{v\}}(\sigma^0(\lambda_{v'},\tau_{v'})\circ \iota_{v'}^{-1})$. We simply write $\sigma^v$ if $S=\{v\}$. We then obtain a patching functor
\[M_\infty^{\sigma^S}: V\mapsto M_\infty (V\otimes(\sigma^{S})^d),\]
which we also denote as $M_\infty$ when there is no ambiguity. \par
In the process of constructing a patching functor, we patch together the space of modular forms and obtain $S_\infty=\mathcal{O}\llbracket y_1,\dots y_{4(|S_p|+1)}, z_1, \dots z_{h+|F^+:\Q|}\rrbracket$ with $\mathfrak{a}_\infty$ generated by all the $z_i, y_i$ \cite[\S~2.8]{CEG}.
For $V$ a finitely generated $\mathcal{O}$-module with an action of $\GL_2(\mathcal{O}_{F_v^+})$, $M_\infty^{\sigma^v}(V)$ is a finitely generated module over $S_\infty$. We can relate the patched modules, the spaces of completed cohomology, and the classical algebraic automorphic forms as follows. \begin{proposition}\label{completed cohomology}
Let $V$ be a finitely generated free $\mathcal{O}$-module with a continuous $U_v$ action, then we have the following isomorphism of $\mathcal{O}$-module, which is compatible with the Hecke action.
\[(M_\infty^{\sigma^v}(V)/a_\infty)^d\cong S(U, W_\Sigma\otimes\sigma^v\otimes V^d )_{\mathfrak{m}'}\cong \Hom_{\mathcal{O}\llbracket U_v\rrbracket}(V, \widetilde{S}(U^v, W_\Sigma\otimes\sigma^v)_{\mathfrak{m}'}).\]
\end{proposition}
\begin{proof}
Let $W'=W_\Sigma\otimes\sigma^v$. By \cite[Proposition~3.2.4]{MR2207783}, we have 
    \[\Hom_{\mathcal{O}\llbracket U_v\rrbracket}(V, \widetilde{S}(U^v, W')_{\mathfrak{m}'})[\frac{1}{p}]=\Hom_{U_v}(V,\widetilde{S}(U^v, W')_{\mathfrak{m}'}^{\mathrm{l.alg}})[\frac{1}{p}]\cong S(U_pU^p, W'\otimes V^d)_{\mathfrak{m}'}[\frac{1}{p}] \]
where $\mathrm{l.alg}$ denotes the subspace given by all the locally algebraic representations. \par
Moreover, by \cite[Corollary~2.2.25]{MR2207783}, $ \widetilde{S}(U^v, W')_{\mathfrak{m}'}$ is the same as $\widehat{S}(U^v,W')_{\mathfrak{m}'}$, the $\varpi$-adic completion of $S(U^v, W')_{\mathfrak{m}'}$. 
As each $S(U_vU^v,W')_{\mathfrak{m}'}$ is $\varpi$-torsion free for $U$ sufficiently small, so is $ \widetilde{S}(U^v W')_{\mathfrak{m}'}$. Therefore, we obtain the second isomorphism. \par
Let $W$ be a free $\mathcal{O}$-module with a $U_p$-action. By the construction of the patching functor (\cite[Appendix A]{localmodel}), 
\[M_\infty (W)=\Hom_{\mathcal{O}\llbracket U_p\rrbracket}^{\cont}(M_\infty, W^\vee)^\vee.\]
By \cite[Equation~A.4]{localmodel} and the adjointness of limit, we have
\[M_\infty/\mathfrak{a}_\infty\cong \varprojlim_{K_p\subset U_p,n}S(K_pU^p, W_\Sigma/\varpi^n)_{\mathfrak{m}'}^\vee\cong(\varinjlim_{K_p\subset U_p,n} S(K_pU^p, W_\Sigma/\varpi^n)_{\mathfrak{m}'})^\vee.\]
Therefore,
 \begin{align*}
 M_\infty (W)/\mathfrak{a}_\infty&\cong\Hom_{\mathcal{O}\llbracket U_p\rrbracket}^{\cont}(M_\infty/\mathfrak{a}_\infty, W^\vee)^\vee\\
    &\cong\Hom_{\mathcal{O}\llbracket U_p\rrbracket}^{\cont}((\varinjlim_{K_p\subset U_p,n}S(K_pU^p, W_\Sigma/\varpi^n)_{\mathfrak{m}'})^\vee, W^\vee)^\vee\\
&\cong\Hom_{\mathcal{O}\llbracket U_p\rrbracket}^{\cont}(W,\varinjlim_{K_p\subset U_p,n}S(K_pU^p, W_\Sigma/\varpi^n)_{\mathfrak{m}'})^\vee
\end{align*}
As $W$ is finitely generated, any map $f:W\to \varinjlim_{K_p\subset U_p,n}S(K_pU^p, W_\Sigma/\varpi^n)_{\mathfrak{m}'}$ factors through $f:W\to S(U_pU^p, W_\Sigma/\varpi^n)_{\mathfrak{m}'}$ for some $n$, hence
\[ \Hom_{\mathcal{O}\llbracket U_p\rrbracket}^{\cont}(W,\varinjlim_{K_p\subset U_p,n}S(K_pU^p, W_\Sigma/\varpi^n)_{\mathfrak{m}'})\cong\varinjlim_{n}\Hom_{\mathcal{O}\llbracket U_p\rrbracket}^{\cont}(W,\varinjlim_{K_p\subset U_p,}S(K_pU^p, W_\Sigma/\varpi^n)_{\mathfrak{m}'}).\]
As $U$ is sufficiently small, one can generalize the proof of \cite[Lemma 7.4.1]{EGH} to deduce that we have natural isomorphism
\[\varinjlim_{K_p\subset U_p,}S(K_pU^p, W_\Sigma/\varpi^n)_{\mathfrak{m}'}\cong S^{\mathrm{l.alg}}(U^p,W_\Sigma)_{\mathfrak{m}'}/\varpi^n.\]
which is compatible with the Hecke action and $\GL_2(F_v)$ action (and $S^{\mathrm{l.alg}}(U^p,W_\Sigma)$ is defined in \cite[\S 7.4]{EGH}). Therefore,
\begin{align*}
\Hom_{\mathcal{O}\llbracket U_p\rrbracket}^{\cont}(W,\varinjlim_{K_p\subset U_p,}S(K_pU^p, W_\Sigma/\varpi^n)_{\mathfrak{m}'})&\cong\Hom_{\mathcal{O}\llbracket U_p\rrbracket}^{\cont}(W/\varpi^n,S^{\mathrm{l.alg}}(U^p,W_\Sigma)/\varpi^n)\\
&\cong\Hom_{\mathcal{O}\llbracket U_p\rrbracket}^{\cont}(W,S^{\mathrm{l.alg}}(U^p,W_\Sigma))/\varpi^n.    
\end{align*}
By \cite[Lemma 7.4.2]{EGH}, we have a natural isomorphism of Hecke modules
\[\Hom_{\mathcal{O}\llbracket U_p\rrbracket}^{\cont}(W,S^{\mathrm{l.alg}}(U^p, W_\Sigma)_{\mathfrak{m}'})\cong W^d\otimes_{\mathcal{O}} (S^{\mathrm{l.alg}}(U^p, W_\Sigma)_{\mathfrak{m}'})^{U_p}\cong S(U_pU^p, W_\Sigma\otimes W^d)_{\mathfrak{m}'}\]
Therefore, with the adjointness of limit, we obtain that
\[
 M_\infty (W)/\mathfrak{a}_\infty\cong(\varinjlim_{n}(S(U_pU^p, W_\Sigma\otimes W^d)_{\mathfrak{m}'})/\varpi^n)^\vee\cong\varprojlim_{n}(S(U_pU^p, W_\Sigma\otimes W^d)_{\mathfrak{m}'}/\varpi^n)^\vee.\]
Since for any $U$ stable $\mathcal{O}$-lattice $\sigma^\circ$, (cf. proof of \cite[Lemma 4.14]{CEG})
\[(\sigma^\circ)^d/\varpi^n\cong\Hom_{\mathcal{O}}(\sigma^\circ/\varpi^n, \mathcal{O}/\varpi^n)\cong (\sigma^\circ/\varpi)^{\vee}.\]
As $S(U_pU^p, W_\Sigma\otimes (V\otimes \sigma^v)^d)_{\mathfrak{m}'}$ is a finitely generated $\mathcal{O}$-module stable under the $U$-action, we obtain 
\[
     M_\infty (V\otimes \sigma^v)/\mathfrak{a}_\infty\cong S(U_pU^p, (W_\Sigma\otimes (V\otimes \sigma^v)^d)_{\mathfrak{m}'})^d.\qedhere\]
\end{proof}

\subsection{Automorphy lifting}
\begin{theorem}\label{automorphy lifting}
    Given a continuous Galois representation $r\colon G_F\to \GL_2(E)$ with the following properties: 
    \hfill\begin{enumerate}
        \item $\overline{r}$ satisfies \cref{condition for r};
        \item $r$ is unramified almost everywhere and satisfies $r^c=r^{\vee}\epsilon^{-1}$;
        \item For all places $v\in S_p$, $r|_{G_{F_{\widetilde{v}}}}$ is potentially crystalline with Hodge--Tate weights $\lambda_v$ and with $(2\ell_v+2)$-generic tame inertial type $\tau_v$, where $\ell_v=\max\{(\lambda_v)_{j,1}-(\lambda_v)_{j,2}\}$ and $4\ell_v\leq N_v$.
        \item $\overline{r}\cong \overline{r}_\iota(\pi)$ for a regular conjugate self-dual cuspidal representation $\pi$ of $\GL_2(\mathbb{A}_F)$ with infinitesimal character $\lambda-\eta$ such that $\otimes_{v\in S_p}\sigma(\tau_v)$ is a $K$-type for $\otimes_{S_p}\pi_v$, where $r_{\iota}(\pi)$ is the continuous representation attached to $\pi$ by \cite[Theorem~2.1.2]{BLGG}.
    \end{enumerate}
Then there exists a RACSDC representation $\pi$ of $\GL_2(\mathbb{A}_F)$ such that $r\otimes_{E}\overline{\Q}_p\cong r_{\iota}(\pi)$.
\end{theorem}
\begin{proof}
By \cref{normal domain}, we know that $R^{\lambda_{\widetilde{v}},\tau_{\widetilde{v}}}_{\overline{r}_{\widetilde{v}}}$ is a domain. Therefore, the automorphy lifting result follows from applying the usual Taylor--Wiles method.
\end{proof}
\begin{remark}
    It is possible to prove the theorem with the results in the existing literature. As explained in \cref{BreuilMezard}, we know that Breuil--M\'{e}zard conjecture holds for $\GL_2$ by \cite[Theorem~1.3.1]{BMandMS}. Then by \cite[Lemma~4.3.9]{GK} (\textit{cf.} \cite[Lemma~5.5.1]{EG14}), we deduce that the support of $M\otimes_{\Z_p}\Q_p$ meets every irreducible component of $\Spec R^{loc}[1/p]$, and hence the automorphy lifting theorem holds.
\end{remark}
\section{Breuil's lattice conjecture}\label{ch5:lattice Conjecture}
\subsection{Structure of lattices}
Recall that we define $\sigma(\lambda,\tau):=\sigma(\tau)\otimes V(\lambda-\eta)$ and write $\overline{\sigma}(\lambda,\tau)$ for the mod $p$ reduction of a $\GL_2(\mathcal{O}_K)$-invariant $\mathcal{O}$-lattice inside $\sigma(\lambda,\tau)$. Replacing $E$ with an extension, we will assume that $\sigma(\tau)$ is defined over $E$. From now on, we will assume $\tau$ is $2\ell+2$-generic where $\ell=\max_j\{\lambda_{j,2}-\lambda_{j,1}\}$.
\begin{lemma}\label{resmultfree}
 Given a Serre weight $\kappa\in \JH(\overline{\sigma}(\lambda,\tau))$, there exists a unique $\mathcal{O}$-lattice up to homothety in $\sigma( \lambda,\tau)$, which we denote as $\sigma(\lambda,\tau)_\kappa$ such that the cosocle of $\sigma(\lambda,\tau)_\kappa$ is precisely $\kappa \in \JH(\overline{\sigma}( \lambda,\tau)).$
\end{lemma}
\begin{remark}
    The lower index notation agreed with the convention in \cite{EGS}, but was opposite to \cite{LatticeforGL3}.
\end{remark}
\begin{proof}
By \cite[Propositions 1.1 and 1.3]{diamond07}, $\overline{\sigma}(\tau)$ is residually multiplicity free. Moreover, by \cref{prop on JH}, its Jordan--H\"older factors are given by the vertices of a hypercube of length $1$ in the extension graph. Note that $\JH(\overline{\sigma}(\lambda, \tau))=\JH(\overline{\sigma}(\tau)\otimes_{\F,j} L(\lambda^{(j)}-\eta))$. By \cref{tensor alg rep}, the tensor product of a Serre weight with $L(\lambda^{(j)}-\eta)$ is the direct sum of $\lambda_{j,2}-\lambda_{j,1}-1$ Serre weights, which are on a line in the $j$-th direction in the extension graph, with an interval of $2$ between consecutive ones. Therefore, for $\kappa,\kappa'\in \JH(\overline{\sigma}(\tau))$, $\kappa\otimes_{\F,j} L(\lambda^{(j)}-\eta))$ and $\kappa'\otimes_{\F,j} L(\lambda^{(j)}-\eta))$ share no common Jordan--H\"older factor unless $\kappa'=\kappa$. Hence, $\sigma(\lambda, \tau)$ is residually multiplicity free. Moreover, the $\JH(\overline{\sigma}(\lambda, \tau))$ are given by the vertices of a hypercuboid of length $2(\lambda_{j,2}-\lambda_{j,1})-1$ in $j$-th direction.\par
To show that $\sigma(\lambda, \tau)$ is an irreducible representation, we prove by induction that if $V$ is a smooth irreducible representation over $\GL_2(\mathcal{O}_K)$, and $V(\lambda)$ is an irreducible algebraic representation with the highest weight $\lambda$, then $V\otimes V(\lambda)$ is an irreducible representation of $\GL_2(\mathcal{O}_K)$. Since $\GL_2(\mathcal{O}_K)$ is an open subgroup, we can consider the representation of the corresponding Lie algebra $\mathfrak{g}$. The associated $V(\lambda)$ is an irreducible $\mathfrak{g}$-representation. As $\mathfrak{g}$ acts trivially on $V$, $V\otimes V(\lambda)\cong \oplus_{i-1}^n V(\lambda)$ as $\mathfrak{g}$-representation. Then, any $\mathfrak{g}$-subrepresentation $W$ of $V\otimes V(\lambda)$ is of the form $V'\otimes V(\lambda)$, where $V'$ is a subspace of $V$. Then $V'=\Hom_{\mathfrak{g}-mod}(V(\lambda), W)$, and it naturally has a $K$-action; therefore, it is a $K$-representation. However, this implies $V'=0$ or $V$, as $V$ is irreducible. Therefore, the result follows from \cite[Lemma~4.1.1]{EGS} 
 \end{proof}
\begin{lemma}\label{lattice is mk1n torsion}
    Let $\ell_j\colonequals\lambda_{j,1}-\lambda_{j,2}$ and $\ell\colonequals\max_j\{\ell_j\}$. Further assume that $\tau$ is $3\ell$-generic. Then for all $\kappa\in \JH(\overline{\sigma}(\lambda,\tau))$, $\overline{\sigma}(\lambda,\tau)_\kappa$ is $\mathfrak{m}_{K_1}^\ell$-torsion. 
\end{lemma}
\begin{proof}
When $\ell=1$, it follows from the definition, hence we can assume $\ell\geq 2$.
For $\kappa\in \JH(\overline{\sigma}(\tau))$, as $\tau$ is $3\ell$-generic, by \cref{prop on JH}, $\kappa$ is a $3\ell$-generic Serre weight. By \cite[Lemma~3.2]{breuil2012towards} $\JH(\Proj_1 \kappa)$ is given by points in the hypercube of length 3 with centre $\kappa$ in the extension graph. By \cref{lemma 3}, we know that $\JH(\Proj_{n} \kappa)$ is given recursively by adding two points in all directions to the ones from $\JH(\Proj_{n-1} \kappa)$ in the extension graph. On the other hand, by the proof of \cref{resmultfree}, $\JH(\overline{\sigma}(\lambda, \tau))$ is given by the points in a hypercuboid of length $2\ell_j-1$ in the $j$-th direction in the extension graph. Therefore, we can deduce that $\JH(\overline{\sigma}(\lambda,\tau)_\kappa)\subset \JH(\Proj_{n}\kappa)$. Since $\tau$ is $3\ell$-generic, and $3\ell\geq 2\ell+2$, for all $\kappa\in \JH(\overline{\sigma}(\lambda,\tau))$, $\kappa$ is $2\ell$-generic by \cref{genericity condition}. By \cref{Cor on condition of being mn rep}, we deduce that $\overline{\sigma}(\lambda,\tau)_\kappa$ is $\mathfrak{m}_{K_1}^\ell$-torsion.
\end{proof}
For the rest of this chapter, we will assume $\tau$ is $2n+2$ generic and $\overline{\rho}$ is $4n$-generic. By \cref{tame type criterion}, $R^{\lambda, \tau}_{\overline{r}_v}=0$ if and only if $\JH(\overline{\sigma}(\lambda, \tau))\cap W(\overline{r}_v)= \emptyset$. Without loss of generality, we assume this is the case and fix $\kappa_\circ$ such that $\kappa_\circ\in\JH(\overline{\sigma}(\lambda, \tau))\cap W(\overline{r}_v)$. 
For $\kappa\in \JH(\overline{\sigma}(\lambda, \tau))$, $\kappa\neq \kappa_\circ$, we fix a lattice $\sigma(\lambda,\tau)_\kappa$ of $\sigma(\lambda, \tau)$ such that $\sigma(\lambda,\tau)_\kappa$ has cosocle $\kappa$ and we have a saturated inclusion $\sigma(\lambda, \tau)_{\kappa_\circ}\hookrightarrow \sigma(\lambda, \tau)_{\kappa}$. We write $\overline{\sigma}(\lambda,\tau)_\kappa$ for its reduction modulo $p$. Given a lattice $\sigma^\circ$, we define $\epsilon_\kappa(\sigma^\circ)$ to be the minimum integer such that $p^{\epsilon_\kappa(\sigma^\circ)}\sigma(\lambda,\tau)_\kappa\hookrightarrow\sigma^\circ$ is saturated. 
 We can therefore reinterpret and generalize the result of \cite[\S~5.2.2]{EGS} using \cref{Cor on socle filtration}, and obtain the following lemma:
\begin{lemma}\label{saturated distance}
    Assume $\ell=\max_j\{\lambda_{j,1}-\lambda_{j,2}\}$ and $\tau$ is $3\ell$-generic. Given Serre weights $\delta, \kappa, \kappa'\in \JH(\overline{\sigma}(\lambda, \tau))$, then 
    \[\epsilon_\delta(\sigma(\lambda, \tau)_{\kappa})+\epsilon_{\kappa}(\sigma(\lambda,\tau)_{\kappa'})\geq \epsilon_\delta(\sigma(\lambda,\tau)_{\kappa'}),\]
    with equality if and only if $\overline{\sigma}(\lambda, \tau)_{\kappa'}$ contains a subquotient with socle $\delta$ and cosocle $\kappa$, which is equivalent to $\kappa-\delta\leq \kappa'-\delta$ (\cref{definition on relations}).
\end{lemma}
\begin{proof}
The inequality follows from the definition of $\epsilon_\delta$. We can compare the inclusion
\[p^{\epsilon_\delta(\sigma(\lambda,\tau)_{\kappa'})}\sigma(\lambda, \tau)_\delta\subset \sigma(\lambda,\tau)_{\kappa'}, \]
\[ p^{\epsilon_\delta(\sigma(\lambda,\tau)_\kappa)+\epsilon_\kappa(\sigma(\lambda,\tau)_{\kappa'})}\sigma(\lambda, \tau)_\delta\subset p^{\epsilon_\kappa(\sigma(\lambda,\tau)_{\kappa'})}\sigma(\lambda, \tau)_{\kappa}\subset \sigma(\lambda,\tau)_{\kappa'}.\]
Therefore, $\overline{\sigma}(\lambda, \tau)_{\kappa'}$ contains a subquotient with socle $\delta$ and cosocle $\kappa$ if and only if 
\[\epsilon_\delta(\sigma(\lambda,\tau)_\kappa)+\epsilon_\kappa(\sigma(\lambda,\tau)_{\kappa'})<\epsilon_\delta(\sigma(\lambda,\tau)_{\kappa'})+1.\]
As $\overline{\sigma}(\lambda,\tau)_\kappa$ is $\mathfrak{m}_{K_1}^\ell$-torsion by \cref{lattice is mk1n torsion}, we can apply the results from section $2$. In particular, recall from \cref{Main theorem on generalization}, $I(\delta, \kappa')$ is the unique multiplicity-free representation with socle $\delta$ and cosocle $\kappa'$. Moreover, $\overline{\sigma}(\lambda, \tau)_{\kappa'}$ contains a subquotient with socle $\delta$ and cosocle $\kappa$ if and only if $\kappa$ is a subquotient of $I(\delta, \kappa')$. By \cref{Cor on socle filtration}, this is equivalent to $\kappa-\delta\leq \kappa'-\delta$.
\end{proof}
\begin{remark}
    We can give the description of the socle of $\overline{\sigma}(\lambda,\tau)_\kappa$ as follows. By \cref{tensor alg rep}, $\JH(\overline{\sigma}(\lambda,\tau))$ is given by a hypercuboid in the extension graph. Using \cref{Cor on socle filtration}, the socle of $\overline{\sigma}(\lambda,\tau)_\kappa$ is the sum of the corners in the hypercuboid, which are different from $\kappa$ in all $f$ dimensions. For instance, if $\kappa$ is at the corner of the hypercuboid, then the socle of $\overline{\sigma}(\lambda,\tau)_\kappa$ is the opposite corner. In general, the socle is not irreducible.
\end{remark}
We have the following proposition analogous to \cite[Theorem~4.1.9]{LatticeforGL3}
\begin{proposition} \label{lattices with neighbouring cosocle}
    Assume $\max_j\{\lambda_{j,1}-\lambda_{j,2}\}=\ell$ and $\tau$ is $(5\ell+1)$-generic. Given Serre weights $\kappa, \kappa'\in \JH(\overline{\sigma}( \lambda,\tau))$ such that $\kappa$ and $\kappa'$ are distance one apart in the extension graph. Assume that $\sigma(\lambda, \tau)_{\kappa}\hookrightarrow\sigma(\lambda,\tau)_{\kappa'}$ is saturated (i.e.\ , $\epsilon_{\kappa}(\sigma(\lambda,\tau)_{\kappa'})= 0$), then $\epsilon_{\kappa'}(\sigma(\lambda,\tau)_\kappa)= 1$. Moreover, for any $\delta\in \JH(\overline{\sigma}( \lambda,\tau))$, we have $\delta\in\JH(\Coker(\sigma(\lambda,\tau)_\kappa\hookrightarrow\sigma(\lambda,\tau)_{\kappa'}))$ if and only if $\kappa'-\kappa\leq \delta-\kappa$. Conversely, $\delta\in\JH(\Coker(p\sigma(\lambda,\tau)_{\kappa'}\hookrightarrow\sigma(\lambda, \tau)_{\kappa}))$ if and only if $\kappa-\kappa'\leq \delta-\kappa'$.
\end{proposition}
\begin{proof}
    We follow the argument of \cite[Proposition~4.3.7]{LatticeforGL3}. As $\tau$ is $(5\ell+1)$-generic, all $\kappa\in \JH(\overline{\sigma}(\lambda,\tau))$ is $(4\ell+1)$-generic. By \cref{prop on W(r)}, we can find a (non-semisimple unless $f=1$), $(4\ell+1)$-generic  $\overline{\rho}$ such that $W(\overline{\rho})\cap\JH(\overline{\sigma}(\lambda,\tau)=\{\kappa,\kappa'\}$. By \cite[Theorem~A.4]{GK}, we can find an imaginary CM field $F$ with a maximal real subfield $F^{+}$ and an RACSDC automorphic representation $\pi$ of $\GL_2(\mathbb{A}_{F})$ such that $\overline{r}_{p, \iota}(\pi)$ satisfies \cref{condition for r} and for each place $v|p$ in $F^+$, there is a place $\widetilde{v}$ of $F$ lying over $v$ such that \(\overline{r}_{p, \iota}(\pi)|_{G_{F_{\widetilde{v}}}}\) is isomorphic to an unramified twist of $\overline{\rho}$. Then, we can obtain a patching functor $M'_\infty$ for $\mathcal{S}=\{v\}$ as in \cref{patching functor existence}. 
    By \cref{normal domain},
    \[R^{\lambda,\tau}_{\overline{\rho}}\cong\mathcal{O}\llbracket x,y, z_1, z_2\rrbracket/(xy-p).\]
    We then choose a subring $S=\Z_p\llbracket x,y, z_1, z_2\rrbracket/(xy-p)\subseteq \mathcal{O}\llbracket x,y, z_1, z_2\rrbracket/(xy-p)$.
    By the axiom of the patching functor, $M'_\infty(\sigma(\lambda, \tau)_{\kappa}), M'_\infty(\sigma(\lambda,\tau)_{\kappa'})$ are $p$-torsion free and maximal Cohen--Macaulay module over $R_\infty(\lambda,\tau)$ and hence also are maximal Cohen--Macaulay $S$-modules. Similarly, $M'_\infty(\kappa),M'_\infty(\kappa')$ are maximal Cohen--Macaulay $S$-modules respectively. \par
    We choose the minimal integer $k\geq 1$ such that we have a chain of saturated inclusions of lattices \[p^k \sigma(\lambda,\tau)_{\kappa'}\subseteq\sigma(\lambda, \tau)_{\kappa}\subseteq \sigma(\lambda,\tau)_{\kappa'}.\]
Let 
\[C\colonequals \Coker(p^k\sigma(\lambda,\tau)_{\kappa'}+p\sigma(\lambda, \tau)_{\kappa}\hookrightarrow\sigma(\lambda, \tau)_{\kappa}).\] 
Then $M_\infty(C)$ is annihilated by $p$ and hence its scheme theoretic support over $S$ is in $\Spec(S/pS)$, which is reduced. Recall that $R_{\overline{\rho}}^{\kappa}$ is reduced, $p$-torsion free quotient of $R^\square_{\overline{\rho}}$ corresponding to the crystalline deformation of Hodge type $\kappa$, and $R_\infty(\kappa)$ is a power series ring over $R_{\overline{\rho}}^{\kappa}$. By the proof \cite[Lemma 3.6.2]{LatticeforGL3}, the scheme theoretic support of $M_\infty(C)$ over $S$ is the reduced subscheme underlying $\bigcup_{\sigma\in \JH(C)}\Supp_S(M_\infty(\sigma))$. If $\delta\neq\kappa$ or $\kappa'$, then by \cref{Serre weight conjecture}, $M_\infty(\delta)=0$. Moreover, as $\sigma(\lambda, \tau)$ is residually multiplicity-free, by descent to the maximal unramified subfield, we can deduce that $C$ does not contain $\kappa'$ as a Jordan--H\"{o}lder factor. Therefore, the scheme theoretic support of $M_\infty(C)$ over $S$ is the reduced scheme underlying $\Spec_S\overline{R}_\infty(\kappa)$.
Moreover, by \cref{normal domain}, we have $ \text{Ann}_{R^{\lambda,\tau}_{\overline{\rho}} }\overline{R}_{\overline{\rho}}^\kappa=(x)$ or $(y)$, and we assume it to be $(x)$ without loss of generality.
Therefore, $x$ annihilates $M_\infty'(C)$, and we have
     \[xM_\infty'(\sigma(\lambda, \tau)_{\kappa})\subset M_\infty' (p^k\sigma(\lambda,\tau)_{\kappa'}+p\sigma(\lambda, \tau)_{\kappa})=p^kM_\infty' (\sigma(\lambda,\tau)_{\kappa'})+pM_\infty' (\sigma(\lambda, \tau)_{\kappa})\]
(The second equality can be shown by using exactness of patching functor). 
By symmetry, we also have
\[yM_\infty'(\sigma(\lambda, \tau)_{\kappa'})\subset M_\infty' (\sigma(\lambda,\tau)_{\kappa})+pM_\infty'(\sigma(\lambda, \tau)_{\kappa'}).\]
Notice that $xy=p$ in $R^{\lambda,\tau}_{\overline{\rho}}$. Multiplying the second inclusion by $x$, and combining with the fist inclusion, we obtain 
\begin{align*}
    pM_\infty'(\sigma(\lambda, \tau)_{\kappa'})&\subset xM_\infty' (\sigma(\lambda,\tau)_{\kappa})+xpM_\infty' (\sigma(\lambda, \tau)_{\kappa'})\\
    &\subset p^kM_\infty' (\sigma(\lambda,\tau)_{\kappa'})+pM_\infty' (\sigma(\lambda, \tau)_{\kappa})+xpM_\infty' (\sigma(\lambda, \tau)_{\kappa'}).
\end{align*}
Therefore, dividing both side by $p$, we get
\[M_\infty'(\sigma(\lambda, \tau)_{\kappa'})\subset p^{k-1}M_\infty' (\sigma(\lambda,\tau)_{\kappa'})+M_\infty' (\sigma(\lambda, \tau)_{\kappa})+xM_\infty' (\sigma(\lambda, \tau)_{\kappa'}).\]
Assume for the sake of contradiction that $k>1$, we consider the image under the composition of the map $M_\infty'(\sigma(\lambda, \tau)_{\kappa'})\twoheadrightarrow M_\infty'(\overline{\sigma}(\lambda, \tau)_{\kappa'})\twoheadrightarrow M_\infty(\kappa')$, we deduce that \(M_\infty'(\kappa)\subset y M_\infty'(\kappa)\). However, as $M_\infty'(\kappa)$ is a free module over $R_{\infty}^{'\kappa}$, which is a power series ring, we have a contradiction. \par
We now determine $\JH(\Coker(\sigma(\lambda, \tau)_{\kappa}\hookrightarrow\sigma(\lambda,\tau)_{\kappa'}))$. Assume $\delta\in \JH(\Coker(\sigma(\lambda,\tau)_{\kappa}\hookrightarrow\sigma(\lambda,\tau)_{\kappa'}))$.
Similar to the proof of \cite[Theorem~5.2.4]{EGS}, we compare the two inclusions:
    \[p^{\epsilon_\delta(\sigma(\lambda,\tau)_\kappa)}\sigma(\lambda, \tau)_\delta\subseteq \sigma(\lambda,\tau)_\kappa \subseteq \sigma(\lambda,\tau)_{\kappa'},\qquad \qquad p^{\epsilon_\delta(\sigma(\lambda,\tau)_{\kappa'})}\sigma(\lambda, \tau)_\delta\subseteq \sigma(\lambda,\tau)_{\kappa'}.\]
We see that $\delta\in \JH(\Coker(\sigma(\lambda,\tau)_{\kappa}\hookrightarrow\sigma(\lambda,\tau)_{\kappa'}))$ if and only if $\epsilon_\delta(\sigma(\lambda,\tau)_\kappa)> \epsilon_\delta(\sigma(\lambda,\tau)_{\kappa'})$, which is equivalent to not having $\kappa-\delta\leq \kappa'-\delta$ by \cref{saturated distance}. This is only possible if $\kappa'-\delta\leq \kappa-\delta$, as $\kappa'$ and $\kappa$ are distance one apart. This is also equivalent to $\kappa'-\kappa\leq \delta-\kappa$, by the remark in \cref{definition on relations}.
The last statement follows analogously from comparing the two inclusions:
\[p^{\epsilon_\delta(\sigma(\lambda,\tau)_\kappa)}\sigma(\lambda, \tau)_\delta\subseteq \sigma(\lambda,\tau)_\kappa,\qquad \qquad p^{\epsilon_\delta(\sigma(\lambda,\tau)_{\kappa'})+1}\sigma(\lambda, \tau)_\delta\subseteq p\sigma(\lambda,\tau)_{\kappa'}\subseteq \sigma(\lambda,\tau)_\kappa.\qedhere\]
\end{proof}
Let $d(\alpha,\beta)$ denote the distance induced by the extension graph as in \cref{ch2: representation result}. We do not need the following result for the rest of the paper but it is included for completeness. For readability, will write $R_\kappa$ for $\sigma(\lambda, \tau)_{\kappa}$.
\begin{lemma}
 Assume $\max_j\{\lambda_{j,1}-\lambda_{j,2}\}=\ell$ and $\tau$ is $(5\ell+1)$-generic. Given Serre weights $\kappa, \kappa'\in \JH(\overline{\sigma}( \lambda,\tau))$. Assume that $R_{\kappa}\hookrightarrow R_{\kappa'}$ is saturated (i.e.\ , $\epsilon_{\kappa}(R_{\kappa'})= 0$), then $\epsilon_{\kappa'}(R_{\kappa})= d(\kappa,\kappa')$.
\end{lemma}
\begin{proof}
For $d(\kappa,\kappa')=1$, it follows from \cref{lattices with neighbouring cosocle}. Assume $\kappa_1, \kappa_2,\kappa\in \JH(\overline{\sigma}( \lambda,\tau))$ with $d(\kappa_1, \kappa_2)=1$, \cite[lemma 4.3.9]{LatticeforGL3} is still valid that if we have two saturated inclusions $R_{\kappa_1}\hookrightarrow R_{\kappa}$ and $R_{\kappa_2}\hookrightarrow R_{\kappa}$, then $R_{\kappa_1}\hookrightarrow R_{\kappa_2}$ or $R_{\kappa_2}\hookrightarrow R_{\kappa_1}$. Then by the proof of \cite[lemma 4.3.10]{LatticeforGL3}, if $d(\kappa_2,\kappa)\geq d(\kappa_1,\kappa)$, then $R_{\kappa_2}\hookrightarrow R_{\kappa_1}$. Following the proof of \cite[lemma 4.3.15]{LatticeforGL3}, we choose a directed path in the extension graph
\[\kappa_0=\kappa'\mbox{---}\kappa_1\mbox{---}\dots\mbox{---}\kappa_n=\kappa\]
where $d(\kappa_i, \kappa_{i+1})=1$ for $0\leq i\leq n-1$. Assume $R_{\kappa_i}\hookrightarrow R_{\kappa_0}$ is a saturated inclusion for all $0\leq i\leq n$, by repeatedly applying the argument above, we have (saturated) inclusions
\[R_{\kappa_n}\hookrightarrow \dots\hookrightarrow R_{\kappa_1}\hookrightarrow R_{\kappa_0}.\]
Let $k_i\colonequals \epsilon_{\kappa_{n-i}}(R_{\kappa_n})$. We want to show that $k_n=d(\kappa_n,\kappa_0)$. By the base case where $d(\kappa_1,\kappa_2)=1$, we deduce that we have inclusions
\[p^n R_{\kappa_0}\hookrightarrow p^{n-1} R_{\kappa_1}\hookrightarrow \dots\hookrightarrow pR_{\kappa_{n-1}}\hookrightarrow R_{\kappa_n}.\]
Therefore, $k_n\leq n$. On the other hand,
\[\kappa_n\mbox{---}\kappa_{n-1}\mbox{---}\dots\mbox{---}\kappa_0\]
is also a directed path in the extension graph. By definition, 
$p^{k_i}R_{\kappa_{n-i}}\hookrightarrow R_{\kappa_n}$ is a saturated inclusion for all $i$. Therefore, by the same argument as above, we have inclusion
\[p^{k_n} R_{\kappa_0}\hookrightarrow p^{k_{n-1}} R_{\kappa_1}\hookrightarrow \dots\hookrightarrow p^{k_1}R_{\kappa_{n-1}}\hookrightarrow R_{\kappa_n}.\]
As $\kappa_i\neq \kappa_{i+1}$, all the inclusion are strict, hence $k_n\geq n$. 
\end{proof}
\subsection{Relations of patched modules of lattices}\label{relation of patched modules}
From now on, we assume that $r_v$ is potentially crystalline with Hodge--Tate weights $\lambda$ with $\ell_v\coloneqq\lambda_{v_{j,1}}-\lambda_{v_{j,2}}>0$ and $\tau_v$ is a $(5\ell_v+1)$-generic inertial type. Moreover, we assume that, up to twisting by a power of $\omega_f$, $\overline{r}_v$ is of the form in \cref{condition on rho bar} with $N_v\geq 4\ell_v+$.\par
Fix $\kappa_\circ\in W(\overline{r})\cap\JH(\overline{\sigma}(\lambda,\tau))$, by \cref{BreuilMezard}, we have $R^{\lambda,\tau}_{\overline{r}}\twoheadrightarrow\overline{R}_{\overline{r}}^{\kappa_\circ}$, and the kernel $\mathfrak{p}({\kappa_\circ})$ is given by $(z(\kappa_\circ)_j)_{j\in \mathcal{K}_v}$ in \cref{normal domain}.
\begin{definition}
We define (\textit{cf.} \cref{prop on W(r)}) 
\[\widetilde{s}_j=\begin{cases}
    \sgn(s_j) &\text{if } j\in \mathcal{K}_v \text{ (i.e.\ , } F(\mathfrak{t}_{\mu-\eta}(0,\dots, \sgn(s_j), 0))\in W(\overline{r}_v));\\
    0 &\text{otherwise}.
\end{cases}\]  
\end{definition} 
We have a patching functor $M_{\infty}^{\sigma^v}$ for the set of $(\lambda,\tau)$ under our genericity assumptions from \cref{ch4:patching_functor}, which we simply denote as $M_\infty$.
 Given Serre weights $\kappa, \kappa'\in\JH(\overline{\sigma}(\lambda, \tau))$ with $\kappa=F(\mathfrak{t}_{\mu-\eta}(\alpha))$ and $\kappa'=F(\mathfrak{t}_{\mu-\eta}(\alpha'))$ are distance $1$ apart in the extension graph, with $\alpha_i=\alpha'_i$ for all $i\neq j$. Assume further that $\sigma(\lambda,\tau)_{\kappa}\hookrightarrow\sigma(\lambda,\tau)_{\kappa'}$ is saturated. 
    Define $\varpi(\kappa, \kappa')\in R^{\lambda, \tau}_{\overline{r}_v}$ as follows,
\begin{equation}\label{varpi}
    \varpi(\kappa, \kappa')\colonequals \begin{cases}
    1 &\text{if } ((\alpha')_j< (\alpha)_j\leq \min\{0, (\widetilde{s})_j\}) \text{ or } (\max\{0, (\widetilde{s})_j\} \leq (\alpha)_j<(\alpha')_j);\\
     y_j&\text{if } (\alpha)_j=0 \text{ and } (\alpha')_j=(\widetilde{s})_j\neq0;\\
     x_j&\text{if } \alpha=(\widetilde{s})_j\neq0 \text{ and } (\alpha')_j=0;\\
    p &\text{if }  ((\alpha)_j<(\alpha')_j\leq \min\{0, (\widetilde{s})_j\}) \text{ or } (\max\{0, (\widetilde{s})_j\} \leq(\alpha')_j< (\alpha)_j).
\end{cases}
\end{equation}
We define $\varpi'(\kappa, \kappa')$ analogously by swapping $1$ with $p$ and $x_j$ with $y_j$.
\begin{remark}
    The first condition should be understood as that $\kappa'$ is ``further away from the sets of modular Serre weights" than $\kappa$. 
\end{remark}
\begin{proposition}\label{equality on patched module}
    We have equalities,
    \[\varpi(\kappa, \kappa') M_\infty(\sigma(\lambda,\tau)_{\kappa'})=M_\infty(\sigma(\lambda,\tau)_{\kappa});\]
    \[\varpi'(\kappa, \kappa') M_\infty(\sigma(\lambda,\tau)_{\kappa})=pM_\infty(\sigma(\lambda,\tau)_{\kappa'}).\]
\end{proposition}
\begin{proof}
The proof goes the same way as in \cite[Proposition~8.1.1]{EGS}. We deduce from \cref{lattices with neighbouring cosocle} that the cokernel of $\sigma(\lambda,\tau)_\kappa\hookrightarrow \sigma(\lambda,\tau)_{\kappa'}$ is a successive extension of the weights $\delta\in \JH(\overline{\sigma}(\lambda, \tau))$ where $\delta=F(\mathfrak{t}_{\mu-\eta}(\beta))$ with $\alpha'-\alpha\leq \beta-\alpha$.\par
By \cref{saturated distance}, we deduce that this cokernel is annihilated by $p$. Then the cokernel of $M_\infty(\sigma(\lambda,\tau)_{\kappa})\hookrightarrow M_\infty(\sigma(\lambda,\tau)_{\kappa'})$ is given by a successive extension of patched modules $M_\infty(\delta)$ and is annihilated by $p$.
Similar to the proof of \cref{saturated distance}, using \cref{normal domain}, we can fix an isomorphism \[R_\infty(\lambda,\tau)\cong\mathcal{O}\llbracket (x_j,y_j)_{j\in\mathcal{K}}, z_1, \dots z_k\rrbracket/(x_jy_j-p)_{j\in\mathcal{K}}\] for some $\mathcal{K}$. We let $S\subset R_\infty(\lambda,\tau)$ be the sub-ring $\Z_p\llbracket (x_j,y_j)_{j\in\mathcal{K}}, z_1, \dots z_k\rrbracket/(x_jy_j-p)_{j\in\mathcal{K}}$. Let $\mathfrak{p}_S(\kappa)=\sum_v \mathfrak{p}_S(\kappa_v)$ where $\mathfrak{p}_S(\kappa_v)=(z_j(\kappa_v))_{j\in \mathcal{K}_v}$. Then $\Supp_S M_\infty(\kappa)$ is annihilated by $\mathfrak{p}_S(\kappa)$. \par
Therefore, the scheme-theoretical support of $M_\infty(\sigma(\lambda,\tau)_{\kappa'}/\sigma(\lambda,\tau)_\kappa)$ in $\Spec S$ is contained in $\Spec S/pS$. Therefore, it is generically reduced, and the underlying reduced subscheme is equal to
    \[\bigcup_{\begin{subarray}{c}\delta\in W(\overline{r}_v, \lambda,\tau)\\ \delta=F(\mathfrak{t}_{\mu-\eta}(\beta))\colon \alpha'-\alpha\leq \beta-\alpha \end{subarray}}\Supp_S M_\infty(\delta).\]
    In the first case of \cref{varpi}, we see that the scheme-theoretic support of the cokernel in $S$ is trivial. In the second case of \cref{varpi}, we must have $\delta=F(\mathfrak{t}_{\mu-\eta}(\beta))$ with $\beta_j=\sgn(s_j)$, and hence the cokernel is annihilated by $y_j$, by \cref{normal domain}. Analogously, in the third case of \cref{varpi}, the cokernel is annihilated by $x_j$. Finally, in the fourth case, by \cref{lattices with neighbouring cosocle}, the cokernel is annihilated by $p$. Therefore,
    \begin{equation}\label{inclusion of patched modules}
        \varpi(\kappa, \kappa')M_\infty(\sigma(\lambda,\tau)_{\kappa'})\subseteq M_\infty(\sigma(\lambda,\tau)_{\kappa}).
    \end{equation}
    Similarly, the cokernel of $pM_\infty (\sigma(\lambda,\tau)_{\kappa'})\hookrightarrow M_\infty(\sigma(\lambda,\tau)_\kappa)$ is annihilated by $\varpi'(\kappa,\kappa')$, and hence
    \begin{equation}\label{inclusion of patched modules2}
        \varpi'(\kappa, \kappa')M_\infty(\sigma(\lambda,\tau)_\kappa)\subseteq p M_\infty(\sigma(\lambda,\tau)_{\kappa'}).
    \end{equation}
    Note that $\varpi(\kappa, \kappa')\varpi'(\kappa, \kappa')=p$. Multiplying \cref{inclusion of patched modules} by $\varpi'(\kappa,\kappa')$, we obtain that
    \[p M_\infty(\sigma(\lambda,\tau)_{\kappa'})\subseteq \varpi'(\kappa, \kappa')M_\infty(\sigma(\lambda,\tau)_\kappa).\]
    Combining with \cref{inclusion of patched modules2}, we deduce that this is indeed an equality. As $M_\infty(\sigma(\lambda,\tau)_\kappa)$, $M_\infty(\sigma(\lambda,\tau)_{\kappa'})$ are $p$-torsion free, the second equality is obtained by multiplying by $\varpi(\kappa,\kappa')$ and dividing by $p$.
\end{proof}
We define $\widetilde{z}(\kappa_\circ)_j=p/z(\kappa_\circ)_j$. (That is $\widetilde{z}(\kappa_\circ)_j=y_j$ if $z(\kappa_\circ)_j=x_j$, and vice versa.)
Assume $\kappa_\circ=F(\mathfrak{t}_{\mu-\eta}(b))\in W(\overline{r}_v,\lambda,\tau)$, where $b_j\in \{0,\sgn(s_j)\}$ is given by \cref{prop on W(r)} applying to $W(\overline{r}_v)$. For $\kappa\in \JH(\overline{\sigma}(\lambda, \tau))$ if $\kappa=F(\mathfrak{t}_{\mu-\eta}(\alpha))$, define
\[ \varpi_j(\kappa)\colonequals \begin{cases}
1 &\text{if }\alpha_j< b_j=\min(0,\widetilde{s}_j) \text{ or }\alpha_j>b_j=\max\{0,\widetilde{s}_j\};\\
\widetilde{z}_j(\kappa_\circ) &\text{ otherwise}.
\end{cases}\]
We define $\varpi(\kappa)\colonequals \prod_{1\leq j\leq f_v}\varpi_j(\kappa)\in R^{\lambda, \tau}_{\overline{r}_v}$.
\begin{remark}
    The last two conditions capture the case for which there is a modular Serre weight between $\kappa$ and $\kappa_\circ$ when we consider the projection to the $j$th coordinate of the extension graph.
\end{remark}
\begin{proposition}\label{relating patched module}
For $\kappa\in \JH(\overline{\sigma}(\lambda, \tau))$, we have an equality:
\[M_\infty(\sigma(\lambda, \tau)_{\kappa_\circ})=\varpi(\kappa)M_\infty(\sigma(\lambda, \tau)_{\kappa}).\]
\end{proposition}
\begin{proof}
We prove by induction on the distance between $\kappa$ and $\kappa_\circ$ in the extension graph. 
Using the extension graph, we fix a sequence $\kappa_0=\kappa_\circ,\cdots, \kappa_f=\kappa$, such that $\kappa_{j-1}$ and $\kappa_{j}$ differ only in the $j$th direction of the extension graph. In particular, if $\kappa_{j-1}=F(\mathfrak{t}_\mu(\alpha^{j-1}))$, then $\alpha^{j-1}_j=b_j$, since we did not change the $j$th coordinate in the first $j-1$ steps. Moreover, $\varpi(\kappa)=\prod_{j=1}^f\varpi(\kappa_{j},\kappa_{j-1})$. Therefore, it suffices to show that $\varpi_j(\kappa)=\varpi(\kappa_{j-1},\kappa_{j})$. If the distance between $\kappa_{j}$ and $\kappa_{j-1}$ is $1$, then by \cref{equality on patched module}, $\varpi_j(\kappa)=\varpi(\kappa_{j-1},\kappa_{j})$. \par 
Assume that it holds for distance $n-1\geq 1$, then let $\kappa'_j$ be the Serre weight on the line segment in the extension graph between $\kappa_{j-1}$ and $\kappa_{j}$ which is distance $1$ away from $\kappa_j$ and $n-1$ away from $\kappa_{j-1}$. By the induction hypothesis $\varpi_j(\kappa')=\varpi(\kappa_{j-1},\kappa'_{j})$. Moreover, as $\kappa'-\kappa_\circ\leq\kappa-\kappa_\circ$, by \cref{saturated distance}, we deduce that $\epsilon_{\kappa'}(\sigma(\lambda,\tau)_\kappa)=0$, that is, $\sigma(\lambda,\tau)_{\kappa'}\hookrightarrow\sigma(\lambda,\tau)_\kappa$. As $n\geq 2$, the $j$th coordinate of $\kappa_j$ is not in $\{0, \widetilde{s}_j\}$, and $\kappa_j$ is further away than $\kappa'_j$ from $0$ in the $j$th coordinate, we are in the first case of \cref{varpi}, by \cref{equality on patched module}, $\varpi(\kappa'_j, \kappa_j)=1$. Therefore, $\varpi_j(\kappa)=\varpi(\kappa_{j-1}, \kappa'_j)\varpi(\kappa'_{j},\kappa_{j})=\varpi(\kappa_{j-1}, \kappa_j)$.
\end{proof}
\subsection{Breuil's lattice conjecture} Suppose $r\colon G_F\to \GL_2(E)$ is a Galois representation attached to an eigenform in $S(U,W)_{\mathfrak{m}}$ 
We assume that $\overline{r}$ satisfies \cref{condition for r}, minimally ramified only at primes $v\nmid p$. In particular, $\overline{r}$ corresponds to $\mathfrak{m}\subset \mathbb{T}_{\mathcal{P}}$, which is non-Eisenstein. As $\overline{r}$ is absolutely irreducible, we can conjugate $r$ such that it takes value in $\mathcal{O}$. We assume that for $v\in S_p$, $r_v$ is potentially crystalline with regular Hodge-Tate weights $\lambda_v$ and with inertial type $\tau_v$. Fix a $v\in S_p$, and write $\lambda$ for $\lambda_v$ and $\tau$ for $\tau_v$. We assume $\tau_v$ is $(5\ell+1)$- generic and $\overline{r}_v$ is $(4\ell+1)$-generic, where $ \ell_v\coloneqq\max_j\{\lambda_{j,1}-\lambda_{j,2}\}$. We fix $\sigma^v$ as in the section after \cref{patching functor existence}. 
Let $\mathbb{T}'(U,W)_{\mathfrak{m}'}$ be the image of the universal Hecke algebra $\mathbb{T'}_{\mathcal{P}}$ in $\End(\widetilde{S}(U^v, W)_{\mathfrak{m}'})$.
We write $\mathfrak{p}$ for the kernel of the system of Hecke eigenvalues $\alpha\colon  \mathbb{T}'(U,W)_{\mathfrak{m}'}\to E$ associated to $r$, i.e.\ ; $\alpha$ satisfies
\[\det(r^{\vee}(\Frob_w)X)=\sum_{j=0}^2(-1)^j(\mathbf{N}_{F/\Q}(w))^{\binom{j}{2}}\alpha(T_{w}^{(j)})X^j\]
for all $w\in \mathcal{P}$.\par
\begin{proposition} We have the following isomorphism,
    \[(\widetilde{S}(U^v, \sigma^v\otimes_\mathcal{O}W_\Sigma)_{\mathfrak{m}'}\otimes_{\mathcal{O}}E)^{\mathrm{l.alg}}[\mathfrak{p}]=\pi_v\otimes V(\lambda-\eta),\]
    where $\pi_v$ corresponds to the Weil--Deligne representation associated to $r_v$ by the local Langlands correspondence. (More precisely, $\text{WD}(r_v)^{F-ss}=\text{rec}(\pi_v\otimes|\det|^{\frac{-1}{2}})$.)
\end{proposition}
\begin{proof}
    We recall the proof from \cite[Theorem~4.35]{CEG}. By \cite[Proposition~3.2.4]{MR2207783}, we deduce that the locally algebraic vectors are precisely the algebraic automorphic forms of that weight. Together with the classical local-global compatibility result \cite[Theorem~1.1]{MR3272276}, we deduce that $\pi_v\otimes V(\lambda-\eta)$ appears in $(\widetilde{S}(U^v, \sigma^v\otimes_\mathcal{O}W_\Sigma)_{\mathfrak{m}'}\otimes_{\mathcal{O}}E)^{\mathrm{l.alg}}[\mathfrak{p}]$ with some multiplicity. By our construction and \cite[Th\'eor\`emes 5.4, 5.9]{labesse}, we deduce that it appears with multiplicity one.
\end{proof}
By the inertial local Langlands correspondence \cite[Appendice]{Breuil-Mezard}, $\sigma(\tau)$ is a subrepresentation of $\pi_v$ as a $\GL_2(\mathcal{O}_{F^+_v})$-representation. Therefore,
\[\sigma(\tau)\otimes V(\lambda-\eta)\hookrightarrow\widetilde{S}(U^v, \sigma^v\otimes_\mathcal{O}W_\Sigma)_{\mathfrak{m}'}[\mathfrak{p}]\otimes_\mathcal{\mathcal{O}}E\]
as $\GL_2(\mathcal{O}_{F^+_v})$-representations. Therefore, the completed cohomology with integral coefficient determines the following $\GL_2(\mathcal{O}_{F^+_v})$-invariant $\mathcal{O}$-lattice inside $\sigma(\lambda,\tau)$,
\begin{equation*}
    \sigma^\circ(\lambda,\tau)\colonequals \sigma(\lambda, \tau)\cap \widetilde{S}(U^v, \sigma^v\otimes_\mathcal{O}W_\Sigma)_{\mathfrak{m}'}[\mathfrak{p}].
\end{equation*}
Let $R^{\text{univ}}_{\Sigma}$ be the universal deformation ring for the deformations of $\overline{r}$ which are unramified outside $\Sigma$. As $r$ is a Galois representation attached to an eigenform $S(U,W)_{\mathfrak{m}'}$, we have a Galois representation $r^{\text{mod}}\colon G_F\to \GL_2(\mathbb{T}'(U,W)_{\mathfrak{m}'})$, with $\overline{r}^\text{mod}\cong\overline{r}$. This induces a map $R^{\text{univ}}_{\Sigma}\to \mathbb{T}'(U,W)_{\mathfrak{m}'}$. The composite map $R_\infty\to R^{\text{univ}}_{\Sigma}\to\mathbb{T}'(U,W)_{\mathfrak{m}'}$ given by modulo $\mathfrak{a}_\infty$ (cf. \cref{completed cohomology}), by the local-global compatibility, further induces a map $h\colon R_{\infty}(\lambda, \tau)\to \mathbb{T}'(U,W)_{\mathfrak{m}'}$. We define $\varpi_\mathfrak{p}(\kappa)$ as the image of $\varpi(\kappa)$ under $\mathfrak{p}\circ h$. As $\varpi(\kappa)\in R^{\lambda, \tau}_{\overline{r}_v}$, the image of $h$ coincides with the image of the natural map $R^{\lambda, \tau}_{\overline{r}_v}\to R^{\text{univ}}_{\Sigma}\to\mathbb{T}'(U,W)_{\mathfrak{m}'}$. Therefore, $\varpi_\mathfrak{p}(\kappa)$ only depends on $r_v$.
We have the following version of Breuil's lattice conjecture:
\begin{theorem}\label{lattice conjecture}
Up to homothety, $\sigma^\circ(\lambda,\tau)$ is equal to \[\sum_{\kappa\in \JH(\lambda, \tau)}\varpi_\mathfrak{p}(\kappa)\sigma(\lambda,\tau)_\kappa.\]
\end{theorem} 
\begin{proof}
We follow the proof of \cite[Theorem~8.2.1]{EGS}. Since we only consider the lattice up to homothety, without loss of generality, we assume that $\sigma_{\kappa_\circ}\hookrightarrow\sigma^\circ(\lambda,\tau)$ is saturated. By our normalization before \cref{lattice is mk1n torsion}, we have $\sigma(\lambda, \tau)_{\kappa_\circ}\hookrightarrow\sigma(\lambda,\tau)_\kappa$. Therefore, we can apply \cite[Proposition~4.1.4]{EGS}, and deduce that $\sigma^\circ(\lambda,\tau)=\sum_{\kappa\in \JH(\lambda, \tau)}p^{v(\kappa)}\sigma(\lambda,\tau)_\kappa$ where $p^{v(\kappa)}$ is an element in $\mathcal{O}$ with valuation $v(\kappa)$ such that $p^{v(\kappa)}\sigma(\lambda,\tau)_\kappa\hookrightarrow\sigma^\circ(\lambda,\tau)$ is saturated. By \cref{completed cohomology}, we have the following isomorphism of Hecke modules
\[(M_\infty^{\sigma^v}(\sigma(\lambda,\tau)_\kappa)/a_\infty)^d\cong \Hom_{\mathcal{O}\llbracket U_v\rrbracket}(\sigma(\lambda,\tau)_\kappa, \widetilde{S}(U^v, W_\Sigma\otimes\sigma^v)_{\mathfrak{m}'}).\]
In particular, $\mathbb{T}'(U,W)$ acts on the left-hand side also. 
For any $\mathbb{T}'(U,W)$-module $M$ finitely which is generated over $\mathcal{O}$, $M[\mathfrak{p]}^d$ is canonically isomorphic to the torsion-free part of $(M/\mathfrak{p})^d$. Hence, the torsion free part of \((\Hom_{\mathcal{O}\llbracket U_v\rrbracket}(\sigma(\lambda,\tau)_\kappa, \widetilde{S}(U^v, W_\Sigma\otimes\sigma^v)_{\mathfrak{m}'}))^d/\mathfrak{p}\) is canonically isomorphic to
\[(\Hom_{\mathcal{O}\llbracket U_v\rrbracket}(\sigma(\lambda,\tau)_\kappa, \widetilde{S}(U^v, W_\Sigma\otimes\sigma^v)_{\mathfrak{m}'})[\mathfrak{p}])^d\cong \Hom_{\mathcal{O}\llbracket U_v\rrbracket}(\sigma(\lambda,\tau)_\kappa, \widetilde{S}(U^v, W_\Sigma\otimes\sigma^v)_{\mathfrak{m}'}[\mathfrak{p}])^d
\]
Any map $\sigma(\lambda,\tau)_\kappa\to \widetilde{S}(U^v, W_\Sigma\otimes\sigma^v)_{\mathfrak{m}'}[\mathfrak{p}]$ lands in $\sigma^\circ(\lambda,\tau)\coloneqq\sigma(\lambda, \tau)\cap \widetilde{S}(U^v, \sigma^v\otimes_\mathcal{O}W_\Sigma)_{\mathfrak{m}'}[\mathfrak{p}]$, therefore, $(\Hom_{U_v}(\sigma(\lambda,\tau)_\kappa, \widetilde{S}(U^v, W_\Sigma\otimes\sigma^v)_{\mathfrak{m}'}[\mathfrak{p}]))^d\cong \mathcal{O}$ is torsion free and 
\[(\Hom_{U_v}(\sigma(\lambda,\tau)_\kappa, \sigma^\circ(\lambda,\tau)))^d=(M_\infty^{\sigma^v}(\sigma(\lambda,\tau)_\kappa)/\mathfrak{a}_\infty)/\mathfrak{p}.\]
By \cref{relating patched module}, we deduce that
\[\Hom_{U_v}(\sigma(\lambda,\tau)_\kappa, \sigma^\circ(\lambda,\tau))=\varpi_\mathfrak{p}(\kappa)\Hom_{U_v}(\sigma(\lambda, \tau)_{\kappa_\circ}, \sigma^\circ(\lambda,\tau)).\]
Therefore, by the uniqueness of the gauges \cite[Proposition~4.1.4]{EGS}, we show that $\varpi_\mathfrak{p}(\kappa)$ has the same valuation as $p^{v(\kappa)}$ and finish the proof.
\end{proof}
\begin{remark}
As $\overline{r}$ is absolutely irreducible irreducible, $r$ admit a unique integral model $r^\circ$. One can describe (the valuation of) $\varpi_{\mathfrak{p}}(\kappa)$ explicitly in terms of the Breuil-Kisin module $\mathfrak{M}$ over $\mathcal{O}$ associated to $r^\circ|_{G_{F_p}}$. Assume we choose a $\widetilde{w}$-gauge basis $\beta$, then we have the matrix $A^{(j)}$ capturing the $\phi$-action corresponding to such an eigenbasis. If $\widetilde{w}_j=\mathfrak{t}_{(m,n)}$ (respectively, $\widetilde{w}_j=\mathfrak{m}\mathfrak{t}_{(m,n)}$), then 
\[A^{(j)}=\begin{pmatrix}
\bm{a}_0^{(j)}(v+p)^m+\dots+\bm{a}_{-m}^{(j)}& \bm{b}_0^{(j)}(v+p)^{n-1}+\dots+\bm{b}_{-n+1}^{(j)}\\
v(\bm{c}_0^{(j)}(v+p)^{m-1}+\dots+\bm{c}_{-m+1}^{(j)})& \bm{d}_0^{(j)}(v+p)^{n}+\dots+\bm{d}_{-n}^{(j)}
\end{pmatrix}\] (respectively, 
\[\begin{pmatrix}
\bm{a}_0^{(j)}(v+p)^{m-1}+\dots+\bm{a}_{-m+1}^{(j)}& \bm{b}_0^{(j)}(v+p)^{n}+\dots+\bm{b}_{-n}^{(j)}\\
v(\bm{c}_0^{(j)}(v+p)^{m-1}+\dots+\bm{c}_{-m+1}^{(j)})& \bm{d}_0^{(j)}(v+p)^{n}+\dots+\bm{d}_{-n}^{(j)}
\end{pmatrix}.)\] Let $x_j(\mathfrak{p})=\bm{b}_0^{(j)}$ (respectively $\bm{a}_0^{(j)}$) and $y_j(\mathfrak{p})=\bm{c}_0^{(j)}-[\overline{\bm{c}}_0^{(j)}]$ (respectively $\bm{d}_0^{(j)}-[\overline{\bm{d}}_0^{(j)}]$). Then, following the computation in \cref{normal domain}, up to multiplying by $\mathcal{O}^\times$, $\varpi_{\mathfrak{p}}(\kappa)=\prod_j z_j(\mathfrak{p})$ where $z_j(\mathfrak{p})=x_j(\mathfrak{p})$ or $y_j(\mathfrak{p})$ depending on $\kappa, \kappa_\circ$ as in \cref{varpi}. Note that by \cite[Propositions 5.1.8, 5.2.7]{localmodel}, if we choose another $\widetilde{w}$-gauge basis, it corresponds to multiplying $\varpi_{\mathfrak{p}}(\kappa)$ by $\mathcal{O}^\times$, hence does not affect the valuation of $\varpi_{\mathfrak{p}}(\kappa)$. If $\mathfrak{m}$ has a $\widetilde{w}$-gauge basis and a $\widetilde{w}'$-gauge basis, then we have matrix $\overline{A}^{(j)}$ and $\overline{A}^{(j)'}$ corresponding to $\overline{\mathfrak{M}}$, they are related by 
\[\overline{A}^{(j)}=\overline{A}^{(j)'}\widetilde{w}_j^{'-1}\widetilde{w}_j\]
where we regard $\widetilde{w}$ as an element of $\GL_2(\F\llbracket v\rrbracket)$ via $s_j\mathfrak{t}_{\mu_j}\mapsto \dot{s}_j v^{\mu_j}$ (c.f. proof of \cref{normal domain}). It then follows that the valuation of $\varpi_{\mathfrak{p}}(\kappa)$ is independent of the choice of $\widetilde{w}$ also. 
\end{remark}
\begin{remark}
One should be able to prove an analogous result for Shimura curves for higher parallel Hodge--Tate weight under the same genericity condition. (The requirement of parallel Hodge--Tate weights comes from the parity condition on quaternion algebras, which does not exist for unitary groups.) One can construct a minimal unramified patching $M_\infty$ using quaternion algebras following \cite[\S~6.2]{EGS}. Since the proof relies on \cref{relating patched module} which compares $M_\infty(\sigma(\lambda,\tau)_\kappa)$ and $M_\infty(\sigma(\lambda,\tau)_{\kappa'})$, which, in turn, relies on the result on the Galois deformation rings in \cref{normal domain}, and results on the mod $p$ representations of $\GL_2(\mathcal{O}_K)$ in \cref{Main theorem on generalization}, it is independent of the construction of the patching functor.
\end{remark}
\begin{remark}
    Breuil's original lattice conjecture in \cite[Conjecture~1.2]{Breuiloriginal} is stated for $\sigma^\circ(\tau)\coloneqq\sigma(\tau)\cap \varinjlim_{U_v} H^1(U_vU^v, \mathcal{O})_\mathfrak{m}$. By \cite[Corollary~2.2.25]{MR2207783}, the $\varpi$-adic completion of $\varinjlim_{U_v} H^1(U_vU^v, \mathcal{O})$ is the same as $\widetilde{H}^1(U^v, \mathcal{O})$. Therefore, $\sigma^\circ(\tau)=\sigma(\tau)\cap \widetilde{H}^1(U^v, \mathcal{O})$. Hence, the formulation here is the natural analog for higher Hodge--Tate weights.
\end{remark}
\begin{remark}
If we compare \cref{lattice conjecture} with \cite[Theorem~8.2]{EGS}, $\sigma(\lambda, \tau)_{\kappa_\circ}$ plays a similar role as $\sigma_{\iota(\varnothing)}$ in \cite[Proposition~8.1.1]{EGS}. We do not claim that $X_j, Y_j$ coincide with the ones in \cite[Theorem~7.1.1]{EGS}, as we have taken a different normalization, and the Galois deformation ring is computed by strongly divisible modules in \cite{EGS} and by Breuil--Kisin modules here. \par
\end{remark}

\section{Cyclicity of patched modules}\label{ch6: cyclicity of patched module}
In this chapter, we assume that the minimal patching functor is minimal  and unramified. Let $\mathcal{S}$ be a subset of $S_p$. If there is a CM field $F$ with maximal real subfield $F^+$ with $\overline{r}\colon G_F\to \GL_2(\F)$ satisfying \cref{condition on rho bar}, such that $L_v=F_v$ and $\overline{\rho}_v=\overline{r}|_{G_{F_v^+}}$ for all $v\in \mathcal{S}$. Then by \cref{patching functor existence} and the discussion that follows, there exists a minimal patching functor for $\{\overline{\rho}_v\}_{v \in \mathcal{S}}$.\par
By extending the coefficients, we can assume that the lattice inside $\sigma(\lambda_v,\tau_v)$ is defined over $\mathcal{O}=W(\F)$. As before, we let $\ell_v=\max_j\{(\lambda_v)_{j,1}-(\lambda_v)_{j,2}\}$ and assume that $\tau_v$ is $3\ell_v$-generic. In particular, all the Jordan--H\"older factors of $\overline{\sigma}(\lambda_v,\tau_v)$ are $3\ell_v$ generic. We write $\sigma(\lambda_\mathcal{S}, \tau_\mathcal{S})\coloneqq\otimes_{v\in \mathcal{S}}\sigma(\lambda_v,\tau_v)$ and write $\overline{\sigma}(\lambda_\mathcal{S},\tau_\mathcal{S})$ for the mod $p$ reduction of a $\prod_{v\in \mathcal{S}}\GL_2(\mathcal{O}_{F^+_v})$-invariant $\mathcal{O}$-lattice inside $\sigma(\lambda_\mathcal{S},\tau_\mathcal{S})$.\par
Note that if $\kappa\in \JH(\overline{\sigma}(\lambda, \tau))$, as $\kappa$ is an irreducible representation of the group $\prod_{v\in S_p}\GL_2(k_v)$, $\kappa=\otimes_{v\in S_p}\kappa_v$ where $\kappa_v$ is a Serre weight for $\GL_2(k_v)$. It follows that $\kappa_v\in\JH(\overline{\sigma}(\lambda_v, \tau_v))$. Conversely, a tensor product of irreducible representations is irreducible as a representation of the product group. Therefore, $\kappa\in \JH(\overline{\sigma}(\lambda, \tau))$ if and only if $\kappa=\otimes_{v\in S_p}\kappa_v$ where $\kappa_v\in \JH(\overline{\sigma}(\lambda_v, \tau_v))$. We write $\sigma(\lambda,\tau)_{\kappa}=\otimes \sigma(\lambda,\tau)_{\kappa_v}$ where $\sigma(\lambda,\tau)_{\kappa_v}$ is a $\mathcal{O}$-lattice with cosocle $\kappa_v$ in $\sigma(\lambda_v,\tau_v)$.
\begin{proposition} \label{Serre weight conjecture}
   Given a Serre weight $\kappa_\mathcal{S}\in \sigma(\lambda_\mathcal{S},\tau_\mathcal{S})$. Then $M_\infty^{\sigma^\mathcal{S}}(\kappa_\mathcal{S})\neq0$ if and only if $\kappa_v\in W(\overline{r}_v)$ for all $v\in \mathcal{S}$. If that is the case, then $M_\infty(\sigma(\lambda,\tau)_\mathcal{S})$ is a cyclic $R_\infty(\lambda_\mathcal{S},\tau_\mathcal{S})$ module. 
\end{proposition}
\begin{proof}
   The first part follows from \cite[Theorem~9.1.1, Remark 9.1.2]{EGS}. The second part follows the same argument as in \cite[Lemma~5.1.2]{LatticeforGL3}. In particular, it follows from the method of \cite{multiplicityone} and our patching functor being minimal.
\end{proof}
\begin{theorem}\label{cyclicity of patched module}
    Given a minimal patching functor with unramified coefficients for $\{\overline{\rho}_v\}_{v\in \mathcal{S}}$, where $\overline{\rho}_v$ is $2\ell_v$-generic for some positive integers $\ell_v$. Assume $(\lambda_v)_{j,1}-(\lambda_v)_{j,2}\leq\ell_v$. Given any Serre weight $\kappa\in \JH(\overline{\sigma}(\lambda_\mathcal{S}, \tau_\mathcal{S}))$, $M_{\infty}(\sigma(\lambda,\tau)_\kappa)$ is a cyclic $R_{\infty}(\lambda_\mathcal{S}, \tau_\mathcal{S})$ module.
\end{theorem}
\begin{proof} 
If $W(\overline{r}_v, \lambda_v, \tau_v)\colonequals  W(\overline{r}_v)\cap \JH(\overline{\sigma}(\lambda_v, \tau_v))=\varnothing$ for some $v$, then by the exactness of the patching functor and \cref{Serre weight conjecture}, $M_{\infty}(\kappa)=0$ for any $\kappa\in \JH(\overline{\sigma}(\lambda,\tau))$ and hence $M_{\infty}(\overline{\sigma}(\lambda,\tau)_\kappa)=0$. By Nakayama's lemma $M_{\infty}(\sigma(\lambda,\tau)_\kappa)=0$. Assume $W(\overline{r}_v, \lambda_v, \tau_v)\neq\varnothing$ for all $v\in \mathcal{S}$, we will prove the statement in the following steps:
\begin{lemma}\label{cutting down} (\textit{cf.} \cite[Lemma~12.8]{breuil2012towards}) Assume that $\overline{\rho}$ is $2n$-generic and $\kappa\in \Inj_{n}\sigma$ for all $\sigma\in W(\overline{\rho})$. We define the distance between $F(\mathfrak{t}_\mu(\omega'))$ and $F(\mathfrak{t}_\mu(\omega))$ to be $\sum_{j} |\omega'_j-\omega_j|$.\par
There exists $\alpha \text{ (resp. }\beta) \in W(\overline{\rho})$ which is the furthest (respectively closest) from $\kappa$ in the extension graph. Moreover, for any $\sigma\in W(\overline{\rho})$, if $\sigma=\beta$, then $I(\beta, \kappa)$ does not contain any other Serre weights of $W(\overline{\rho})$ as subquotients. If $\sigma\neq \beta$, then $\beta$ is a subquotient of $I(\sigma, \kappa)$.\par
Moreover, assume ${(\lambda_{v})}_{j,1}-{(\lambda_{v})}_{j,2}\leq n$ for all $j,v$ and $W(\overline{\rho}, \lambda,\tau)\neq \varnothing$. There exists $\alpha \text{ (resp. }\beta) \in W(\overline{\rho}, \lambda,\tau)$ which is the furthest (respectively closest) from $\kappa$ in the extension graph. Moreover, $\JH(I(\alpha',\beta'))=W(\overline{\rho},\lambda,\tau)$ and $I(\alpha', \beta')$ is a $\Gamma$-representation.
\end{lemma}
\begin{proof}
As $\overline{\rho}$ is $2n$-generic, by \cref{prop on W(r)}, all $\sigma\in W(\overline{\rho})$ is $(2n-1)$-generic and \cref{Main theorem on generalization} applies. We assume $\kappa=F(\mathfrak{t}_\mu(\omega))$. By \cref{prop on W(r)}, if $\sigma\in W(\overline{\rho})$, then $\sigma=F(\mathfrak{t}_\mu(\xi))$ such that $\xi_j=0$ if $\gamma_{f-1-j}=0$ and $\xi_j\in \{0, \sgn(s_j)\}$ otherwise. The value of $\xi_j$ for each $j$ is independent. Therefore, we let $\alpha=F(\mathfrak{t}_\mu(\xi'))$ (respectively $\beta=F(\mathfrak{t}_\mu(\xi''))$), where $|\xi'_j-\omega_j|$ (respectively $|\xi''_j-\omega_j|$) is maximum (respectively minimum) for all $j$. As the Serre weights in $W(\overline{\rho})$ are given by the corner of a hypercube in the extension graph \cref{prop on W(r)}, there exists a unique corner with the maximal (respectively minimal) distance from $\kappa$. \par
By \cref{Cor on socle filtration}, if $\sigma'\in I(\beta,\kappa)$, then $\sigma'$ is closer to $\kappa$ than $\beta$; therefore $\sigma'\notin W(\overline{\rho})$. Continued with the notation in 1, if $\sigma\colonequals F(\mathfrak{t}_\mu(\xi))\in W(\overline{\rho})$, then for each $j$, $|\xi''_j-\omega_j|\leq |\xi_j-\omega_j|$ and $|\xi''_j-\xi_j|\leq 1$, by the choice of $\beta$ and \cref{prop on W(r)}. Therefore, we must have for each $j$, $0\leq\xi''_j-\omega_j\leq  \xi_j-\omega_j$ or $0\geq\xi''_j-\omega_j\geq  \xi_j-\omega_j$, that is $\beta-\kappa\leq \sigma-\kappa$. By \cref{Cor on socle filtration}, this implies that $\beta$ is a subquotient of $I(\sigma,\kappa)$.\par
By the proof of \cref{lattice is mk1n torsion}, $\JH(\overline{\sigma}(\lambda,\tau))\subset \JH(\Inj_n\sigma)$ for all $\sigma\in W(\overline{\rho}, \lambda,\tau)$. Again, by \cref{prop on JH}, the Serre weights in $W(\overline{\rho})$ are given by the corner of a hypercube. Similarly, by the proof of \cref{resmultfree}, $\JH(\overline{\sigma}(\lambda,\tau))$ is given by integral points of a hypercuboid. As the intersection of a hypercuboid with a hypercube is still a hypercube (of possibly smaller dimension), we can find the vertice $\alpha=F(\mathfrak{t}_\mu(\xi'))$ (respectively $\beta=F(\mathfrak{t}_\mu(\xi''))$) which is furthest (respectively closest from $\kappa$). Then for any $\sigma=F(\mathfrak{t}_\mu(\xi))\in W(\overline{\rho},\lambda,\tau)$, we have $|\xi'_j-\omega_j|\leq |\xi_j-\omega_j|\leq|\xi''_j-\omega_j|$ for all $j$. Furthermore, by \cref{prop on W(r)}, $|\xi'_j-\xi''_j|\leq 1$ for all $j$. Therefore, for each $j$, $ \xi_j=\xi'_j$ or $\xi''_j$ and hence $\sigma-\alpha'\leq \beta'-\alpha'$. By \cref{Cor on socle filtration}, we deduce that $\sigma\in I(\alpha',\beta')$. The last claim follows from the fact that $|\xi'-\xi''|\leq 1$ and \cref{Main theorem on generalization}.
\end{proof}
By Nakayama's lemma, $M_\infty (\sigma(\lambda,\tau)_\kappa)$ is cyclic if and only if $M_\infty (\overline{\sigma}(\lambda,\tau)_\kappa)$ is cyclic.
 For any $v\in \mathcal{S}$, if $\kappa_v\notin W(\overline{r}_v)$, then by \cref{Serre weight conjecture} and the exactness of the patching functor, $M_\infty(\kappa_v\otimes W^v)=0$ for any $W^v$ which is a $\prod_{S\setminus\{v\}}\GL_2(k_v)$-representation over $\F$. Let $W=\otimes_{v\in \mathcal{S}}W_v$. If $\kappa_v\in \soc(W_v)$ and $\kappa_v\notin W(\overline{r}_v) $, by the exactness of the patching functor, $M_\infty(W)=M_\infty(W/(\kappa_v\otimes_{w\in \mathcal{S}\setminus\{v\}} W_w))=M_\infty((W_v/\kappa_v)\otimes_{w\in \mathcal{S}\setminus\{v\}} W_w)$. Similarly, if $\kappa_v\subset\cosoc(W_v)$ and $\kappa_v\notin W(\overline{r}_v) $, we take $W'_v$ to be the pre-image of the quotient map $W_v\twoheadrightarrow\kappa_v$. Then $M_\infty(W)=M_\infty(W'_v\otimes_{w\in \mathcal{S}\setminus\{v\}}W_w)$. 
 Therefore, applying the argument recursively, we reduce it to the subquotient $W=\otimes_v W_v$ for which all its socle and cosocle are modular Serre weights. By \cref{prop on W(r)} and \cref{Cor on socle filtration}, we deduce that all the Jordan--H\"older factors of $W$ are modular Serre weights.\par
Since the cosocle filtration of $\overline{\sigma}(\lambda,\tau)_\kappa$ is given by \cref{Cor on socle filtration}), together with \cref{cutting down}, we can deduce that $W_v$ is the $\Gamma$-representation $I(\alpha_v, \beta_v)$ such that $ W(\overline{r}_v,\lambda_v, \tau_v)=\JH(W_v)$. Therefore,
\begin{equation}\label{reduction of Serre weight}
     M_\infty (\overline{\sigma}(\lambda,\tau)_\kappa)\cong M_\infty (\otimes_vW_v)\cong M_\infty (\otimes_vI(\alpha_v, \beta_v)).
 \end{equation}
 By \cite[Proposition~3.5.2]{EGS}, for each $v\in \mathcal{S}$, there exists a tame type $\tau'_v$, such that $\JH(I(\alpha_v, \beta_v))\subset \JH(\overline{\sigma}(\tau'_v))$. We can find a lattice $\sigma(\tau'_v)_{\beta_v}\subset \sigma(\tau'_v)$ with cosocle $\beta_v$, then $\otimes_{v\in \mathcal{S}}I(\alpha_v, \beta_v)$ is isomorphic to a quotient $\widetilde{W}$ of $\overline{\sigma}(\tau'_\mathcal{S})_\beta\colonequals \otimes_{v\in\mathcal{S}}\overline{\sigma}(\tau'_v)_{\beta_v}$. We will finish the proof by showing that $M_\infty(\overline{\sigma}(\tau'_\mathcal{S})_\beta)$ is cyclic. When $\mathcal{S}=\{v\}$, such result follows from \cite[Theorem 10.1.1]{EGS}. We will generalize such result for any $\mathcal{S}$ in the follows. \par
 By \cite[Theorem~7.2.1]{EGS}, the special fibre $\overline{R}_{\infty}^{\tau'_\mathcal{S}}$ (defined in \cref{def of Rtau}) is a power series ring over
$$\widehat{\otimes}_{v\in \mathcal{S}}\F\llbracket (X'_{j_v}, Y'_{j_v})_{j_v\in \mathcal{K}_v}\rrbracket /(X'_{j_v}Y'_{j_v})_{j_v\in \mathcal{K}_v}.$$
for some $\mathcal{K}_v\subset\{1, \dots, f_v\}$.
 Let $\mathcal{K}=\prod_{v\in\mathcal{S}}\mathcal{K}_v$. Using the notation of \cite{EGS}, for each $\prod_{v\in \mathcal{S}} J_v\subset \mathcal{K}$, we have $(\sigma_{J_v})\in W(\overline{r}_v)$. 
We generalize the proof of \cite[Theorem~10.1.1]{EGS}. (We swapped the notation of $\mathcal{W}$ and $\mathcal{J}$ appearing in \cite{EGS}.)
For $\mathcal{W}\colonequals \prod_v \mathcal{W}_v \subset \mathcal{J}$, we write $J=\prod_v J_v \in \mathcal{W}$ if $J_v\in \mathcal{W}_v$ for all $v\in \mathcal{S}$. 
 
Moreover, given a Serre weight $\sigma_{J_v}$ for each $v\in \mathcal{S}$, we define $\sigma_J\colonequals \otimes_{v} \sigma_{J_v}$. Then we write $I_\mathcal{W}$ for the radical ideal in $\overline{R}^{\tau_S}$, which cuts out the induced reduced structure on the closed subspace $\bigcup_{ J'\in \mathcal{W}}X_\infty(\sigma_J)$.
 The notion of interval (\textit{cf.} \cite[Definition~10.1.4]{EGS}) and capped interval is still well-defined. We define $\mathcal{F}(J_1,J_2)$ and $\mathcal{F}(J_1,J_2)^\times$ analogously to \cite[Definition~10.1.5]{EGS}, and generalize \cite[Lemmas~10.1.6, 10.1.8]{EGS} as follows.
\begin{lemma}\label{quotient of intervals}
    The quotient $I_{\mathcal{F}(J_1, J_2)^\times}/I_{\mathcal{F}(J_1, J_2)}$ is isomorphic to $R_\infty^{\tau_\mathcal{S}}/I_{\{J_1\}}$, in particular it is cyclic and is generated by $\prod_{v\in \mathcal{S}}\prod_{j_v\in {J_2}_v\setminus {J_1}_v}X'_{j_v}$. 
\end{lemma}
\begin{lemma}\label{ideal of two intervals}
    If $\mathcal{W}_1, \mathcal{W}_2$ are two capped intervals in $\mathcal{J}$ that share a common cap, then $I_{\mathcal{W}_1}+I_{\mathcal{W}_2}=I_{\mathcal{W}_1\cap\mathcal{W}_2}$.
\end{lemma}
By the argument in the proof of \cite[Theorem 10.1.1]{EGS}, for each interval $\mathcal{W}\subset \mathcal{J}$, there is
a subquotient $\overline{\sigma}(\tau_\mathcal{S})^\mathcal{W}$ of $\overline{\sigma}(\tau'_\mathcal{S})_\beta$, uniquely characterized by the property that $\JH(\overline{\sigma}(\tau'_\mathcal{S})^\mathcal{W})=\{\sigma_{J'}\}_{J'\in \mathcal{W}}$. We finish the proof by proving the following proposition and take $\mathcal{W}=\mathcal{J}.$
\begin{proposition}
    For any capped interval $\mathcal{W}\subset \mathcal{J}$, $M_\infty(\overline{\sigma}(\tau_\mathcal{S})^\mathcal{W})$ is cyclic.
\end{proposition}
\begin{proof}
We will prove this by inducting on $|\mathcal{W}|$.
If $|\mathcal{W}|=1$, the ring $R_\infty(\tau'_\mathcal{S})$ is regular. As $M_\infty$ is a minimal patching functor, by the method of \cite{multiplicityone}, $M_\infty(\sigma^\circ(\tau'_\mathcal{S}))$ is of rank one over $R_\infty(\tau'_\mathcal{S})$ for any lattice $\sigma^\circ(\tau'_\mathcal{S})\subset\sigma(\tau'_\mathcal{S})$. The argument relies on studying $R_\infty$ and is independent of patching using unitary group or quaternion algebra. For the induction step, it follows exactly as in \cite[Lemma~10.1.12, 10.1.13]{EGS} and with Lemmas 10.1.6, 10.1.8 replaced by \cref{quotient of intervals}, \cref{ideal of two intervals}.
\end{proof} 
\end{proof}
\section{Properties of \texorpdfstring{$\pi[\mathfrak{m}_{K_1}^n]$}{TEXT} %
}\label{ch7: Properties of pi(rho)}
\begin{proposition}\label{existence of D_n}
    Given a finite set $\mathcal{D}$ of distinct $(2n-1)$-generic Serre weights. There exists a unique, up to isomorphism, representation $D_0^n(\mathcal{D})$, which is $\m_{K_1}^n$-torsion such that \hfill\begin{enumerate}[label=(\roman*)]
        \item $\soc(D_0^n(\mathcal{D}))=\bigoplus_{\sigma\in \mathcal{D}} \sigma$
        \item $[D_0^n(\mathcal{D})\colon \sigma]=1$ for all $\sigma\in \mathcal{D}$
        \item $D_0^n(\mathcal{D})$ is maximal with respect to properties (i) and (ii)
    \end{enumerate}
    Moreover, there is an isomorphism $D_0^n(\mathcal{D})=\bigoplus_{\sigma\in D}D^n_{0,\sigma}(\mathcal{D})$ where $D^n_{0,\sigma}(\mathcal{D})$ is the unique maximal subrepresentation of $\Inj_n \sigma$ such that $[D^n_{\sigma}(\mathcal{D})\colon \sigma]=1$ and $[D^n_{\sigma}(\mathcal{D})\colon \sigma']=0$ for any $\sigma'\in \mathcal{D}$ with $\sigma'\neq \sigma$.
\end{proposition}
\begin{proof}
The first three statements and the isomorphism follow from the same proof in \cite[Proposition~13.1]{breuil2012towards}, replacing $\Hom_\Gamma$ with $\Hom_{K/Z_1}$ by \cite[Lemma~2.4.6]{BHHMS} and replacing representations of $\Gamma$ with representations of $K/Z_1$ which are $\mathfrak{m}_{K_1}^n$-torsion etc. The last statement follows from the same proof in \cite[Corollary~4.2]{hu2022mod}.
\end{proof}
Assume that $\overline{\rho}$ is $2n$-generic for some $n\geq 0$. Then by \cref{prop on W(r)}, if $\sigma\in W(\overline{\rho})$, $\sigma$ is $(2n-1)$-generic. In this case, we define $D^n_0(\overline{\rho})\colonequals D^n_0(W(\overline{\rho}))$ and similarly $D^n_{0,\sigma}(\overline{\rho})\colonequals D^n_{0,\sigma}(W(\overline{\rho}))$.\par
By \cref{Cor on socle filtration}, we have $\dim_{\overline{\F}}(\Hom(I(\sigma, \tau), I(\sigma, \tau')))=1$ if $\tau-\sigma\leq\tau'-\sigma$ and 0 otherwise. If $\Hom(I(\sigma, \tau), I(\sigma, \tau'))\neq0$, we fix $\iota_{\tau}\colon \sigma\hookrightarrow I(\sigma, \tau)$, and let $\phi_{\tau, \tau'}\colon I(\sigma, \tau)\hookrightarrow I(\sigma, \tau')$ be the unique embedding such that $\iota_{\tau'}=\phi_{\tau, \tau'}\circ\iota_{\tau}$.
\begin{lemma}\label{lemma on rep with modular serre wight soc}
\hfill\begin{enumerate}
    \item We have $$D^n_{0,\sigma}(\overline{\rho})=\underset{\leq}{\varinjlim}I(\sigma, \tau),$$ where the inductive limit is taken over $\phi_{\tau,\tau'}$ and such that $I(\sigma, \tau)$ does not contain any other $\sigma'\in W(\overline{\rho})$ if $\sigma'\neq \sigma$.
    \item $D^{n}_0(\overline{\rho})=\bigoplus_{\sigma\in W(\overline{\rho})}D^n_{0, \sigma}(\overline{\rho})$ is multiplicity free.
    \item For any $\sigma\in W(\overline{\rho})$, we have $D_{0,\sigma}(\overline{\rho})\subset D^n_{0,\sigma}(\overline{\rho})$ and $D^n_{0,\sigma}(\overline{\rho})^{K_1}=D_{0,\sigma}(\overline{\rho})$.
    \end{enumerate}
\end{lemma} 
\begin{proof}
  The proof of 1,2 follows verbatim from \cite[Proposition~13.4, Corollary 13.5]{breuil2012towards} with \cref{cutting down} in place of \cite[Lemma~12.8]{breuil2012towards}. The proof of 3 follows from the construction (cf. \cite[Theorem~4.6]{hu2022mod}).
\end{proof}
Let $F$ be a totally real number field in which $p$ is unramified. Fix $v$ a place dividing $p$. Let $D$ be a quaternion algebra with centre $F$, which splits at exactly one infinite place. Fix a compact open subgroup $U^v$ of $D\otimes_F\mathbb{A}_{F,f}^{v}$. Given a compact open subgroup $U$ of $(D\otimes_F\mathbb{A}_{F,f})^\times$, we let $X_U$ be the associated smooth projective Shimura curve over $F$. Letting $U_v$ run over compact open subgroups of $(D\otimes_F F_v)^\times\cong \GL_2(F_v)$, we consider
\[\pi(\overline{\rho})\colonequals \varinjlim_{U_v}\Hom_{G_F}(\overline{r},H^1_{\textrm{\small \'et}}(X_{U^vU_v}\times_F \overline{F}, \F)),\]
which is an admissible smooth representation of $\GL_2(F_v)$ over $\F$. It is expected that $\pi$ corresponds to $\overline{\rho}\colonequals \overline{r}|_{G_{F_{v}}}$ under the conjectural mod $p$ Langlands Program.
\begin{corollary}\label{nth torsion of pi is mult free}
   Assume that $\overline{\rho}$ is $\max\{2n,12\}$-generic, then
\[\pi(\overline{\rho})[\mathfrak{m}_{K_1/Z_1}^n]\cong D_0^n(\overline{\rho}),\] In particular, it follows from \cref{lemma on rep with modular serre wight soc} that $\pi(\overline{\rho})[\mathfrak{m}_{K_1/Z_1}^n]$ is multiplicity free.
\end{corollary}
\begin{remark}
    The case where $n=1$ is proven by \cite{LMS}, \cite{HuWang2} and \cite{multone}; while the case where $n=2$ is proven in \cite[Theorem~1.9]{BHHMS}, \cite[Theorem~6.3]{Yitong} (here $r=1$, as we are considering the case with minimal level), and \cite[Corollary~8.13]{hu2022mod}.
\end{remark}
\begin{proof}
To show that $D^n(\overline{\rho})\subseteq \pi(\overline{\rho})[\mathfrak{m}_{K_1/Z_1}^n]$, we modify the proof of \cite[Theorem~8.4.2]{BHHMS} by replacing $\widetilde{D}_{\sigma_v}$ by $D^n_{0,\sigma}$ and $\pi(\overline{\rho})[\mathfrak{m}_{K_1}^2]$ by $\pi(\overline{\rho})[\mathfrak{m}_{K_1}^n]$. We will sketch the proof as follows. In \cite{CEG}, $\mathbb{M}_{\infty}$ is constructed so that $\pi(\overline{\rho})^{\vee}=\mathbb{M}_{\infty}/\mathfrak{m}_{\infty}$. Moreover, we have 
\[\mathbb{M}_{\infty}/(p,x_1, \ldots, x_{4|S|+q})\cong \bigoplus_{\sigma\in W(\overline{r}_v^\vee)} (\Proj_{K/Z_1}\sigma^\vee)^{m_\sigma}\]
for some $m_\sigma\geq 1$ and $q$ is an integer greater than or equal to $[F\colon Q]$.
Therefore, we can deduce that
\begin{equation*}
\begin{split}
 \Hom_{K/Z_1}(D^n_{0,\sigma},\pi(\overline{\rho}))=\Hom_{K/Z_1}(D^n_{0,\sigma},\pi(\overline{\rho})[\mathfrak{m}_{K_1/Z_1}^n])\\
 \xrightarrow[]{\sim}\Hom_{K/Z_1}(\sigma,\pi(\overline{\rho}))=\Hom_{K/Z_1}(\sigma,\soc(\pi(\overline{\rho}))).\end{split}
 \end{equation*}
Since $\soc \pi(\overline{\rho})=\oplus_{\sigma\in W(\overline{r}_v^{\vee})}\sigma$, we have indeed $D_0^n(\overline{\rho})\subseteq \pi(\overline{\rho})[\mathfrak{m}_{K_1/Z_1}^n]$ (\textit{cf.} \cite[Lemma~9.2]{Breuiloriginal}).
If we can show that all $\sigma\in W(\overline{r}_v^{\vee})$ appear only once in $\pi(\overline{\rho})[\mathfrak{m}_{K_1/Z_1}^n]$, then by the maximal property of $D^n(\overline{\rho})$, we have the other inclusion.\par
Since the cases where $n=1,2$ are already proven, we assume $n>2$. By our genericity assumption, all $\sigma\in W(\overline{\rho})$ is $(2n-1)$ generic. Assume for the sake of contradiction that there exists a Serre weight $\sigma\in W(\overline{\rho})$, such that $[\pi(\overline{\rho})[\m_{K_1}^n]\colon \sigma]>1$. Since $\pi(\overline{\rho})[\m_{K_1}^2]$ is multiplicity free, so is $\soc(\pi(\overline{\rho}))$. Hence, the map 
    \[f\colon \pi(\overline{\rho})[\m_{K_1}^n]\hookrightarrow\bigoplus_{\sigma\in W(\overline{\rho})}\Inj_n\sigma\]
    is injective as it is injective on the socle. Therefore, $\pi(\overline{\rho})[\m_{K_1}^n]\cong \oplus_{\sigma\in W(\overline{\rho})} V_\sigma$, where $V_\sigma$ is the image of $p_n\circ f$ and $p_n$ is the projection map onto $\Inj_n\sigma$. It suffices to show that for $\tau\in W(\overline{\rho})$, $[V_\sigma:\tau]=1$ if $\tau=\sigma$ and $[V_\sigma:\tau]=0$ if $\tau\neq \sigma$. Since taking $\mathfrak{m}_{K_1}^2$-torsion is compatible with $f$, by \cref{lemma on rep with modular serre wight soc}, we deduce that $V_\sigma[\mathfrak{m}_{K_1}^2]\cong D_{0,\sigma}^2(\overline{\rho}) $. We finish the proof by the following lemma.
\begin{lemma}\label{lemma on multiplicity}
    Let $\sigma, \tau\in W(\overline{\rho})$ be $(2n-1)$-generic and $V\subset \Inj_n \sigma$ be a subrepresentation. Assume $V[\mathfrak{m}_{K_1}^2]$ is multiplicity-free, $[V[\mathfrak{m}_{K_1}^2]:\tau]=1$ if $\tau=\sigma$, and $0$ if $\tau\neq \sigma$. Then, $[V:\tau]=1$ if $\tau=\sigma$ and $0$ if $\tau\neq \sigma$. 
\end{lemma}
\begin{proof}
Assume for contradiction that there exists a $(2n-1)$-generic $\tau\in W(\overline{\rho})$ such that $\tau\in \JH(V/V[\mathfrak{m}_{K_1}^2])$. If $\tau\neq\sigma$, by \cref{prop on W(r)} and \cref{Main theorem on generalization}, $ I(\sigma,\tau)$, a $\Gamma$-representation, is a subrepresentation of $V[\mathfrak{m}_{K_1}^1]$ and hence $I(\sigma,\tau)\subset V[\mathfrak{m}_{K_1}^1]$, contradicting our assumption.
Now assume $\tau=\sigma$. Considering the image of $\Proj_n \sigma\to V$, without loss of generality, we can replace $V$ with a subrepresentation with cosocle $\sigma$. Let $\widetilde{V}\colonequals V/\sigma$. By \cite[Lemma~2.4.6]{BHHMS}, and the fact that $V[\m_{K_1}^2]$ is multiplicity free, we deduce that $\soc(\widetilde{V})\subset \bigoplus_{\sigma'\in \mathcal{E}(\sigma)} \sigma'$, where $\mathcal{E}(\sigma)$ are the sets of Serre weights adjacent to $\sigma$. As $\cosoc(\widetilde{V})=\sigma$, the image of $\Proj_n \sigma\twoheadrightarrow\widetilde{V}$ lies in $\bigoplus_{\sigma'\in \JH(\soc(\widetilde{V}))}I(\sigma', \sigma)$ which is killed by $\m_{K_1}$ by \cref{Main theorem on generalization}. Therefore, $V$ is $\m_{K_1}^2$-torsion but not multiplicity-free, a contradiction. 
\end{proof}
\end{proof}
Using the patching functor $M_\infty$ constructed in \cite[\S~6]{Yitong}, which is based on \cite[\S~8]{BHHMS}, we have the following result.
\begin{corollary}
   Assume that $\overline{\rho}$ is $\max\{2n,12\}$-generic and $\sigma\in \JH(\pi(\overline{\rho})[\m_{K_1}^n])$, $M_\infty (\Proj_n\sigma)$ is multiplicity free.
\end{corollary}
\begin{proof}
    By \cref{nth torsion of pi is mult free}, we show that for any $\sigma\in \JH(\pi(\overline{\rho})[\m_{K_1}^n])$, we have $$\dim_\F(\Hom_{K_1/Z_1}(\Proj_n \sigma, \pi(\overline{\rho})))= 1.$$ Then, by the proof of \cite[Theorem~8.4.2]{BHHMS} we deduce that
    $$M_\infty(\Proj_n \sigma)/\m_\infty \cong \Hom_{K_1/Z_1}(\Proj_n \sigma, \pi(\overline{\rho}))^{\vee}$$ is cyclic.
\end{proof}
We can extend the result of \cite[Proposition~5.1]{BHHMS5} as follows:
\begin{corollary}
    Assume that $\overline{\rho}$ is split reducible and $\max\{9, 2f+1, 2n+1\}$-generic.
\hfill\begin{enumerate}
    \item Let $\pi'$ be a subquotient of $\pi(\overline{\rho})$. Then there is a unique subset $\Sigma'\subset \{1, \dots, f\}$ such that
    \[\pi'[\mathfrak{m}_{K_1}^n]\cong \bigoplus_{i \in \Sigma'}\bigoplus_{\sigma\in W(\overline{\rho}), |J_{\sigma}|=i}D_{0,\sigma}^n(\overline{\rho}).\]
    
    \item Let $\pi_1\subset\pi_2$ be a subrepresentations of $\pi(\overline{\rho})$. Then the induced sequence of $\F\llbracket K/Z_1\rrbracket/\mathfrak{m}_{K_1}^n$-modules
    \[0\to \pi_1[\mathfrak{m}_{K_1}^n]\to \pi_2[\mathfrak{m}_{K_1}^n]\to \pi_1/\pi_2[\mathfrak{m}_{K_1}^n]\to 0\]
    is split exact. 
\end{enumerate}
\end{corollary}
\begin{proof}
The proof of 1 follows the same argument as in \cite[Proposition~5.1]{BHHMS5}, with appropriate generalization, such as replacing $\mathfrak{m}_{K_1}^2$ with $\mathfrak{m}_{K_1}^n$, $\widetilde{D}_{0,\sigma}(\overline{\rho})$ with $D^n_{0, \sigma}(\overline{\rho})$, \cite[Proposition~3.2.8]{BHHMS5} with \cref{nth torsion of pi is mult free} and \cite[Theorem~4.6]{hu2022mod} with \cref{lemma on rep with modular serre wight soc}.\par
2. As in the proof of \cite[Corollary~3.2.5]{BHHMS4}, it suffices to prove the special case $\pi_2=\pi(\overline{\rho})$. By the argument in \cite{BHHMS5}, if 2 does not hold, we have a non-split extension of $\F\llbracket K/Z_1\rrbracket/\mathfrak{m}_{K_1}^n$-modules
\[0\to \bigoplus_{i \in \Sigma'}\bigoplus_{\sigma\in W(\overline{\rho}), |J_{\sigma}|=i}D_{0,\sigma}^n(\overline{\rho})\to V\to \tau\to 0\]
where $\tau\in W(\overline{\rho})$. Hence, we have a non-split extension of $\F\llbracket K/Z_1\rrbracket/\mathfrak{m}_{K_1}^n$-modules between $\tau$ and $D_{0, \sigma}^n(\overline{\rho})$ for some $\sigma\in W(\overline{\rho})$. 
As this corollary is proved for $n=2$, it implies that $\tau\notin\JH(V[\mathfrak{m}_{K_1}^2])$, hence $V[\mathfrak{m}_{K_1}^2]\cong D_{0, \sigma}^2(\overline{\rho})$.
We can then conclude the proof using \cref{lemma on multiplicity} .
\end{proof}
\bibliography{thesis}{}
\bibliographystyle{amsalpha.bst}
\end{document}